\documentclass[12pt,a4paper]{amsart}
\usepackage{amsmath,amsthm,amssymb,stmaryrd,latexsym,a4wide,epic,eepic,bbm,mathrsfs,amsfonts,cases,yfonts,indentfirst,bbold}
\usepackage{dsfont,color,mathdots}

\usepackage{amsmath,amssymb}

\usepackage{amsfonts}
\usepackage{tipa}

\usepackage{CJK, CJKnumb}
\usepackage{color}      
\usepackage{indentfirst}        
\usepackage{latexsym, bm}        

\usepackage{graphicx}
\usepackage{cases}
\usepackage{fancyhdr}
\usepackage{pifont}
\newtheorem{theorem}{Theorem}
\newtheorem{defi}[theorem]{Definition}
\newtheorem{lemma}[theorem]{Lemma}
\newtheorem{coro}[theorem]{Corollary}
\newtheorem{proposition}[theorem]{Proposition}

\newtheorem{remark}[theorem]{Remark}
\usepackage[active]{srcltx}
\usepackage{enumerate}
\numberwithin{equation}{section}

\begin{document}

\title[Quantum Grassmann superalgebra as $\mathcal U_q(\mathfrak{gl}(m|n))$-module algebra]{Quantum (dual) Grassmann superalgebra as $\mathcal U_q(\mathfrak{gl}(m|n))$-module algebra and beyond}

\author[G. Feng]{Ge Feng}
\address{School of Mathematical Sciences, Shanghai Key Laboratory of PMMP, East China Normal University, Shanghai 200241, China}
\email{1187816868@qq.com}

\author[N.H. Hu]{Naihong Hu}
\address{School of Mathematical Sciences, Shanghai Key Laboratory of PMMP, East China Normal University, Shanghai 200241, China}
\email{nhhu@math.ecnu.edu.cn}

\author[M.R. Zhang]{Meirong Zhang}
\address{School of Mathematical Sciences, Shanghai Key Laboratory of PMMP, East China Normal University, Shanghai 200241, China}
\email{mrzhang@math.ecnu.edu.cn}

\author[X.T. Zhang]{Xiaoting Zhang}
\address{Department of Mathematics, Uppsala University, Box. 480, SE-75106, Uppsala, Sweden}
\email{xiaoting.zhang@math.uu.se}

\thanks{The paper is supported by the NNSFC (Grant No.
 11771142) and in part by Science and Technology Commission of Shanghai Municipality (No. 18dz2271000).}

\begin{abstract}
We introduce and define the quantum affine $(m|n)$-superspace (or say quantum Manin superspace)
$A_q^{m|n}$ and its dual object, the quantum Grassmann superalgebra $\Omega_q(m|n)$. Correspondingly,
a quantum Weyl algebra $\mathcal W_q(2(m|n))$ of $(m|n)$-type is introduced as the quantum differential operators (QDO for short) algebra $\textrm{Diff}_q(\Omega_q)$ defined over $\Omega_q(m|n)$, which is a smash product of the quantum differential Hopf algebra $\mathfrak D_q(m|n)$ (isomorphic to the bosonization of the quantum Manin superspace) and the quantum Grassmann superalgebra $\Omega_q(m|n)$. An interested point of this approach here is that even though $\mathcal W_q(2(m|n))$ itself is in general no longer a Hopf algebra, so are some interesting sub-quotients existed inside. This point of view gives us one of main expected results, that is, the quantum (restricted) Grassmann superalgebra $\Omega_q$ is made into
the $\mathcal U_q(\mathfrak g)$-module (super)algebra structure,
$\Omega_q=\Omega_q(m|n)$ for $q$ generic, or $\Omega_q(m|n, \bold 1)$ for $q$ root of unity, and  $\mathfrak g=\mathfrak{gl}(m|n)$ or $\mathfrak {sl}(m|n)$, the general or special linear Lie superalgebra. This QDO approach
provides us with explicit realization models for some simple $\mathcal U_q(\mathfrak g)$-modules, together with
the concrete information on their dimensions. Similar results hold for the
quantum dual Grassmann superalgebra $\Omega_q^!$ as $\mathcal U_q(\mathfrak g)$-module algebra.

This paper is a sequel to \cite{Hu}, some examples of pointed Hopf algebras can arise from the QDOs, whose idea is an expansion of the spirit noted by Manin in \cite{Ma}, \& \cite{Ma1}. For instance, as byproducts, if considering the  bosonizations of $A_q^{m|n}$, $\Omega_q(m|n)$, $\mathcal D_q(m|n)$ and
their finite dimensional restricted objects when $\textbf{char}(q)=\ell>2$, we can obtain some examples of pointed Hopf algebras, e.g., the known multi-rank (generalized) Taft Hopf algebras.
\end{abstract}

\maketitle

\tableofcontents

\section{Introduction}

\noindent{\bf 1.1.}
For the Drinfeld-Jimbo type quantum groups $U_q(\mathfrak g)$ of semisimple Lie algebras $\mathfrak g$, there was a Majid
``quantum tree" question (\cite{Ma1}), that was to claim, any $U_q(\mathfrak g)$ can be constructed from $U_q(\mathfrak{sl}_2)$ via a series of suitable double-bosonization procedures. Here the double bosonization construction (roughly speaking, which consists of two bosonizations or say Radford biproducts well-arranged in a certain manner) was obtained by Sommerh\"auser (\cite{Som}) in a Yetter-Drinfeld category $^H_H{\mathcal YD}$ (a braided monoidal category) for $H=\Bbb k[G]$, and $G$ a finitely generated abelian group (this was originally motivated by Lusztig's construction of $U_q(\mathfrak g)$ via his defining algebra $\mathfrak f'$ in \cite{Lus2}), and a general version for $H$ to be quasi-triangular was then obtained by Majid in another braided category $_H\mathcal M$ (his formulation for the theory directly depends on the explicit information provided by the $R$-matrices involved rather than the braidings). As an important application of this theory named the double-bosonization, combined an observation from the $U_q(\mathfrak {gl}(n))$-module algebra structure (via realizing  $U_q(\mathfrak {gl}(n))$ as a certain quantum differential operators defined over the quantum divided power algebra $\mathcal A_q(n)$ introduced in \cite{Hu}) with  the motivation from \cite{Ma1} and Faddeev-Reshetikhin-Takhtajan (\cite{FRT}), the Majid expectation just mentioned has been solved in a series of joint papers of H.M. Hu and the second author (\cite{HH}, \cite{HH1}, \cite{HH2}, \cite{HH3}), among which more efforts when our treating with the more subtle and more complicated exceptional cases need to be made via establishing the so-called generalized double-bosonization construction. {\it A further interesting problem is to ask what the Majid ``quantum tree" of $U_q(\mathfrak g)$'s looks like for the (simple) Lie superalgebras $\mathfrak g$ (see the Kac-classification list \cite{Kac1})}. Actually, one of the main results of this paper, namely, the quantum Grassmann superalgebra or dual Grassmann superalgebra we will define here is made into a $U_q(\mathfrak{gl}(m|n))$-module superalgebra respectively, can be regarded as the first starting step towards this goal.

\smallskip
\noindent{\bf 1.2.}
 For the general linear Lie superalgebra $\mathfrak g=\mathfrak{gl}(m|n)$, or special linear Lie superalgebra $\mathfrak{sl}(m|n)$, in order to construct some $U_q(\mathfrak g(m|n))$-module superalgebras of suitable size such that they are fitted into the framework working well for the rank-induction construction from $U_q(\mathfrak{g}(m|n))$ to $U_q(\mathfrak {g}(m{+}1|n))$ or $U_q(\mathfrak{g}(m|n{+}1))$ ($\mathfrak g=\mathfrak {gl}$, or $\mathfrak {sl}$), let us start with the natural $U_q(\mathfrak{g}(m|n))$-module $\bold V$ of dimension $m+n$, for $\mathfrak g=\mathfrak{gl}$ or $\mathfrak{sl}$, and introduce a new notion of the so-called quantum affine $(m|n)$-superspace or say the quantum affine Manin superspace $A_q^{m|n}$ in subsection 3.2, which is a generalization of the quantum affine Manin spaces $A_q^{m|0}$ and $A_{q^{-1}}^{0|n}$ (see \cite{Ma1}). In order to simultaneously treat with the generic case or $\textbf{char}(q)=\ell>2$ (the root of unity case), we define its dual version, the quantum Grassmann superalgebra $\Omega_q(m|n)$ and the quantum restricted Grassmann superalgebra $\Omega_q(m|n,\bold 1)$ when $\textbf{char}(q)=\ell>2$ in subsection 3.3. Following the same spirit of \cite{Hu}, we shall generalize those quantum differential operators (QDOs for short) defined over the quantum divided power algebra $\mathcal A_q$ to the quantum Grassmann superalgebra $\Omega_q$, see subsection 3.4. Actually, these QDOs generate a superalgebra $\mathcal D_q(m|n)$, which is isomorphic to the quantum affine Manin superspace just mentioned (Proposition 7).
Furthermore, the bosonization of $\mathfrak D_q(m|n)$ (resp. $\mathfrak D_q(m|n, \bold 1)$) is a (resp. finite-dimensional) pointed Hopf algebra, see Theorem 11 (resp. Corollary 12 under the assumption $\textbf{char}(q)=\ell>2$). Lemma 9 indicates that the quantum (resp. restricted) Grassmann superalgebra $\Omega_q(m|n)$ (resp. $\mathfrak D_q(m|n, \bold 1)$) is a $\mathfrak D_q(m|n)$-module algebra (resp.  $\mathfrak D_q(m|n, \bold 1)$-module algebra). This allows us to define their smash product
as the quantum differential operator algebra $\textrm{Diff}_q(\Omega_q)$ over the quantum (restricted) Grassmann superalgebra
$\Omega_q$, named the quantum (restricted) Weyl algebra $\mathcal W_q(2(m|n))$ (resp. $\mathcal W_q(2(m|n), \bold 1)$) of $(m|n)$-type in subsection 3.10. Here, an interested point of this approach is that even though $\mathcal W_q(2(m|n))$ itself is in general no longer a Hopf algebra (see Remark 3.5 in \cite{Hu}), so are some interesting sub-quotients existed inside. This allows us in Section 4 to realize the Hopf algebra $\mathcal U_q(\mathfrak g)$ by defining some compatible QDOs in $\mathcal W_q(2(m|n))$, such that the quantum (restricted) Grassmann superalgebra $\Omega_q$ is made into the $\mathcal U_q(\mathfrak g)$-module (super)algebra structure, $\Omega_q=\Omega_q(m|n)$ for $q$ generic, or $\Omega_q(m|n, \bold 1)$ for $q$ root of unity, and $\mathfrak g=\mathfrak{gl}(m|n)$ or $\mathfrak {sl}(m|n)$, the general or special linear Lie superalgebra. In this way, we get a realization model for some simple $\mathcal U_q(\mathfrak g)$-modules since we will prove the components of $\Omega_q$ to be simple as $\mathcal U_q(\mathfrak g)$-modules and also give their dimension-formulae. Similar results for the quantum dual Grassmann (resp. restricted) superalgebra $\Omega_q^!(m|n)$ (resp. $\Omega_q^!(m|n,\bold 1)$) as $\mathcal U_q(\mathfrak g)$-module algebra are obtained in Section 5, and their simple submodules are also described.

\smallskip
\noindent{\bf 1.3.}
Let us digress briefly to recall some related research background. Various discussions on quantum differential operators
intensively appeared in the early 1990s since the seminal work of Wess-Zumino \cite{WZ} (1990) and Woronowicz \cite{Wo} (1989)
published, but seemed lack of a unified definition (for QDOs), which is unlike the case by Woronowicz's treating with the dual concept (i.e., quantum differential forms, QDF for short): his axiomatic internal approach to the first order differential calculi (FODC, for short) satisfying the usual Leibniz rule, not only yields the equivalent terminology ``bicovariant bimodule" as the known ``Hopf bimodule" in Hopf algebra theory, but also leads to the appearance of Woronowicz's braiding (Proposition 3.1 \cite{Wo}, also see Theorem 6.3 \cite{Sch}). In fact, the defining condition of Yetter-Drinfeld module appeared implicitly in Woronowicz's work a bit earlier than Yetter \cite{Y} (see formula (2.39) in \cite{Wo}), as was witnessed by Schauenburg in Corollaries 6.4 \& 6.5 of \cite{Sch} proving that the category of Woronowicz's bicovariant bimodules is categorically equivalent to the category of Yetter-Drinfeld modules, while the latter has currently served as an important working framework for classifying the finite-dimensional pointed Hopf algebras. In a word, Woronowicz' framework on the FODC (essentially, the QDF)
initiated the further development so rich and deep. Roughly speaking, QDO, as a dual notion of QDF, also deserves more attention, as was early noted by Manin in \cite{M1} (see 2.2. Basic problem, pp. 1010) (his motivation mainly from \cite{WZ} \& \cite{Wo}), but not yet sufficient. It should be noticed that the design of our QDO in subsections 2.6, 3.4 (also see \cite{Hu} \& \cite{ZHn}) leads to our quantum differential (form) $d$ satisfying the twisted Leibniz rule (see \cite{GH}, pp. 5), which is different from both \cite{WZ} and \cite{Wo}.
This seemly corresponds to the broader ``Hom-"picture phenomenon beyond standard quantization, a current hot research area.

\medskip
\section{Preliminaries}
\subsection{Notation and setup}

Throughout this paper, we work on an algebraically closed field $\Bbbk$ of characteristic zero.
We denote by $\mathbbm{Z}_+,\mathbbm{N}$ the set of nonnegative integers, positive integers, respectively.

For any $m,n\in\mathbbm{Z}_+$, we denote by $I$ the set $\{1,2,\cdots,m+n\}$ with the convention that $I=\varnothing$ if $m=n=0$. Set $I_0:=\{1,2,\cdots,m\}$, $I_1:=\{m+1,m+2,\cdots,m+n\}$. Thus we have $I=I_0\cup I_1$. Denote by $J$ the set $\{1,2,\cdots,m+n-1\}$ with the convention that $J=\varnothing$ if $m+n<2$. For any $i,j\in I$, denote by $E_{ij}$ the elementary matrix of size $(m+n)\times(m+n)$ with $1$ in the $(i,j)$ position and zero in others.
For any $\mathbbm{Z}_2$-graded vector space $V:=V_{\overline{0}}\oplus V_{\overline{1}}$, denote by $\overline{v}$ the parity of the homogeneous element $v$ in $V$.

\subsection{The general linear Lie superalgebras}

Let $\mathfrak{g}$ denote by the general linear Lie superalgebra $\mathfrak{gl}(m|n)$ which has a standard basis $E_{ij}, i,j\in I$, see \cite{Kac1}. Then we have $\mathfrak{g}=\mathfrak{g}_{\overline{0}}\oplus\mathfrak{g}_{\overline{1}}$, where
\begin{align*}
\mathfrak{g}_{\overline{0}}:=&\mathrm{span}_\Bbbk\{E_{ij}\mid i,j\in I_0\text{ or } i,j\in I_1\},\\
\mathfrak{g}_{\overline{1}}:=&\mathrm{span}_\Bbbk\{E_{ij}\mid i\in I_0,j\in I_1 \text{ or } i\in I_1, j\in I_0\}.
\end{align*}
If $m=0$ (resp. $n=0$), then $\mathfrak{g}$ is exactly the general linear Lie algebra $\mathfrak{gl}(n)$ (resp. $\mathfrak{gl}(m)$).

The standard Cartan subalgebra $\mathfrak{h}$ of $\mathfrak{g}$ consists of all diagonal matrices in $\mathfrak{g}$, that is, the $\Bbbk$-span of $E_{ii},i\in I$. For each $i\in I$, we
denote by $\epsilon_i$ the dual of $E_{ii}$, which forms a basis of $\mathfrak{h}^\ast$.
The root system of $\mathfrak{g}$ with respect to $\mathfrak{h}$ is $\Delta:=\Delta_{\overline{0}}\cup\Delta_{\overline{1}}$, where
\begin{displaymath}
\Delta_{\overline{0}}:= \{\epsilon_i-\epsilon_j\mid\ i,j\in I_{\overline{0}} \ \textrm{or} \ i,j\in I_{\overline{1}}\}\quad\text{and}\quad
\Delta_{\overline{1}}:= \{\pm(\epsilon_i-\epsilon_j)\mid\  i\in I_{\overline{0}}, j\in I_{\overline{1}}\},
\end{displaymath}
 and its standard fundamental (simple root) system is given by the set
 \begin{displaymath}
 \Pi:=\{\alpha_i:=\epsilon_i-\epsilon_{i+1}\mid 1\leq i\leq m+n-1\}.
 \end{displaymath}
 Let $\Lambda:=\mathbbm{Z}\epsilon_1+\mathbbm{Z}\epsilon_2+\cdots+\mathbbm{Z}\epsilon_{m+n}$. For each $i\in I$
 denote by $\omega_i$ the fundamental weight $\epsilon_1+\epsilon_2+\cdots+\epsilon_{i}$.
From \cite{Mu}, we see that there is a symmetric bilinear form on $\Lambda$:
\begin{displaymath}
(\epsilon_i,\epsilon_j)=\left\{\begin{array}{ll} \delta_{ij} \quad & \mathrm{if} \ 1\leq i\leq m,\\-\delta_{ij} \quad & \mathrm{if} \ m+1\leq i \leq m+n,
\end{array}\right.
\end{displaymath}
which is induced from the supertrace $\mathfrak{str}$ on $\mathfrak{g}$ defined as:
$\mathfrak{str}(g):=\mathrm{tr}(a)-\mathrm{tr}(d)$ if
\begin{displaymath}
g=\left(\begin{array}{cc}
a & b \\
c & d
\end{array}\right)
\end{displaymath}
with $a, b, c$ and $d$ being $m\times m,m\times n, n\times m$ and $n\times n$ matrices, respectively.

\subsection{The quantum general linear superalgebras $U_q(\mathfrak{gl}(m|n))$}

Set $\Bbbk^\times:=\Bbbk\setminus\{0\}$. We assume that $q$ $(\ne1)\in \Bbbk^\times$. The quantum general linear
superalgebra $U_q(\mathfrak{gl}(m|n))$ (for instance, see \cite{BKK}, \cite{DM}, \cite{Zh}, \cite{Y}, etc.)
is defined as the $\Bbbk$-superalgebra with generators $E_j,\; F_j\ (j\in J),
K_i,\; K_i^{-1}\ (i\in I)$ and relations:
\begin{eqnarray*}
&(R1)&\quad K_iK_j=K_jK_i, \quad K_iK_i^{-1}=1=K_i^{-1}K_i;\\
&(R2)&\quad K_iE_j=q_i^{\delta_{ij}}E_jK_i, \quad  K_iF_j=q_i^{-\delta_{ij}}F_jK_i;\\
&(R3)&\quad E_iF_j-(-1)^{p(E_i)p(F_i)}F_jE_i=\delta_{ij}\frac{\mathcal K_i-\mathcal K_i^{-1}}{q_i-q^{-1}_i};\\
&(R4)&\quad E_iE_j=E_jE_i, \quad F_iF_j=F_jF_i, \quad|\,i{-}j\,|>1;\\
&(R5)&\quad E_i^2E_j-(q+q^{-1})E_iE_jE_i+E_jE_i^2=0,\quad |\,i{-}j\,|=1\text{ and }i\neq m;\\
&\quad&\quad F_i^2F_j-(q+q^{-1})F_iF_jF_i+F_jF_i^2=0, \quad |\,i{-}j\,|=1\text{ and }i\neq m;\\
&(R6)&\quad E_m^2=F_m^2=0;\\
&(R7)&\quad E_{m-1}E_mE_{m+1}E_m+E_mE_{m-1}E_mE_{m+1}+E_{m+1}E_mE_{m-1}E_m\\
&\quad&\quad +E_mE_{m+1}E_mE_{m-1}-(q+q^{-1})E_mE_{m-1}E_{m+1}E_m=0,\\
&\quad&\quad F_{m-1}F_mF_{m+1}F_m+F_mF_{m-1}F_mF_{m+1}+F_{m+1}F_mF_{m-1}F_m\\
&\quad&\quad +F_mF_{m+1}F_mF_{m-1}-(q+q^{-1})F_mF_{m-1}F_{m+1}F_m=0,
\end{eqnarray*}
where $\mathcal K_i=K_iK_{i+1}^{-1}$ for $i\in J$, and
\begin{displaymath}
q_i=\left\{\begin{array}{ll} q,  \quad &  i\in I_0,\\ q^{-1}, \quad & i\in I_1,
\end{array}\right.\quad\text{and}\quad
p(E_i)=p(F_i)=\delta_{im}.\leqno{(*)}
\end{displaymath}

Moreover, there is a Hopf superalgebra structure $(\Delta, \epsilon, S)$ on $U_q(\mathfrak{gl}(m|n))$
such that for all $i\in I, j\in J$:
\begin{eqnarray}
\Delta(E_j)=E_j\otimes \mathcal K_j+1\otimes E_j, & \epsilon(E_j)=0, & S(E_j)=-E_j\mathcal K_j^{-1},\label{hs1}\\
\Delta(F_j)=F_j\otimes 1+\mathcal K_j^{-1}\otimes F_j, & \epsilon(F_j)=0, & S(F_j)=-\mathcal K_jF_j,\label{hs2}\\
\Delta(K_i^{\pm1})=K_i^{\pm1}\otimes K_i^{\pm1}, & \epsilon(K_i^{\pm1})=1, & S(K_i^{\pm1})=K_i^{\mp1}\label{hs3}.
\end{eqnarray}

{\bf Remark.} We note that relation $(R3)$ and convention $(*)$ imply an important inclusion relation, that is,
the quantum general linear superalgebra $U_q(\mathfrak{gl}(m|n))$ contains the tensor product of quantum general
linear subalgebras $U_q(\mathfrak{gl}(m))\otimes U_{q^{-1}}(\mathfrak{gl}(n))$ as its sub-Hopf algebra rather than
$U_q(\mathfrak{gl}(m))\otimes U_q(\mathfrak{gl}(n))$. This is a remarkable observation in our design for constructing
some modules of $U_q(\mathfrak{gl}(m|n))$ later on.


\subsection{Arithmetic properties of $q$-binomials}

Let $\mathbbm{Z}[v,v^{-1}]$ be the Laurent polynomial ring in a variable $v$. For any integer $n\geq0$,
define
\begin{displaymath}
[\,n\,]_v:=\frac{v^n-v^{-n}}{v-v^{-1}}\quad\text{and}\quad[\,n\,]_v!:=[\,n\,]_v[\,n{-}1\,]_v\cdots[\,1\,]_v,
\end{displaymath}
both of which lie in $\mathbbm{Z}[v,v^{-1}]$.
Then for two integers $s,\,r$ with $r\geq0$, one has, see \cite{Lus2},
\begin{displaymath}
{\,s\,\brack\, r\,}_v:=\prod_{i=1}^r\frac{v^{s-i+1}-v^{-s+i-1}}{v^i-v^{-i}}\in\mathbbm{Z}[v,v^{-1}].
\end{displaymath}
Set ${\tiny{\,s\,\brack\, r\,}_v}=0$ when $r<0$. By definition, it follows that
\begin{enumerate}
\item ${\tiny{\,s\,\brack\, r\,}_v}=\frac{[\,s\,]_v!}{[\,r\,]_v![\,s-r\,]_v!}$, if $0\leq r\leq s$;
\item ${\tiny{\,s\,\brack\, r\,}_v}=0$, if $0\leq s\leq r$;
\item ${\tiny{\,s\,\brack\, r\,}_v}=(-1)^r{\tiny{\,-s+r-1\brack r\,}_v}$ if $s<0$.
\end{enumerate}

For any $q\in\Bbbk^\times,\, n\in \mathbbm{N}$ and $s, \,r\in \mathbbm{Z}$, we denote by
\begin{displaymath}
[\,n\,]:=[\,n\,]_{v=q},\quad[\,n\,]!:=[\,n\,]_{v=q}! \quad \text{and}\quad
{\tiny{\,s\,\brack\, r\,}}:={\tiny{\,s\,\brack\, r\,}_{v=q}}.
\end{displaymath}
Note that the $q$-binomial coefficients ${\tiny{\,n\,\brack\, r\,}}$, where $0\leq r\leq n$, satisfy
\begin{displaymath}
{\,n\,\brack\, r\,}=q^{r-n}{\,n{-}1\,\brack\, r{-}1\,}+q^r{\,n{-}1\,\brack\, r\,}.
\end{displaymath}
The second author introduced the \textit{characteristic} of $q$, see \cite{Hu},
which is defined as the minimal integer $\ell\in \mathbbm{Z}_+$ such that $[\,\ell\,]=0$, denoted by $\bold{char}(q)$.
It is clear that $\bold{char}(q)=0$ if and only if $q$ is generic. If $q\neq\pm1$, then the identity
$\bold{char}(q)=\ell>0$ implies that either $q$ is a $2\ell$-th primitive root of unity or $q$ is an $\ell$-th
primitive root of unity with $\ell$ odd.

\begin{lemma} $($\cite{GH, Hu,Lus,Lus2}$)$
Assume that $q\in\Bbb k^\times$ and $\bold{char}(q)=\ell\geq3$.
\begin{enumerate}[$(i)$]
\item If $s=s_0+s_1\ell,\, r=r_0+r_1\ell$ with $0\leq s_0,\; r_0<\ell,\; s_1,\; r_1\in\mathbbm{Z_+}$,
and $s\geq r$, then we have ${\tiny{\,s\,\brack\, r\,}}={\tiny{s_0 \brack r_0}\binom{s_1}{r_1}}$
when $q$ is an $\ell$-th primitive root of unity with $\ell$ odd and ${\tiny{\,s\,\brack\, r\,}}
=(-1)^{(s_1+1)r_1\ell+s_0r_1-r_0s_1}{\tiny{\,s_0 \,\brack\, r_0\,}\binom{s_1}{r_1}}$ when $q$ is a $2\ell$-th
primitive root of unity, where ${s_1\choose r_1}$ is an ordinary binomial coefficient.
\item If $s=s_0+s_1\ell$, with $0\leq s_0<\ell,\,s_1\in\mathbbm{Z}$, then we
have ${\tiny{\,s\,\brack\, \ell\,}}=s_1$ when $q$ is an $\ell$-th primitive root of unity with $\ell$ odd
and ${\tiny{\,s\,\brack\, \ell\,}}=(-1)^{(s_1+1)\ell+s_0}s_1$ when $q$ is a $2\ell$-th primitive root of unity.
\item If $s=s_0+s_1\ell,\,s'=s'_0+s'_1\ell\in\mathbbm{Z}$ with $0\leq s_0,\;s'_0<\ell$
satisfy $q^s=q^{s'},\,{\tiny\begin{matrix}\begin{bmatrix}
s\\
\ell\end{bmatrix}
\end{matrix}=\begin{matrix}\begin{bmatrix}
s'\\
\ell\end{bmatrix}
\end{matrix}}$, then $s=s'$ if $\ell$ is odd or $\ell$ is even but $s_1s'_1\geq0$; and $s'=\overline{s}=s_0-s_1\ell$
if $\ell$ is even but $s_1s'_1<0$.
\end{enumerate}
\end{lemma}

\subsection{Quantum (restricted) divided power algebras}\label{s2.5}

From \cite[Subsection~2.1]{Hu},
for any $\beta=(\beta_1,\cdots,\beta_m),\,\gamma=(\gamma_1,\cdots,\gamma_m)\in\mathbbm{Z}^m$,
one can define a map $\ast:\mathbbm{Z}^m\times\mathbbm{Z}^m\to\mathbbm{Z}$ as follows:
\begin{displaymath}
\beta\ast\gamma=\sum_{j=1}^{m-1}\sum_{i>j}\beta_i\gamma_j.
\end{displaymath}
Denote by $\epsilon_i=(0,\cdots,0,\underset{i}{1},0\cdots,0)$.
By definition, one has the following property
\begin{lemma} $($\cite[Lemma~2.1]{Hu}$)$
The product $\ast$ satisfies the following distributive laws:
\begin{eqnarray}
&(\beta+\gamma)\ast\zeta=\beta\ast\zeta+\gamma\ast\zeta,\nonumber\\
&\beta\ast(\gamma+\zeta)=\beta\ast\gamma+\beta\ast\zeta;\nonumber
\end{eqnarray}
in particular,
\begin{eqnarray}
&&\epsilon_i\ast\beta=\sum_{s<i}\beta_s,\quad\beta\ast\epsilon_i=\sum_{s>i}\beta_s\quad(1\leq i\leq m)\label{eq1}\\
&&(\epsilon_i-\epsilon_{i+1})\ast\beta=-\beta_i,\quad
\beta\ast(\epsilon_i-\epsilon_{i+1})=\beta_{i+1}\quad(1\leq i< m).\nonumber
\end{eqnarray}
\end{lemma}

For any $q \in \Bbbk^\times$, the \textit{quantum divided power algebra} $\mathcal{A}_q{(m)}$
is defined as a $\Bbbk$-vector space with the monomial basis $\{x^{(\beta)}\mid\beta\in\mathbbm Z_+^m \}$
with $x^{(0)}=1$, see \cite{Hu}. Its multiplication is given by
\begin{equation}\label{eq2}
x^{(\beta)}x^{(\gamma)}=q^{\beta\ast\gamma}{
\beta+\gamma\brack
\beta}x^{(\beta+\gamma)},
\end{equation}
where ${\tiny{\beta+\gamma\brack\beta}}:=\prod_{i=1}^m{\tiny{
\beta_i+\gamma_i\brack\beta_i}}$ and ${\tiny{
\beta_i+\gamma_i\brack\beta_i}}=[\beta_i+\gamma_i]!/[\beta_i]![\gamma_i]!$
for any $\beta_i,\gamma_i\in\mathbbm{Z}_+$. For simplicity, we denote $x^{(\epsilon_i)}$ by $x_i$ for each $i$.

In particular, when $\bold{char}(q)=\ell\geq 3$, the second author introduced a subalgebra
$\mathcal{A}_q{(m,\bold {l})}$ of $\mathcal{A}_q{(m)}$ with the basis
$\{\,x^{(\beta)}\mid\beta\in\mathbbm Z_+^m, \ \beta\leq\tau(m)\,\}$, where
$\tau(m):=(\ell-1,\cdots,\ell-1)\in\mathbbm Z_+^m $, called the quantum \textit{restricted} divided power
algebra. By $\beta\leq\gamma$ we mean $\beta_i\leq\gamma_i$ for all $i$. Since $[\,\ell\,]=0$, it is clear that
$\dim\mathcal{A}_q(m,\bold {l})=\ell^m$.

\subsection{$q$-Derivatives on $\mathcal{A}_q{(m)}$}

For each $i\in I_0$, following \cite{Hu}, we
define the algebra automorphism $\sigma_i$ of $\mathcal{A}_q{(m)}$ as
\begin{eqnarray}
&\sigma_i(x^{(\beta)})=q^{(\beta,\,\epsilon_i)}x^{(\beta)}=q^{\beta_i}x^{(\beta)},\quad\forall\; x^{(\beta)}\in\mathcal{A}_q{(m)},
\end{eqnarray}
and the special $q$-derivative $\partial_q/\partial x_i$ over $\mathcal{A}_q{(m)}$ as
\begin{align}
\frac{\partial_q}{\partial x_i}(x^{(\beta)})&=q^{-\epsilon_i\ast\beta}x^{(\beta-\epsilon_i)},\quad\forall\; x^{(\beta)}\in\mathcal{A}_q{(m)}.
\end{align}
The $U_q(\mathfrak{g})$-module algebra structure of $\mathcal{A}_q(m)$ can be realized by virtue of the
generators $\sigma_i^{\pm 1},\Theta(\pm\epsilon_i), x_i,\partial_i$ in the quantum Weyl
algebra $\mathcal{W}_q(2m)$ defined in \cite{Hu}, where $\mathfrak g=\mathfrak{gl}(m)$, or $\mathfrak{sl}(m)$.

\begin{theorem}\label{th1} $($\cite[Corollary~4.1]{Hu}$)$
For any monomial $x^{(\beta)}\in\mathcal{A}_q(m)$, $i\in I_0$, $1\leq j<m$, set
\begin{align}
E_j(x^{(\beta)})&= (x_j\partial_{j+1}\sigma_j)(x^{(\beta)})=[\,\beta_j{+}1\,]x^{(\beta+\epsilon_j-\epsilon_{j+1})},\label{r1}\\
F_j(x^{(\beta)})&= (\sigma_j^{-1}x_{j+1}\partial_j)(x^{(\beta)})=[\,\beta_{j+1}{+}1\,]x^{(\beta-\epsilon_j+\epsilon_{j+1})},\label{r1}\\
K_i(x^{(\beta)})&= \sigma_i(x^{(\beta)})=q^{\beta_i}x^{(\beta)},\\
K_i^{-1}(x^{(\beta)})&= \sigma_i^{-1}(x^{(\beta)})=q^{-\beta_i}x^{(\beta)}.\label{r3}
\end{align}
Formulae \eqref{r1}---\eqref{r3} define a $U_q(\mathfrak{gl}(m))$-module structure on $\mathcal{A}_q(m)$.
\end{theorem}
This theorem also equips $\mathcal{A}_q(m,\bold {1})$ with
a $\mathfrak{u}_q(\mathfrak{gl}(m))$-module
algebra at roots of $1$, where $\mathfrak{u}_q(\mathfrak{gl}(m)):
=U_q(\mathfrak{gl}(m))/(E_j^\ell, \,F_j^\ell, \,K_i^{2\ell}-1,\,\forall\,j<m, i\in I_0)$ is
the restricted quantum algebra which is still a Hopf algebra.

Set $|\,\beta\,|:=\sum_{i=1}^m\beta_i$ and $N:=|\,\tau(m)\,|=m(\ell{-}1)$. We denote by $U:=U_q(\mathfrak{gl}(m))$ or
$\mathfrak{u}_q(\mathfrak{gl}(m))$ and $\mathcal{A}_q:=\mathcal{A}_q(m)$ or $\mathcal{A}_q(m,\bold{1})$,
respectively. For each $t\in\mathbbm{Z_+}$,
let $\mathcal{A}_q^{(t)}$ denote by the subspace of $\mathcal{A}_q$ spanned by homogeneous elements
of degree $t$, that is, $\mathrm{span}_\Bbbk\{x^{(\beta)}\in\mathcal{A}_q\mid |\,\beta\,|=t\}.$
Thus we have $\mathcal{A}_q(m,\bold 1)^{(t)}=0$ if $t>N$.

\begin{proposition} $($\cite[Proposition~4.2]{Hu}$)$\label{pr1}
Each subspace $\mathcal{A}_q^{(t)}$ is a $U$-submodule of $\mathcal{A}_q$.
\begin{enumerate}[$(i)$]
\item If $\bold{char}(q)=0$, then each submodule $\mathcal{A}_q(m)^{(t)}\cong V(t\omega_1)$ is a simple
module generated by highest weight vector $x^{(\bold{t})}$, where
$\bold{t}=(t,0,\cdots,0)$.
\item If $\bold{char}(q)=\ell\geq3$, then the submodule $\mathcal{A}_q(m,\bold{1})^{(t)}\cong
V((\ell-1-t_i)\omega_{i-1}+t_i\omega_i)$  is a simple module generated by highest weight vector
$x^{(\bold{t})}$, where $t=(i-1)(\ell-1)+t_i$ with $0\leq t_i\leq \ell-1$ for all $i\in I_0$ and
$\bold{t}=(\ell-1,\cdots,\ell-1,\underset{i}{t_i},0,\cdots,0)$.
\end{enumerate}
\end{proposition}

\section{Quantum Grassmann superalgebra and quantum Weyl superalgebra}

\subsection{Quantum exterior superalgebras}\label{s3.1}

The \textit{quantum exterior algebra} is defined as the quotient of the free associative algebra
$\Bbbk\,\{x_{m+1},\cdots,x_{m+n}\}$ by the quadratic ideal consisting of quantum antisymmetric relations as follows
\begin{displaymath}
\Lambda_q(n):=\Bbbk\,\{x_{m+1},\cdots,x_{m+n}\}/
\langle x_i^2,\, x_jx_i+q^{-1}x_ix_j, \,j>i, \ i, j\in I_1\rangle,
\end{displaymath}
which is a (right) comodule-algebra of the quantum general (resp. special) linear group $\mathrm{GL}_q(n)$
(resp. $\mathrm{SL}_q(n))$,
dually, a (left) module-algebra of the quantum algebra $U_q(\mathfrak{gl}(n))$ (cf. \cite{Hu}).

Define
\begin{eqnarray*}
&B_s:=\{\langle i_1,\cdots,i_s\rangle\mid m+1\leq i_1<\cdots<i_s\leq m+n\},\quad \text{\it where } 1\leq s\leq n,\\ &B_0:=\varnothing, \quad\text{and}\quad B(n):=\displaystyle \cup_{s=0}^nB_s.
\end{eqnarray*}
For any $0\leq s\leq n$, if $\mu=\langle i_1,\cdots,i_s\rangle\in B_s$, then we denote by
\begin{displaymath}
x^\mu:=x_{i_1}\cdots x_{i_s}\quad\text{and}\quad {\mu}:=\epsilon_{i_1}+\cdots+\epsilon_{i_s},
\end{displaymath}
with $x^\varnothing=1$ for the case $s=0$. Set $|\mu|=s$. Then the set $\{x^\mu\mid \mu\in B(n)\}$
forms a monomial $\Bbbk$-basis of $\Lambda_q(n)$. Furthermore, if we define
\begin{displaymath}
\Lambda_q(n)_{\overline{0}}:=\mathrm{span}_
{\Bbbk}\{x^\mu\in B_s\mid s\text{ \it even }\}, \qquad
\Lambda_q(n)_{\overline{1}}:=\mathrm{span}_{\Bbbk}\{x^\mu\in B_s\mid s \text{ \it odd }\},
\end{displaymath}
then we have $\Lambda_q(n)=\Lambda_q(n)_{\overline{0}}\oplus
\Lambda_q(n)_{\overline{1}}$, which is exactly an associative $\Bbbk$-superalgebra.

\subsection{Quantum affine $(m|n)$-superspace $A_q^{m|n}$}

The natural representation of $U_q(\mathfrak{gl}(m|n))$ is its $(m+n)$-dimensional vector representation
$\bold V$ which is the analogue of the $\mathfrak{gl}(m|n)$-module $V$. Associated with such $\bold V$, we define
a notion of the so-called quantum affine $(m|n)$-superspace $A_q^{m|n}$, which is a generalization of Manin's
quantum affine $m$-space $A_q^{m|0}$ (see \cite{M}).
\begin{defi}
Associated to the natural $U_q(\mathfrak{gl}(m|n))$-module $\bold V$, the quantum affine $(m|n)$-superspace $A_q^{m|n}$
is defined to be the quotient of the free associative algebra $\Bbbk\{\,v_i\mid i\in I\,\}$ over $\Bbbk$ by the quadratic ideal
$I(\bold V)$ generated by
\begin{equation}
\begin{aligned}\label{eq0}
v_jv_i-qv_iv_j, \quad\textit{for }\ i\in I_0, \ j\in I, \ j>i;\\
v_i^2, \ v_jv_i+qv_iv_j, \quad\textit{for }\ i,\;  j\in I_1, \ j>i.
\end{aligned}
\end{equation}
\end{defi}

Obviously, $A_q^{m|n}\cong A_q^{m|0}\otimes_{\Bbbk} A_{q^{-1}}^{0|n}=A_q^{m|0}\otimes_{\Bbbk} \Lambda_{q^{-1}}(n)=A_q^{m|0}\otimes_{\Bbbk} \Lambda_{q^{-1}}(n)_{\bar 0}\oplus A_q^{m|0}\otimes_{\Bbbk} \Lambda_{q^{-1}}(n)_{\bar 1}$, as vector spaces. In fact, the quantum affine $(m|n)$-superspace $A_q^{m|n}$ is an associative $\Bbbk$-superalgebra which has a natural monomial basis consisting of $\{v^{\langle \alpha, \mu\rangle}:=v^{\alpha}\otimes v^\mu\mid \alpha=(\alpha_1,\cdots,\alpha_m)\in \mathbb Z_+^m, \mu=\langle  \mu_1, \cdots, \mu_n\rangle\in \mathbb Z_2^n\}$, where $v^\alpha=v_1^{\alpha_1}\cdots v_m^{\alpha_m}$ and $v^\mu=v_{m+1}^{\mu_1}\cdots v_{m+n}^{\mu_n}$.
It is convenient for us to consider $(m+n)$-tuple $\langle \alpha, \mu\rangle\in \mathbb Z_+^{m+n}$ with the convention:
 $v^{\langle \alpha, \mu\rangle}=0$ for $\langle \alpha, \mu\rangle\not\in \mathbb Z_+^m\times \mathbb Z_2^n$, and
 extend the $*$-product in Lemma 2 to the $(m+n)$-tuples in $\mathbb Z^{m+n}$. We still have the following
\begin{equation}
\langle\alpha, \mu\rangle * \langle \beta, \nu\rangle=\alpha*\beta+\mu*\nu+\mu*\beta=\alpha*\beta+\mu*\nu+|\mu||\beta|.
\end{equation}
Now we can write down its multiplication formula explicitly on $A_q^{m|n}$ as follows
\begin{equation}
v^{\langle \alpha, \mu\rangle}\cdot v^{\langle \beta, \nu\rangle}
=(-1)^{\mu*\nu}q^{\langle\alpha, \mu\rangle*\langle\beta, \nu\rangle} v^{\langle \alpha+\beta, \mu+\nu\rangle}.
\end{equation}

\subsection{Quantum Grassmann superalgebra}

For convenience to handle simultaneously both cases of $q\in\Bbbk^\times$ being generic or root of unity,
we shall introduce the dual object of the quantum affine $(m|n)$-superspace  $A_q^{m|n}$, so-called quantum
Grassmann superalgebra.
Now we consider the quantum divided power algebra
$\mathcal{A}_q{(m)}$ as an associative $\Bbbk$-superalgebra which is concentrated on $0$-component with a
trivial $\mathbbm{Z}_2$-grading, where
\begin{displaymath}
\mathcal{A}_q{(m)}_{\overline{0}}:=\mathcal{A}_q{(m)}\quad\text{and}\quad\mathcal{A}_q{(m)}_{\overline{1}}:=0,
\end{displaymath}
and so does the quantum restricted divided power algebra $\mathcal{A}_q{(m,\bold{1})}$.

Now construct the tensor space of the
quantum (restricted) divided power algebra and the quantum exterior algebra, that is,
\begin{displaymath}
\Omega_q(m|n):=\mathcal{A}_q(m)\otimes_\Bbbk\Lambda_{q^{-1}}(n).
\end{displaymath}
This tensor space has a natural $\mathbbm{Z}_2$-grading induced from those of $\mathcal{A}_q(m)$
and $\Lambda_{q^{-1}}(n)$, namely,
\begin{displaymath}
\Omega_q(m|n)_{\overline{0}}
=\mathcal{A}_q(m)\otimes \Lambda_{q^{-1}}(n)_{\overline{0}},\quad \Omega_q(m|n)_{\overline{1}}
=\mathcal{A}_q(m)\otimes \Lambda_{q^{-1}}(n)_{\overline{1}}.
\end{displaymath}

\begin{defi}
The quantum Grassmann superalgebra $\Omega_q(m|n)$ is defined as a superspace over $\Bbbk$ with the multiplication given by
\begin{equation}
(x^{(\alpha)}\otimes x^\mu)\cdot(x^{(\beta)}\otimes x^\nu)=q^{\mu*\beta}x^{(\alpha)}x^{(\beta)}\otimes x^\mu x^\nu,
\end{equation}
for any $x^{(\alpha)}, x^{(\beta)}\in\mathcal{A}_q(m),\, x^\mu, x^\nu\in\Lambda_{q^{-1}}(n)$,
which is an associative $\Bbbk$-superalgebra.

When $\text{\bf char}(q)=\ell\geq 3$, $\Omega_q(m|n,\bold{1}):=\mathcal{A}_q(m,\bold {1})\otimes\Lambda_{q^{-1}}(n)$
is a sub-superalgebra, which is referred to as the quantum restricted Grassmann superalgebra.
\end{defi}

Note that our definition of \textit{the quantum Grassmann superalgebra} is different from that defined in \cite{Kac1}.

From Subsections \ref{s2.5} and \ref{s3.1}, we see that the set
\begin{displaymath}
\{\,x^{(\beta)}\otimes x^\mu\mid\beta\in\mathbbm Z_+^m,\, \mu\in \mathbbm Z_2^n\,\}
\end{displaymath}
forms a monomial $\Bbbk$-basis of $\Omega_q(m|n)$,  and the set
\begin{displaymath}
\{\,x^{(\beta)}\otimes x^\mu\mid\beta\in\mathbbm Z_+^m,\,\beta\leq\tau(m),\, \mu\in\mathbbm Z_2^n\,\}
\end{displaymath}
forms a monomial $\Bbbk$-basis of $\Omega_q(m|n,\bold{1})$.

In particular, over the quantum Grassmann superalgebra $\Omega_q$, we will need the parity automorphism
$\tau: \Omega_q\longrightarrow \Omega_q$ of order $2$ defined by
\begin{equation}
\tau(x^{(\alpha)}\otimes x^\mu)=(-1)^{|\mu|}x^{(\alpha)}\otimes x^\mu, \quad\forall\  x^{(\alpha)}\otimes x^\mu\in\Omega_q.
\end{equation}

\subsection{Quantum differential operators on $\Omega_q$}

Set $\Omega_q:=\Omega_q(m|n)$ or $\Omega_q(m|n, \bold 1)$ when $\text{\bf char}(q)=l\ge 3$.
For each $i\in I_1$, and $x^\mu\in\Lambda_{q^{-1}}(n)$,
define the algebra automorphisms $\sigma_i$ and $\tau_i$ on $\Lambda_{q^{-1}}(n)$ as
\begin{equation}
\begin{aligned}\label{eq0}
\sigma_i(x^\mu)=(-q)^{({\mu},\,\epsilon_i)}x^\mu,\\
\tau_i(x^\mu)=(-1)^{(\mu,\epsilon_i)}x^\mu.
\end{aligned}
\end{equation}
Both the algebra automorphisms $\sigma_i\ (i\in I_{0})$ defined on $\mathcal{A}_q$ and $\sigma_i, \ \tau_i\ (i\in I_{1})$
defined on $\Lambda_{q^{-1}}(n)$ can be extended to the quantum Grassmann superalgebra $\Omega_q$.
Indeed, for any $x^{(\beta)}\otimes x^\mu\in\Omega_q$, we have
\begin{equation}
\begin{aligned}\label{eq3}
\sigma_i(x^{(\beta)}\otimes x^\mu):=\sigma_i(x^{(\beta)})\otimes x^\mu,
\quad\text{ \it for } i\in I_0,\\
\sigma_i(x^{(\beta)}\otimes x^\mu):=x^{(\beta)}\otimes \sigma_i(x^\mu),
\quad\text{ \it for } i\in I_1,\\
\tau_i(x^{(\beta)}\otimes x^\mu):=x^{(\beta)}\otimes \tau_i(x^\mu),
\quad\text{ \it for } i\in I_1.
\end{aligned}
\end{equation}
By definition, $\sigma_i\sigma_j=\sigma_j\sigma_i$ ($i, j\in I$); $\tau_i\sigma_j=\sigma_j\tau_i$ ($i\in I_1, j\in I$);
$\tau_i^2=\mathrm{id}$
and $\tau_i\tau_j=\tau_j\tau_i$ ($i, j\in I_1$).

For each $i\in I_1$, define the special $q$-derivative $\partial_q/\partial x_i$ on $\Lambda_{q^{-1}}(n)$ as
\begin{align}\label{eq4}
\frac{\partial_q}{\partial x_i}(x^\mu)=(-q)^{-\epsilon_i*\mu}x^{\mu-\epsilon_i}=\delta_{1,\mu_i}(-q)^{-\epsilon_i*\mu}
x^{\mu-\epsilon_i},
\end{align}
where $\mu\in\mathbb Z_2^n$, and (3.8) will vanish if $\mu-\epsilon_i\not\in\mathbb Z_2^n$.

Now, let us extend both the $q$-derivations $\partial_q/\partial x_i \ (i\in I_0)$ defined on $\mathcal{A}_q$
and $\partial_q/\partial x_i \ (i\in I_1)$ defined on $\Lambda_{q^{-1}}(n)$ to $\Omega_q$, that is, for
any $x^{(\beta)}\otimes x^\mu\in\Omega_q$, we have
\begin{equation}
\begin{split}\label{eq5}
\frac{\partial_q}{\partial x_i}(x^{(\beta)}\otimes x^\mu):&=\frac{\partial_q}{\partial x_i}(x^{(\beta)})\otimes x^\mu,
\quad\text{ \it for } i\in I_0,\\
\frac{\partial_q}{\partial x_i}(x^{(\beta)}\otimes x^\mu):&
=q^{-\epsilon_i*\beta}x^{(\beta)}\otimes\frac{\partial_q}{\partial x_i} (x^\mu),
\quad \text{ \it for } i\in I_1.
\end{split}
\end{equation}
For short, we use the notation $\partial_i$ to denote $\partial_q/\partial x_i$, for each $i$.

By definition (cf. (2.7), (3.8) \& (3.9)), it is easy to check that the $q$-derivations $\partial_i$'s ($i\in I$) defined over
$\Omega_q(m|n)$ satisfy the same relations as (3.1). Actually we have
\begin{proposition}
The associative $\Bbbk$-superalgebra $\mathcal D_q(m|n)$ generated by the $q$-derivations $\partial_i$'s $(i\in I)$
defined over the quantum Grassmann superalgebra $\Omega_q(m|n)$, in which $\mathcal D_q(m|n)=\mathcal D_q(m|n)_{\bar 0}\oplus \mathcal D_q(m|n)_{\bar 1}$ with the generator parity $|\partial_i|=\bar 1$ only for $i\in I_1$, is exactly isomorphic to the quantum affine
$(m|n)$-superspace $A_q^{m|n}$.
\end{proposition}
\begin{proof} Note that for each $i\in I_1$, we have $\partial_i^2=0$, by (3.9).

(i) When $i, j\in I_0$, $i<j$: by defining formula (3.9) of $\partial_i$ ($i\in I_0$), $\partial_i,\ \partial_j$ are essentially defined on $\mathcal A_q$. We have the usual relations:  $\partial_j\partial_i=q\partial_i\partial_j$,
by Proposition 3.1 (3) in \cite{Hu}.

(ii) When $i\in I_0$, $j\in I_1$ and $i<j$: in this case, noting that $\epsilon_j*\epsilon_i=1$, and by defining formula (3.9), we have
\begin{gather*}
\partial_j\partial_i(x^{(\alpha)}\otimes x^\mu)=q^{-\epsilon_j*(\alpha-\epsilon_i)}\partial_i(x^{(\alpha)})\otimes\partial_j(x^\mu),\\
\partial_i\partial_j(x^{(\alpha)}\otimes x^\mu)=q^{-\epsilon_j*\alpha}\partial_i(x^{(\alpha)})\otimes\partial_j(x^\mu),
\end{gather*}
that is, $\partial_j\partial_i=q\partial_i\partial_j$.

(iii) When $i, j\in I_1$, $i<j$: in this case, noting that $\epsilon_i*\epsilon_j=0$, and by defining formula (3.9), we have
\begin{equation*}
\begin{split}
\partial_j\partial_i(x^{(\alpha)}\otimes x^\mu)&=q^{-(\epsilon_i+\epsilon_j)*\alpha}(-q)^{-\epsilon_i*\mu-\epsilon_j*(\mu-\epsilon_i)}x^{(\alpha)}\otimes x^{\mu-\epsilon_i-\epsilon_j},\\
\partial_i\partial_j(x^{(\alpha)}\otimes x^\mu)&=q^{-(\epsilon_i+\epsilon_j)*\alpha}(-q)^{-\epsilon_j*\mu-\epsilon_i*(\mu-\epsilon_j)}x^{(\alpha)}\otimes x^{\mu-\epsilon_i-\epsilon_j},
\end{split}
\end{equation*}
that is, $\partial_j\partial_i=-q\partial_i\partial_j$.

Combining the relations of three cases (i), (ii) \& (iii) above with relations $\partial_i^2=0$ ($i\in I_1$), and noting that no more other relations among $\partial_i$'s\ $(i\in I)$ occurred, we get the required result.
\end{proof}

In this subsection, we shall give a sufficient description on the properties of the quantum differential operators we defined above. To this end, we extend the bicharacter $\theta_+$ on the abelian group $\mathbb Z^m$ (see the notation $\theta$ in {\bf 2.1} of \cite{Hu}) to the abelian group $\mathbb Z^m\times\mathbb Z_2^n$ using the extended $*$-product (3.2), i.e., we define a mapping $\theta: (\mathbb Z^m\times\mathbb Z_2^n)\times (\mathbb Z^m\times\mathbb Z_2^n)\longrightarrow \Bbbk^*$ by
\begin{equation}
\theta(\langle\alpha,\mu\rangle,\langle\beta,\nu\rangle)
=(-1)^{\mu*\nu-\nu*\mu}q^{(\alpha+\mu)*\beta-(\beta+\nu)*\alpha+\mu*\nu-\nu*\mu}
=\theta_+(\alpha,\beta)\theta_-(\mu,\nu)q^{\mu*\beta-\nu*\alpha},
\end{equation}
where $\theta_+(\alpha,\beta)=q^{\alpha*\beta-\beta*\alpha}$, and $\theta_-(\mu,\nu)=(-q)^{\mu*\nu-\nu*\mu}$.

Associated to the bicharacter $\theta$ on $\mathbb Z^m\times\mathbb Z_2^n$, we can define super-algebra automorphisms
$\Theta(\epsilon_i)$ ($i\in I$) on $\Omega_q(m|n)$ as follows
\begin{equation}
\begin{split}
\Theta(\epsilon_i)(x^{(\alpha)}{\otimes} x^\mu)
&=\theta(\langle\epsilon_i,0\rangle,\langle\alpha,\mu\rangle)(x^{(\alpha)}{\otimes} x^\mu)=\theta_+(\epsilon_i,\alpha)q^{-\mu*\epsilon_i} (x^{(\alpha)}{\otimes} x^\mu),\ \forall\;  i\in I_0;\\
\Theta(\epsilon_i)(x^{(\alpha)}{\otimes} x^\mu)&=\theta(\langle 0,\epsilon_i\rangle,\langle\alpha,\mu\rangle)(x^{(\alpha)}{\otimes}x^\mu)
=\theta_-(\epsilon_i,\mu) q^{\epsilon_i*\alpha}(x^{(\alpha)}{\otimes} x^\mu),\ \forall\;  i\in I_1.
\end{split}
\end{equation}

\begin{remark} Actually, by abuse of notation, considering $\mathbb Z^m\times\mathbb Z_2^n$ as a subset of $\mathbb Z^{m+n}$ and writing $\alpha+\mu:=\langle\alpha,\mu\rangle$, for $\langle\alpha,\mu\rangle\in \mathbb Z^m\times\mathbb Z_2^n$.
In general, we can define the super-algebra automorphisms $\Theta(\alpha+\mu)$ for
$\langle\alpha,\mu\rangle\in\mathbb Z^m\times\mathbb Z_2^n$ by the bicharacter $\theta$ in $(3.10)$.
Furthermore, we can rewrite the mapping $\theta: (\mathbb Z^m\times\mathbb Z_2^n)\times (\mathbb Z^m\times\mathbb Z_2^n)\longrightarrow \Bbbk^*$ as follows:
\begin{equation}
\theta(\epsilon_i,\epsilon_j)=\begin{cases}q^{\epsilon_i*\epsilon_j-\epsilon_j*\epsilon_i}, & \textit{for }\ i\in I_0, j\in I;\\
(-q)^{\epsilon_i*\epsilon_j-\epsilon_j*\epsilon_i}, & \textit{for }\ i, j\in I_1.
\end{cases}
\end{equation}
In this sense, the relations in Proposition 7 can be expressed uniformly as follows
\begin{equation}
\partial_i\partial_j=\theta(\epsilon_i, \epsilon_j)\partial_j\partial_i,\quad \textit{for }\ i\not= j.
\end{equation}
\end{remark}

Similarly to Proposition 3.1 in \cite{Hu}, we have the following
\begin{lemma} $(1)$ For each $i\in I_0$, $\partial_i$ is a $\bigl(\Theta(-\epsilon_i)\sigma_i^{\pm1}, \sigma_i^{\mp1}\bigr)$-derivation of $\Omega_q$, namely,
$$\partial_i\bigl((x^{(\beta)}{\otimes} x^\mu)(x^{(\gamma)}{\otimes} x^\nu)\bigr)=
\partial_i(x^{(\beta)}{\otimes} x^\mu)\sigma_i^{\mp1}(x^{(\gamma)}{\otimes} x^\nu)+
(\Theta(-\epsilon_i)\sigma_i^{\pm1})(x^{(\beta)}{\otimes} x^\mu)\partial_i(x^{(\gamma)}{\otimes} x^\nu).
$$

For each $i\in I_1$, $\partial_i$ is a $\bigl(\Theta(-\epsilon_i)\tau_i, 1\bigr)$-derivation of $\Omega_q$, namely,
$$\partial_i\bigl((x^{(\beta)}{\otimes} x^\mu)(x^{(\gamma)}{\otimes} x^\nu)\bigr)=
\partial_i(x^{(\beta)}{\otimes} x^\mu)(x^{(\gamma)}{\otimes} x^\nu)+
\bigl(\Theta(-\epsilon_i)\tau_i\bigr)(x^{(\beta)}{\otimes} x^\mu)\partial_i(x^{(\gamma)}{\otimes} x^\nu).
$$

$(2)$ $\Theta(\alpha{+}\mu)\Theta(\beta{+}\nu)=\Theta(\alpha{+}\beta{+}\mu{+}\nu)$, in particular, $\Theta(-\alpha_i)=\sigma_i\sigma_{i+1}$, $\alpha_i=\epsilon_i{-}\epsilon_{i+1}$, for $i\in J$ and $i\not=m$; $\Theta(-\alpha_m)=\sigma_m\sigma_{m+1}\tau$, where $\tau=\Pi_{j\in I_1}\tau_j$ is the parity automorphism of $\Omega_q$ in a super root system of type $A(m{-}1|n{-}1)$.

\smallskip

$(3)$ $\Theta(\epsilon_j)\,\partial_i\,\Theta(\epsilon_j)^{-1}=
\theta\,(\epsilon_i, \epsilon_j)\partial_i$; \ $\tau\partial_i\tau^{-1}=(-1)^{\delta_{\bar 1,|\partial_i|}}\partial_i$; \  $\,\sigma_j\partial_i\sigma_j^{-1}=(-1)^{\delta_{\bar 1,|\partial_i|}\delta_{ij}}q^{-\delta_{ij}}\partial_i$, $(i,\, j\in I)$; \ $\tau_i\partial_j=(-1)^{\delta_{ij}}\partial_j\tau_i$, \ $(i\in I_1, j\in I)$.

\smallskip

$(4)$ $(x^{(\alpha)}{\otimes} x^\mu)\bigl((x^{(\beta)}{\otimes} x^\nu)(x^{(\gamma)}{\otimes} x^\eta)\bigr)=\bigl((x^{(\alpha)}{\otimes} x^\mu)(x^{(\beta)}{\otimes} x^\nu)\bigr)(x^{(\gamma)}{\otimes} x^\eta)$

\hskip6.4cm$=\theta(\alpha{+}\mu,\beta{+}\nu)(x^{(\beta)}{\otimes} x^\nu)\bigl((x^{(\alpha)}{\otimes} x^\mu)(x^{(\gamma)}{\otimes} x^\eta)\bigr)
$

\hskip6.4cm$=\Theta(\alpha{+}\mu)(x^{(\beta)}{\otimes} x^\nu)\bigl((x^{(\alpha)}{\otimes} x^\mu)(x^{(\gamma)}{\otimes} x^\eta)\bigr)$.

\smallskip

$(5)$ For each $i\in I_0$, $(x^{(\alpha)}{\otimes} x^\mu)\partial_i$ is a $(\Theta(\alpha{+}\mu{-}\epsilon_i)\sigma_i^{\pm1}, \sigma_i^{\mp1})$-derivation of $\Omega_q$, namely,

\noindent\hskip0.2cm
$\bigl((x^{(\alpha)}{\otimes} x^\mu)\partial_i\bigr)\bigl((x^{(\beta)}{\otimes} x^\nu)(x^{(\gamma)}{\otimes} x^\eta)\bigr)=
\bigl((x^{(\alpha)}{\otimes} x^\mu)\partial_i\bigr)\bigl(x^{(\beta)}{\otimes} x^\nu\bigr)\sigma_i^{\pm1}(x^{(\gamma)}{\otimes} x^\eta)$

\hskip6.2cm$+\bigl(\Theta(\alpha{+}\mu{-}\epsilon_i)\sigma_i^{\mp1}\bigr)(x^{(\beta)}{\otimes} x^\nu)\bigl((x^{(\alpha)}{\otimes} x^\mu)\partial_i\bigr)\bigl(x^{(\gamma)}{\otimes} x^\eta\bigr)$.

\smallskip

For each $i\in I_1$, $(x^{(\alpha)}{\otimes} x^\mu)\partial_i$ is a $(\Theta(\alpha{+}\mu{-}\epsilon_i)\tau_i, 1)$-derivation of $\Omega_q$, namely,

\noindent\hskip0.2cm
$\bigl((x^{(\alpha)}{\otimes} x^\mu)\partial_i\bigr)\bigl((x^{(\beta)}{\otimes} x^\nu)(x^{(\gamma)}{\otimes} x^\eta)\bigr)=
\bigl((x^{(\alpha)}{\otimes} x^\mu)\partial_i\bigr)\bigl(x^{(\beta)}{\otimes} x^\nu\bigr)(x^{(\gamma)}{\otimes} x^\eta)$

\hskip6.25cm$+\bigl(\Theta(\alpha{+}\mu{-}\epsilon_i)\tau_i\bigr)(x^{(\beta)}{\otimes} x^\nu)\bigl((x^{(\alpha)}{\otimes} x^\mu)\partial_i\bigr)\bigl(x^{(\gamma)}{\otimes} x^\eta\bigr)$.

\end{lemma}
\begin{proof} (1)
(i) When $i\in I_0$, by multiplication formula (3.4), defining formula (3.9) of $\partial_i \ (i\in I_0)$, and the fact that $\partial_i$ is a $(\Theta(-\epsilon_i)\sigma_i^{\pm 1}, \sigma_i^{\mp1})$-derivation of $\mathcal A_q$ (see Proposition 3.1 (2) in \cite{Hu}), as well as (3.11) \& (3.7), we have
\begin{equation*}
\begin{split}
LHS&=q^{\mu*\gamma}\partial_i(x^{(\beta)}x^{(\gamma)})\otimes x^\mu x^\nu \\
&=q^{\mu*\gamma}\Bigl(\partial_i(x^{(\beta)})\sigma_i^{\pm 1}(x^{(\gamma)})\otimes x^\mu x^\nu+
\bigl(\Theta(-\epsilon_i)\sigma_i^{\mp1}\bigr)(x^{(\beta)})\partial_i(x^{(\gamma)})\otimes  x^\mu x^\nu\Bigr)\\
&=(\partial_i(x^{(\beta)})\otimes x^\mu)(\sigma_i^{\pm 1}(x^{(\gamma)})\otimes x^\nu)+
\bigl(\Theta(-\epsilon_i)\sigma_i^{\mp1}\bigr)(x^{(\beta)}\otimes x^\mu)\bigl(\partial_i(x^{(\gamma)})\otimes x^\nu\bigr)\\
&=RHS.
\end{split}
\end{equation*}

(ii) When $i\in I_1$, by multiplication formula (3.4), defining formulae (3.9) \& (3.8) of $\partial_i \ (i\in I_1)$, and (3.11), as well as the fact $x^\mu x^\nu=(-q)^{\mu*\nu}x^{\mu+\nu}$, we have
\begin{equation*}
\begin{split}
LHS&=q^{\mu*\gamma-\epsilon_i*(\beta+\gamma)}x^{(\beta)}x^{(\gamma)}\otimes \partial_i(x^\mu x^\nu)\\
&=q^{\mu*\gamma-\epsilon_i*(\beta+\gamma)}(-q)^{\mu*\nu-\epsilon_i*(\mu+\nu)}\delta_{1,\mu_i+\nu_i}x^{(\beta)}x^{(\gamma)}\otimes x^{\mu+\nu-\epsilon_i}\\
&=q^{\mu*\gamma-\epsilon_i*(\beta+\gamma)}x^{(\beta)}x^{(\gamma)}\otimes \bigl(\delta_{1,\mu_i}\delta_{0,\nu_i}\partial_i(x^{\mu})x^\nu+\delta_{0,\mu_i}\delta_{1,\nu_i}\theta_-(-\epsilon_i,\mu)x^\mu \partial_i(x^{\nu})\bigr)\\
&=q^{-\epsilon_i*\beta}(x^{(\beta)}{\otimes} \partial_i(x^\mu))(x^{(\gamma)}{\otimes} \delta_{0,\nu_i}x^\nu)+q^{-\epsilon_i*(\beta+\gamma)}\theta_-(-\epsilon_i,\mu)(x^{(\beta)}{\otimes} \delta_{0,\mu_i}x^\mu)(x^{(\gamma)}{\otimes} \partial_i(x^\nu))\\
&=\partial_i(x^{(\beta)}\otimes x^\mu)(x^{(\gamma)}\otimes \delta_{0,\nu_i}x^\nu)+\theta_-(-\epsilon_i,\mu)q^{-\epsilon_i*\beta}(x^{(\beta)}\otimes \delta_{0,\mu_i}x^\mu)\partial_i(x^{(\gamma)}\otimes x^\nu)\\
&=\partial_i(x^{(\beta)}\otimes x^\mu)(x^{(\gamma)}\otimes x^\nu)+\bigl(\Theta(-\epsilon_i)\tau_i\bigr)(x^{(\beta)}\otimes x^\mu)\partial_i(x^{(\gamma)}\otimes x^\nu)\\
&=RHS.
\end{split}
\end{equation*}

(2) The first one follows from (3.10). For any $i<m$, by Lemma 2.1 (2) \& (3) in \cite{Hu}, we have $\theta_+(-\epsilon_i+\epsilon_{i+1}, \alpha)=q^{\alpha_i+\alpha_{i+1}}$ and $\mu*(\epsilon_i-\epsilon_{i+1})=0$, so (3.11) gives the second claim. For $i>m$, since $(\epsilon_i-\epsilon_{i+1})*\alpha=0$, $-(\epsilon_i-\epsilon_{i+1})*\mu=\mu_i$ and $\mu*(\epsilon_i-\epsilon_{i+1})=\mu_{i+1}$ (ibid.), we also get the second claim,
by (3.11). Now let us look at the interesting case $i=m$: noting $\alpha*\epsilon_m=0$, $\epsilon_{m+1}*\mu=0$ and $\mu*\epsilon_m=|\mu|$ (ibid.), by (3.11) \& (3.5), we have
\begin{equation*}
\begin{split}
\Theta(-\alpha_m)(x^{(\alpha)}\otimes x^\mu)&=\theta_+(-\epsilon_m,\alpha)\theta_-(\epsilon_{m+1},\mu)q^{\epsilon_{m+1}*\alpha+
\mu*\epsilon_m}(x^{(\alpha)}\otimes x^\mu)\\
&=q^{(-\epsilon_m+\epsilon_{m+1})*\alpha+|\mu|}(-q)^{-\mu*\epsilon_{m+1}}(x^{(\alpha)}\otimes x^\mu)\\
&=q^{\alpha_m+|\mu|}(-q)^{-|\mu|+\mu_{m+1}}(x^{(\alpha)}\otimes x^\mu)\\
&=(-1)^{|\mu|}q^{\alpha_m}(-q)^{\mu_{m+1}}(x^{(\alpha)}\otimes x^\mu)\\
&=\sigma_m\sigma_{m+1}\tau(x^{(\alpha)}\otimes x^\mu).
\end{split}
\end{equation*}

(3) For $i,\ j\in I_0$: by (3.11) \& (3.12), we have
\begin{equation*}
\begin{split}
\Theta(\epsilon_j)\partial_i\Theta(-\epsilon_j)(x^{(\alpha)}\otimes x^\mu)&=
\theta_+(-\epsilon_j,\alpha)q^{\mu*\epsilon_j}\Theta(\epsilon_j)(\partial_i(x^{(\alpha)})\otimes x^\mu)\\
&=\theta_+(\epsilon_j,-\epsilon_i)\partial_i(x^{(\alpha)}\otimes x^\mu)\\
&=\theta(\epsilon_i,\epsilon_j)\partial_i(x^{(\alpha)}\otimes x^\mu).
\end{split}
\end{equation*}

For $i\in I_0,\ j\in I_1$: by (3.11) \& (3.12), we have
\begin{equation*}
\begin{split}
\Theta(\epsilon_j)\partial_i\Theta(-\epsilon_j)(x^{(\alpha)}\otimes x^\mu)&=
\theta_-(-\epsilon_j,\mu)q^{-\epsilon_j*\alpha}\Theta(\epsilon_j)(\partial_i(x^{(\alpha)})\otimes x^\mu)\\
&=q^{-\epsilon_j*\epsilon_i}\partial_i(x^{(\alpha)}\otimes x^\mu)\\
&=\theta(\epsilon_i,\epsilon_j)\partial_i(x^{(\alpha)}\otimes x^\mu).
\end{split}
\end{equation*}

For $i\in I_1,\ j\in I_0$: by (3.11), (3.9) \& (3.12), we have
\begin{equation*}
\begin{split}
\Theta(\epsilon_j)\partial_i\Theta(-\epsilon_j)(x^{(\alpha)}\otimes x^\mu)&=
\theta_-(-\epsilon_j,\mu)q^{-\epsilon_j*\alpha-\epsilon_i*\alpha}\Theta(\epsilon_j)(x^{(\alpha)}\otimes \partial_i(x^\mu))\\
&=\theta_-(\epsilon_j,-\epsilon_i)\partial_i(x^{(\alpha)}\otimes x^\mu)\\
&=\theta(\epsilon_i,\epsilon_j)\partial_i(x^{(\alpha)}\otimes x^\mu).
\end{split}
\end{equation*}

For $i,\ j\in I_1$: by (3.11), (3.9) \& (3.12), we have
\begin{equation*}
\begin{split}
\Theta(\epsilon_j)\partial_i\Theta(-\epsilon_j)(x^{(\alpha)}\otimes x^\mu)&=
\theta_-(-\epsilon_j,\mu)q^{-\epsilon_j*\alpha-\epsilon_i*\alpha}\Theta(\epsilon_j)(x^{(\alpha)}\otimes \partial_i(x^\mu))\\
&=\theta_-(\epsilon_j,-\epsilon_i)\partial_i(x^{(\alpha)}\otimes x^\mu)\\
&=\theta(\epsilon_i,\epsilon_j)\partial_i(x^{(\alpha)}\otimes x^\mu).
\end{split}
\end{equation*}

For the second identity: by definition, we have
$$\tau\partial_i\tau^{-1}(x^{(\alpha)}\otimes x^\mu)=(-1)^{\delta_{\bar 1,|\partial_i|}}\partial_i(x^{(\alpha)}\otimes x^\mu).$$

For the third identity: when $i, j\in I_0$, this is the same as {\bf 3.2} (iv) in \cite{Hu};
when $i, j\in I_1$, by definition, we have
$$\sigma_j\partial_i\sigma_j^{-1}(x^{(\alpha)}\otimes x^\mu)=(-1)^{\delta_{\bar 1,|\partial_i|}\delta_{ij}}q^{-\delta_{ij}}\partial_i(x^{(\alpha)}\otimes x^\mu);$$
in both cases when
$i\in I_0, j\in I_1$ or $i\in I_1, j\in I_0$, by definition, we get $\sigma_j\partial_i\sigma_j^{-1}=\partial_i$.

It is obvious to see the last identity: $\tau_i\partial_j=(-1)^{\delta_{ij}}\partial_j\tau_i$, by definition.

\smallskip
(4) is clear.

\smallskip
(5) follows from (2) \& (4).
\end{proof}

\subsection{Quantum differential Hopf algebra $\mathfrak D_q(m|n)$}

Now we are in a position to construct the quantum differential Hopf algebra $\mathfrak D_q(m|n)$ to be the bosonization
of $\mathcal D_q(m|n)$, the latter viewing as a Nichols superalgebra, which generalizes the structure of quantum differential Hopf algebra $\mathfrak D_q(m)$ defined in {\bf 3.2} \cite{Hu} by the second author.

\begin{defi} Let $\mathfrak D_q(m|n)$ be the associative algebra over $\Bbbk$ generated by elements $\partial_i\in \mathcal D_q(m|n),\ \Theta(\pm\epsilon_i),\ \sigma_i\, (i\in I)$, $\tau_j\, (j\in I_1)$, as well as the parity element $\tau=\Pi_{j\in I_1}\tau_j$, associated to the bicharacter
$\theta$ on $\mathbb Z^m\times \mathbb Z_2^n$ given in $(3.10)$, subject to the following relations:
\begin{gather}
\sigma_i\sigma_i^{-1}=1=\sigma_i^{-1}\sigma_i, \quad \sigma_i\sigma_j=\sigma_j\sigma_i, \quad (i, j\in  I), \\
\tau^2=1=\tau_i^2, \quad \tau_i\tau_j=\tau_j\tau_i,\quad (i, j\in I_1), \quad \sigma_i\tau_j=\tau_j\sigma_i,\quad(i\in I, j\in I_1),
\\
\Theta(\epsilon_i)\Theta(-\epsilon_i)=1=\Theta(-\epsilon_i)\Theta(\epsilon_i), \quad
\Theta(\epsilon_i)\Theta(\epsilon_j)=\Theta(\epsilon_i{+}\epsilon_j)=\Theta(\epsilon_j)\Theta(\epsilon_i),\\
\sigma_j\Theta(\epsilon_i)=\Theta(\epsilon_i)\sigma_j, \quad (i, j\in I), \quad \tau_j\Theta(\epsilon_i)=\Theta(\epsilon_i)\tau_j, \quad (i\in I, j\in I_1),\\
\Theta(-\epsilon_m{+}\epsilon_{m+1})=
\tau\sigma_m\sigma_{m+1},\quad\Theta(-\epsilon_i{+}\epsilon_{i+1})=\sigma_i\sigma_{i+1}, \quad (m\not=i\in J), \\
\Theta(\epsilon_j)\partial_i\Theta(\epsilon_j)^{-1}=\theta(\epsilon_i,\epsilon_j)\partial_i, \quad \sigma_j\partial_i\sigma_j^{-1}=(-1)^{\delta_{\bar 1,|\partial_i|}\delta_{ij}}q^{-\delta_{ij}}\partial_i, \quad (i, j\in I),\\
\tau_j\partial_i=(-1)^{\delta_{ij}}\partial_i\tau_j, \quad (i\in I, j\in I_1),\\
\partial_i\partial_j=\theta(\epsilon_i,\epsilon_j)\partial_j\partial_i, \quad (i, j\in I),\\
\partial_j^2=0, \quad (j\in I_1).
\end{gather}

Moreover, $\mathfrak D_q(m|n)$ can be equipped with the following mappings: $\Delta$, $\epsilon$,
$S$ on the generators of $\mathfrak D_q(m|n)$ as
\begin{gather}
\Delta: \mathfrak D_q(m|n)\longrightarrow
\mathfrak D_q(m|n)\otimes \mathfrak D_q(m|n)\\
\Delta(\sigma_i^{\pm 1})=\sigma_i^{\pm1}\otimes \sigma_i^{\pm1}, \quad (i\in I)\qquad \Delta(\tau_j)=\tau_j\otimes \tau_j,\quad (j\in I_1)\notag\\
\Delta(\Theta(\pm\epsilon_i))=\Theta(\pm\epsilon_i)\otimes \Theta(\pm\epsilon_i), \quad (i\in I)\notag\\
\Delta(\partial_i)=\partial_i\otimes \sigma_i^{-1}+\Theta(-\epsilon_i)\sigma_i\otimes \partial_i, \quad (i\in I_0)\notag\\
\Delta(\partial_j)=\partial_j\otimes 1+\Theta(-\epsilon_j)\tau_j\otimes \partial_j, \quad (j\in I_1)\notag\\
\epsilon: \mathfrak D_q(m|n)\longrightarrow \Bbbk\\
\epsilon(\sigma_i^{\pm1})=\epsilon(\Theta(\pm\epsilon_i))=\epsilon(\tau_j)=1,\notag\\
\epsilon(\partial_i)=0,\notag\\
S: \mathfrak D_q(m|n)\longrightarrow \mathfrak D_q(m|n)\\
S(\sigma_i^{\pm1})=\sigma_i^{\mp1},\quad
S(\Theta(\pm\epsilon_i))=\Theta(\mp\epsilon_i), \quad (i\in I), \quad S(\tau_j)=\tau_j,\quad (j\in I_1),\notag\\
S(\partial_i)=-q\,\Theta(\epsilon_i)\partial_i, \quad (i\in I_0), \quad S(\partial_j)=-\Theta(\epsilon_j)\tau_j\partial_j, \quad (j\in I_1). \notag
\end{gather}
\end{defi}

Again we extend the definitions of $\Delta$, $\epsilon$ (resp. $S$) on $\mathfrak D_q(m|n)$ (anti-)algebraically.
Thus we obtain the following
\begin{theorem}
$(\mathfrak D_q(m|n), \Delta, \epsilon, S)$ is a pointed Hopf algebra with the above relations (3.14)--(3.25).
\end{theorem}
\begin{proof}
First, by Proposition 7, we note that $\mathcal D_q(m|n)=\mathcal D_q(m|n)_{\bar0}\oplus \mathcal D_q(m|n)_{\bar1}$, where the skew-derivation generators $\partial_i$ ($i\in I_0$) are of parity $\bar0$, and $\partial_i$ ($i\in I_1$) of parity $\bar1$, and the
multiplication on $\mathcal D_q(m|n)\otimes \mathcal D_q(m|n)$ is given by
$$(\partial^\alpha\otimes \partial^\mu)(\partial^\beta\otimes \partial^\nu)=q^{\mu*\beta}\partial^\alpha\partial^\beta\otimes \partial^\mu\partial^\nu,$$
where $\partial^\alpha=\Pi_{i\in I_0}\partial_i^{\alpha_i}$ for $\alpha\in\mathbb Z_+^m$, $\partial^\mu=\Pi_{j\in I_1}\partial_j^{\mu_j}$ for $\mu\in\mathbb Z_2^n$. Clearly,
we have $\mathcal D_q(m|n)_{\bar i}\cdot\mathcal D_q(m|n)_{\bar j}\subseteq \mathcal D_q(m|n)_{\bar i+\bar j}$.
That is, the super vector subspace $\mathcal D_q(m|n)$ is a superalgebra with respect to the above multiplication,
which can be regarded as a Nichols superalgebra with generators $\partial_i \ (i\in I)$ as its all primitive elements.

In fact, we shall show that $\mathfrak D_q(m|n)$ is a bosonization of the Nichols superalgebra $\mathcal D_q(m|n)$, which is a usual Hopf algebra.
Besides elements in $\mathcal D_q(m|n)$, all other generators $\sigma_i,\ \Theta(\epsilon_i)\ (i\in I)$, $\tau_j\ (j\in I_1)$,  are group-likes. So it is clear that $\Delta$, $\epsilon$ and $S$ preserve the algebraic relations (3.14)--(3.18) of $\mathfrak D_q(m|n)$.

So it needs to check it for $\Delta$ with respect to relations (3.19)--(3.22), while it is trivial to check the relations (3.19) \& (3.20). Note that relation (3.21) for the case when $i, j\in I_0$ has been checked in Theorem 3.2 in \cite{Hu}. It remains to check it for the cases: (i) $i\in I_0, j\in I_1$; (ii) $i, j\in I_1$.
Note that
\begin{equation*}
\begin{split}
(\Theta(-\epsilon_i)\sigma_i)\partial_j&=\theta(\epsilon_i,\epsilon_j)\partial_j(\Theta(-\epsilon_i)\sigma_i), \quad (i\in I_0, j\in I_1),\\
\partial_i(\Theta(-\epsilon_j)\tau_j)&=(-1)^{\delta_{ij}}\theta(\epsilon_i,\epsilon_j)(\Theta(-\epsilon_j)\tau_j)\partial_i, \quad (i\in I, j\in I_1),\\
(\Theta(-\epsilon_i)\tau_i)\partial_j&=(-1)^{\delta_{ij}}\theta(\epsilon_i,\epsilon_j)\partial_j(\Theta(-\epsilon_i)\tau_i), \quad (i, j\in I_1).
\end{split}
\end{equation*}

Thus, (i) for $i\in I_0,\, j\in I_1$, we have
\begin{equation*}
\begin{split}
\Delta(\partial_i)\Delta(\partial_j)&=
(\partial_i\otimes\sigma_i^{-1}+\Theta(-\epsilon_i)\sigma_i\otimes\partial_i)(\partial_j\otimes 1+\Theta(-\epsilon_j)\tau_j\otimes\partial_j)\\
&=\partial_i\partial_j\otimes \sigma_i^{-1}+\Theta(-\epsilon_i{-}\epsilon_j)\sigma_i\tau_j\otimes \partial_i\partial_j\\
&\quad+(\Theta(-\epsilon_i)\sigma_i)\partial_j\otimes \partial_i+\partial_i(\Theta(-\epsilon_j)\tau_j)\otimes\sigma_i^{-1}\partial_j\\
&=\theta(\epsilon_i,\epsilon_j)\bigl(\partial_j\partial_i\otimes \sigma_i^{-1}+\Theta(-\epsilon_i{-}\epsilon_j)\tau_j\sigma_i\otimes \partial_j\partial_i\\
&\quad+\partial_j(\Theta(-\epsilon_i)\sigma_i)\otimes \partial_i+(\Theta(-\epsilon_j)\tau_j)\partial_i\otimes\partial_j\sigma_i^{-1}\bigr)\\
&=\theta(\epsilon_i,\epsilon_j)\Delta(\partial_j)\Delta(\partial_i);
\end{split}
\end{equation*}

(ii) for $i,\,j\in I_1$ but $i\ne j$, we have
\begin{equation*}
\begin{split}
\Delta(\partial_i)\Delta(\partial_j)&=
(\partial_i\otimes 1+\Theta(-\epsilon_i)\tau_i\otimes\partial_i)(\partial_j\otimes 1+\Theta(-\epsilon_j)\tau_j\otimes\partial_j)\\
&=\partial_i\partial_j\otimes 1+\Theta(-\epsilon_i{-}\epsilon_j)\tau_i\tau_j\otimes \partial_i\partial_j\\
&\quad+(\Theta(-\epsilon_i)\tau_i)\partial_j\otimes \partial_i+\partial_i(\Theta(-\epsilon_j)\tau_j)\otimes\partial_j\\
&=\theta(\epsilon_i,\epsilon_j)\bigl(\partial_j\partial_i\otimes 1+\Theta(-\epsilon_i{-}\epsilon_j)\tau_j\tau_i\otimes \partial_j\partial_i\\
&\quad+\partial_j(\Theta(-\epsilon_i)\tau_i)\otimes \partial_i+(\Theta(-\epsilon_j)\tau_j)\partial_i\otimes\partial_j\bigr)
\\
&=\theta(\epsilon_i,\epsilon_j)\Delta(\partial_j)\Delta(\partial_i).
\end{split}
\end{equation*}

For $j\in I_1$: it suffices to observe that $(\partial_j\otimes 1)(\Theta(-\epsilon_j)\tau_j\otimes \partial_j)=-(\Theta(-\epsilon_j)\tau_j\otimes \partial_j)(\partial_j\otimes 1)$,
which results in $\Delta(\partial_j)^2=0$.

In view of the fact just proved, together with (3.23) \& (3.24), we see that $(1\otimes \Delta)\Delta=(\Delta\otimes 1)\Delta$ and $(1\otimes \epsilon)\Delta=1=(\epsilon\otimes 1)\Delta$ hold. Again by (3.23) \& (3.25), we have
\begin{equation*}
\begin{split}
m\circ (1\otimes S)\circ \Delta(\partial_i)&=\partial_i\sigma_i+\Theta(-\epsilon_i)\sigma_i(-q\Theta(\epsilon_i)\partial_i)
=\partial_i\sigma_i-q\sigma_i\partial_i=0, \quad (i\in I_0)\\
m\circ (S\otimes 1)\circ \Delta(\partial_i)&=-q\Theta(\epsilon_i)\partial_i\sigma_i^{-1}+\sigma_i^{-1}\Theta(\epsilon_i)\partial_i
=0, \quad (i\in I_0)\\
m\circ (1\otimes S)\circ \Delta(\partial_j)&=\partial_j+\Theta(-\epsilon_j)\tau_j(-\Theta(\epsilon_j)\tau_j\partial_j)=0, \quad (j\in I_1)\\
m\circ (S\otimes 1)\circ \Delta(\partial_j)&=-\Theta(\epsilon_j)\tau_j\partial_j+\Theta(\epsilon_j)\tau_j\partial_j=0, \quad (j\in I_1)
\end{split}
\end{equation*}
thus we get $m\circ(S\otimes 1)\circ\Delta(\partial_i)=\epsilon(\partial_i)=m\circ(1\otimes S)\circ\Delta(\partial_i)$. On the other hand, owing to (3.23) \& (3.14)--(3.16), there holds $m\circ(S\otimes 1)\circ\Delta=\eta\circ\epsilon=m\circ(1\otimes S)\circ\Delta$, where $(\mathfrak D_q(m|n), m, \eta)$ is the algebra structure of $\mathfrak D_q(m|n)$.

Therefore, $(\mathfrak D_q(m|n), m, \eta, \Delta, \epsilon, S)$ is a (pointed) Hopf algebra.
\end{proof}

In the case when $\textbf{char}(q)=\ell$ $(\geq 3)$, if consider the quantum differential operators defined on the quantum restricted Grassmann superalgebra $\Omega_q(m|n,\bold 1)$, we see that $\partial_i^\ell\equiv 0$, for all $i\in I_0$.
Let us denote by $\mathfrak D_q(m|n,\bold 1)$ the {\it quantum restricted differential Hopf algebra}, which is defined to be the quotient of $\mathfrak D_q(m|n)$ by the Hopf ideal $\mathcal I$ generated by $\sigma_i^\ell-1$ (resp. $\sigma_i^{2\ell}-1$), $\Theta(\epsilon_i)^\ell-1$ (resp. $\Theta(\epsilon_i)^{2\ell}-1$) ($i\in I$), and $\partial_i^\ell$ ($i\in I_0$), when $\ell$ is odd (resp. even).
\begin{coro}
Assume that $\textbf{char}(q)=\ell\, (\geq 3)$ is odd (resp. even), $\mathfrak D_q(m|n,\bold 1)$ is a finite-dimensional pointed Hopf algebra associated with the bicharacter $\theta$ defined over $\mathbb Z_\ell^m\times \mathbb Z_2^n$ (resp. $\mathbb Z_{2\ell}^m\times \mathbb Z_2^n$).
\end{coro}
\begin{proof}
It suffices to notice that $(\partial_i\otimes \sigma_i^{-1})(\Theta(-\epsilon_i)\sigma_i\otimes \partial_i)=q^2(\Theta(-\epsilon_i)\sigma_i\otimes \partial_i)(\partial_i\otimes \sigma_i^{-1})$.
This implies that $\Delta(\partial_i)^\ell=\partial_i^\ell\otimes 1+1\otimes \partial_i^\ell$.
\end{proof}

\begin{remark} $(\text{\rm i})$ Actually, we can equip $\mathfrak D_q(m|n)$ with another Hopf algebra structure $(\mathfrak D_q(m|n), \Delta^{(-)}, \epsilon, S^{(-)})$, where the only differences are the actions of $\Delta^{(-)}$, $S^{(-)}$ on $\partial_i$ $(i\in I)$ given respectively by
\begin{gather*}
\Delta^{(-)}(\partial_i)=\partial_i{\otimes} \sigma_i+\Theta(-\epsilon_i)\sigma_i^{-1}{\otimes}\partial_i, \ (i\in I_0);\quad
\Delta^{(-)}(\partial_j)=\partial_j{\otimes} 1+\Theta(-\epsilon_j)\tau_j{\otimes}\partial_j, \ (j\in I_1), \\
S^{(-)}(\partial_i)=-q^{-1}\Theta(\epsilon_i)\partial_i,\ (i\in I_0); \quad S^{(-)}(\partial_j)=-\Theta(\epsilon_j)\tau_j\partial_j, \ (j\in I_1).
\end{gather*}

$(\text{\rm ii})$ The results here provide a method of constructing some (new) Hopf algebras arising from an investigation of introducing suitable quantum differential operators acting on some given quantum (super) vector spaces in advance. This is a continuation of the spirit of \cite{Hu}, which can be also
regarded as addressing the same questions posed earlier by Manin in his notes \cite{Ma1}.
Another marked point I would like to mention is that the Radford-Majid's bosonization of a superalgebra doesn't not necessarily follow the way of smash-extension only via one copy of $\Bbbk [\mathbb Z_2]$ $($cf. \cite{AAH}, \cite{Ma}$)$, our Theorem 11 and Corollary 12 are interesting examples via multi-copies $\Bbbk [\mathbb Z_2^n]$.  The next discussion shows once more that such bosonizations of sharing the same Nichols superalgebra are not unique. This seems a significant task to consider how to seek explicit minimal bosonizations of those Nichols superalgebras appeared in \cite{H}.
\end{remark}

\subsection{Bosonization of quantum affine $(m|n)$-superspace} By Proposition 7, $A_q^{m|n}\cong \mathcal D_q(m|n)$ as superalgebras, we can augment $A_q^{m|n}$ through adding a certain multiplication abelian group $\Gamma$ (as group-likes)
and construct another Hopf algebra $\mathfrak A_q(m|n)$ as a bosonization of $A_q^{m|n}$ such that it contains the quantum
affine $(m|n)$-superspace as its sub-superalgebra.
\begin{defi} Assume that $\mathfrak A_q(m|n)$ is an associative $\Bbbk$-algebra generated by the $A_q^{m|n}$, together with
the group-like elements $\mathcal K(\epsilon_i)$, where $\epsilon_i\in \Gamma=\mathbb Z^m\times\mathbb Z_2^n$, associated with
the bicharacter $\theta$ defined in $(3.10)$, satisfying the following relations:
\begin{gather}
\mathcal K(\epsilon_i)\mathcal K(\epsilon_j)=\mathcal K(\epsilon_i{+}\epsilon_j)=\mathcal K(\epsilon_j)\mathcal K(\epsilon_i),
\quad \mathcal K(\epsilon_i)^{\pm1}\mathcal K(\epsilon_i)^{\mp1}=\mathcal K(0)=1,\notag\\
\mathcal K(\epsilon_j)^2=1, \quad (j\in I_1),\notag\\
\mathcal K(\epsilon_i)^\ell=\mathcal K(\ell\epsilon_i)=1, \quad (\textit{for} \ i\in I_0, \textit{when} \ \textbf{ord}\,(q)=\ell),\notag\\
\mathcal K(\epsilon_i)x_j\mathcal K(\epsilon_i)^{-1}=q^{\delta_{ij}}\theta(\epsilon_i,\epsilon_j)\,x_j,\quad (j\in I_0),\\
\mathcal K(\epsilon_i)x_j\mathcal K(\epsilon_i)^{-1}=(-1)^{\delta_{ij}}\theta(\epsilon_i,\epsilon_j)\,x_j,\quad (j\in I_1),\notag\\
x_ix_j=\theta(\epsilon_i,\epsilon_j)\,x_jx_i, \quad (i, j\in I),\notag\\
x_j^2=0, \quad (j\in I_1).\notag
\end{gather}

Moreover, $\mathfrak A_q(m|n)$ can be equipped with the following mappings: $\Delta$, $\epsilon$,
$S$ on the generators of $\mathfrak A_q(m|n)$ as
\begin{gather}
\Delta: \mathfrak A_q(m|n)\longrightarrow
\mathfrak A_q(m|n)\otimes \mathfrak A_q(m|n)\\
\Delta(\mathcal K(\epsilon_i)^{\pm 1})=\mathcal K(\epsilon_i)^{\pm1}\otimes \mathcal K(\epsilon_i)^{\pm1}, \quad (i\in I),\notag\\
\Delta(x_i)=x_i\otimes 1+\mathcal K(\epsilon_i)\otimes x_i, \quad (i\in I)\notag\\
\epsilon: \mathfrak A_q(m|n)\longrightarrow \Bbbk\\
\epsilon(\mathcal K(\epsilon_i)^{\pm1})=1,\notag\\
\epsilon(x_i)=0,\notag\\
S: \mathfrak A_q(m|n)\longrightarrow \mathfrak A_q(m|n)\\
S(\mathcal K(\epsilon_i)^{\pm1})=\mathcal K(\epsilon_i)^{\mp1},\notag\\
S(x_i)=-\mathcal K(-\epsilon_i)x_i. \notag
\end{gather}
\end{defi}

We extend the definitions of $\Delta$, $\epsilon$ (resp. $S$) on $\mathfrak A_q(m|n)$ (anti-)algebraically.
The following quantum object in the case when $n=0$ is exactly Theorem 5.1 in \cite{Hu1}, namely, the structure of the quantized universal enveloping algebra of the abelian Lie algebra of dimension $m$.
\begin{theorem}
$\mathfrak A_q(m|n)$ is a pointed Hopf algebra with the group algebra $\Bbbk[\mathbb Z^m\times\mathbb Z_2^n]$ as its coradical,
associated with the bicharacter $\theta$ defined on $\mathbb Z^m\times\mathbb Z_2^n$ $(see\ (3.10))$,
which contains the quantum affine $(m|n)$-superspace $A_q^{m|n}$ as its Nichols superalgebra.
\end{theorem}
\begin{proof}
By analogy of argument of Theorem 11, one can prove that $\mathfrak A_q(m|n)$ is a Hopf algebra as desired.
\end{proof}

\begin{remark}
In Definition 14, when $\textbf{ord}\,(q)=\ell\, (\ge 3)$, i.e., $q^\ell=1$, we assumed $\mathcal K(\epsilon_i)^\ell=1$ for $i\in I_0$. This
defining condition is reasonable, which is based on the interpretation below. Noting that
the quasi-commutative relation: $(\mathcal K(\epsilon_i)\otimes x_i)(x_i\otimes 1)=q(x_i\otimes 1)(\mathcal K(\epsilon_i)\otimes x_i)$, we have
\begin{equation}
\Delta(x_i^p)=(\Delta(x_i))^p=\sum_{r=0}^p\binom{p}{r}_q x_i^{p-r}\mathcal K(\epsilon_i)^r\otimes x_i^r,
\end{equation}
where $\binom{p}{r}_q=\frac{(p)_q!}{(p-r)_q!(r)_q!}$, $(r)_q!=(r)_q(r-1)_q\cdots (1)_q$, and $(r)_q=\frac{q^r-1}{q-1}$. This, together with the assumption condition above, implies $\Delta(x_i^\ell)=x_i^\ell\otimes 1+1\otimes x_i^\ell$ for $i\in I_0$ when $\textbf{ord}\,(q)=\ell$. Clearly, $x_i^\ell\ (i\in I_0)$ are central in $\mathfrak A_q(m|n)$.
On the other hand, observing $q^{\ell(\ell+1)}=q^r=-1$ when $\ell=2r$, we have
\begin{equation*}
\begin{split}
S(x_i^\ell)&=(-1)^\ell\bigl(\mathcal K(\epsilon_i)^{-1}x_i\bigr)^\ell=(-1)^\ell q^{\frac{\ell(\ell+1)}2}\mathcal K(\epsilon_i)^{-\ell}x_i^\ell\\
&=\begin{cases} -x_i^\ell, & \ell\ \textit{odd},\\
q^{r(\ell+1)}x_i^\ell=-x_i^\ell, & \ell\ \textit{even}.
\end{cases}
\end{split}
\end{equation*}
So, this means that $\mathfrak A_q(m|n)$ contains the polynomial $($Hopf$)$ algebra $\Bbbk [x_1^\ell,\cdots,x_m^\ell]$ $($with the usual Hopf algebra structure$)$ as its central $($Hopf$)$ subalgebra inside.
\end{remark}

\subsection{Multi-rank Taft (Hopf) algebra of $(m|n)$-type}
In the case when $\textbf{ord}\,(q)=\ell$ $(>0)$, we have already gotten the construction/notion of the so-called $m$-rank Taft (Hopf) algebra implicitly in the case when $n=0$ (see Section 5 in \cite{Hu1}), namely, the terminology here ``$(m|0)$-type".
Denote $\mathcal {TH}_q(m|n):=\mathfrak A_q(m|n,\bold 1)$ by the quotient of
$\mathfrak A_q(m|n)$ by the Hopf ideal $\mathcal I_0$ generated by the central elements $x_i^\ell$ ($i\in I_0$), namely,
\begin{defi}
 Assume that $\textbf{ord}\,(q)=\ell$ $(>2)$, and $\mathcal {TH}_q(m|n)$ is an associative $\Bbbk$-algebra, which we call the multi-rank Taft algebra of $(m|n)$-type, generated by $x_i$ $(i\in I)$, together with
the group-like elements $\mathcal K(\epsilon_i)$, where $\epsilon_i\in \Gamma=\mathbb Z_\ell^m\times\mathbb Z_2^n$, associated with the bicharacter $\theta$ defined in $(3.10)$, satisfying the following relations:
\begin{gather*}
\mathcal K(\epsilon_i)\mathcal K(\epsilon_j)=\mathcal K(\epsilon_i{+}\epsilon_j)=\mathcal K(\epsilon_j)\mathcal K(\epsilon_i),
\qquad \mathcal K(\epsilon_i)^{\pm1}\mathcal K(\epsilon_i)^{\mp1}=\mathcal K(0)=1,\\
\mathcal K(\epsilon_i)^\ell=\mathcal K(\ell\epsilon_i)=1, \quad (i\in I_0),\qquad \mathcal K(\epsilon_j)^2=1, \quad (j\in I_1),\\
\mathcal K(\epsilon_i)x_j\mathcal K(\epsilon_i)^{-1}=q^{\delta_{ij}}\theta(\epsilon_i,\epsilon_j)\,x_j,\quad (j\in I_0),\\
\mathcal K(\epsilon_i)x_j\mathcal K(\epsilon_i)^{-1}=(-1)^{\delta_{ij}}\theta(\epsilon_i,\epsilon_j)\,x_j,\quad (j\in I_1),\\
x_ix_j=\theta(\epsilon_i,\epsilon_j)\,x_jx_i, \quad (i, j\in I),\\
x_i^\ell=0, \quad (i\in I_0), \qquad x_j^2=0, \quad (j\in I_1).
\end{gather*}
\end{defi}

\begin{coro}
Assume that $\textbf{ord}\,(q)=\ell$ $(>2)$, then $(\mathcal {TH}_q(m|n), \Delta, \epsilon, S)$ with formulae $(3.27)$--$(3.29)$ forms an $(\ell^m\cdot 2^n)^2$-dimensional pointed Hopf algebra with
$\Bbbk [\mathbb Z_\ell^m\times\mathbb Z_2^n]$ as its coradical. In particular, $\mathcal {TH}_q(1|0)$ is the usual Taft algebra, and $\mathcal {TH}_q(m|0)$ is the $m$-rank Taft algebra in \cite{Hu1}.
\end{coro}

\subsection{Multi-rank Taft (Hopf) algebra of \underline{$\ell$}-type}
Based on the above results (e.g. Corollary 18), we can generalize the multi-rank Taft algebra $\mathcal{TH}_q(m|n)$ of $(m|n)$-type to the most general case. Suppose a matrix $\mu:=(\mu_{ij})_{n\times n}$ with $\mu_{ij}\in\Bbbk^*$ satisfies:
\begin{equation}
\mu_{ij} \mu_{ji}=1, \quad (1\le i\ne j\le n);  \qquad \textbf{ord}(\mu_{ii})=\ell_i\in \mathbb N, \ i.e., \ \mu_{ii}^{\ell_i}=1, \quad
 (1\le i\le n).
\end{equation}
 Set $\underline{\ell}:=(\ell_1,\cdots,\ell_n)\in \mathbb N^n$, and denote by $\Gamma:=\mathbb Z_{\ell_1}\times\cdots\times\mathbb Z_{\ell_n}$ any finite abelian group.
Obviously, the matrix $\widehat\mu=\mu-\text{diag}\{\mu_{11}-1,\cdots,\mu_{nn}-1\}$ defines a bicharacter $\theta$ on $\Gamma$.
\begin{coro}
The multi-rank Taft algebra $\mathcal {TH}_\mu(\underline\ell)$ of $\underline{\ell}$-type is an associative $\Bbbk$-algebra generated by $x_i$ and $\mathcal K_i$ $(1\le i\le n)$, associated with the matrix $\mu$ in $(3.31)$, subject to relations below
\begin{gather*}
\mathcal K_i\mathcal K_j=\mathcal K_j\mathcal K_i,
\quad \mathcal K_i^{\pm1}\mathcal K_i^{\mp1}=1, \quad \mathcal K_i^{\ell_i}=1,\\
\mathcal K_ix_j\mathcal K_i^{-1}=\mu_{ij}\,x_j,\\
x_ix_j=\mu_{ij}\,x_jx_i, \quad (i\ne j)\\
x_i^{\ell_i}=0.
\end{gather*}

\noindent
Then $\mathcal {TH}_\mu(\underline\ell)$ is an $(\ell_1\cdots \ell_n)^2$-dimensional pointed Hopf algebra with the group algebra $\Bbbk [\Gamma]$ as its coradical, where its comultiplication $\Delta$, counit $\epsilon$, and antipode $S$ are given by
\begin{gather*}
\Delta: \mathcal {TH}_\mu(\underline{\ell})\longrightarrow
\mathcal {TH}_\mu(\underline{\ell})\otimes \mathcal {TH}_\mu(\underline{\ell})\\
\Delta(\mathcal K_i^{\pm 1})=\mathcal K_i^{\pm1}\otimes \mathcal K_i^{\pm1}, \\
\Delta(x_i)=x_i\otimes 1+\mathcal K_i\otimes x_i, \\
\epsilon: \mathcal {TH}_\mu(\underline{\ell})\longrightarrow \Bbbk\notag\\
\epsilon(\mathcal K_i^{\pm1})=1,\qquad
\epsilon(x_i)=0,\notag\\
S: \mathcal {TH}_\mu(\underline{\ell})\longrightarrow \mathcal {TH}_\mu(\underline{\ell})\\
S(\mathcal K_i^{\pm1})=\mathcal K_i^{\mp1},\qquad
S(x_i)=-\mathcal K_i^{-1}x_i.
\end{gather*}
\end{coro}
\begin{proof}
Replacing $q$ by $\mu_{ii}$ in (3.30), we get that $\Delta$ preserves relation: $x_i^{\ell_i}=0$, as well as
$x_ix_j=\mu_{ij}\,x_jx_i$. These are crucial points in the proof as that of Theorem 11.
\end{proof}
\begin{remark}
Replacing the assumption condition $\mathcal K_i^{\ell_i}=1$ $(1\le i\le n)$ in the above Theorem by $\mathcal K_i^{m_i}=1$ with $\ell_i\,|\,m_i$, $\underline \ell\ne \underline m$ $($where $m_i\in\mathbb N$$)$, we obtain the {\bf generalized multi-rank Taft algebra} $\mathcal {TH}_\mu(\underline \ell\,|\,\underline m)$, which is a pointed Hopf algebra of dimension $\Pi_{i=1}^n (\ell_im_i)$. In particular, when $n=1$, this is the generalized Taft algebra.

Notice that for the generalized multi-rank Taft algebras $\mathcal {TH}_\mu(\underline\ell\,|\,\underline m)$, By Andruskiewitsch-Schneider's lifting observation \cite{AS}, we have the well-known liftings $\mathcal {TH}_\mu^{\lambda,\nu}(\underline\ell\,|\,\underline m)$ for any family of parameters $\lambda=(\lambda_i)$, $\nu=(\nu_{ij})$ with $\lambda_i, \nu_{ij}\in\mathbb C$, which are pointed Hopf algebras satisfying the relations
$x_ix_j=\mu_{ij}x_jx_i+\nu_{ij}(1-\mathcal K_i\mathcal K_j)$,
$x_i^{\ell_i}=\lambda_i(1-\mathcal K_i^{\ell_i})$ and $\mathcal K_i^{m_i}=1$ for each $i$ with $\ell_i\,|\,m_i$ and at least one $i_0$ such that $\ell_{i_0}<m_{i_0}$. $\lambda=(0)$ and $\nu=(0)$ corresponds to the trivial lifting, $\mathcal{TH}_\mu(\underline \ell\,|\,\underline m)$.
\end{remark}

\subsection{Bosonization of quantum Grassmann superalgebra} Recall the quantum Grassmann superalgebra $\Omega_q(m|n)=\mathcal A_q(m)\otimes_{\Bbbk}\Lambda_{q^{-1}}(n)$, as vector superspaces, where $\mathcal A_q(m)$ has a quantum divided power algebra structure described in (2.5). The assertion of Proposition 7 indicates that $\Omega_q(m|n)$ is indeed a dual object of the quantum affine $(m|n)$-superspace $A_q^{(m|n)}$. In order to get a suitable bosonization of $\Omega_q(m|n)$ compatible with its quantum divided power algebra structure, we cannot direct adopt Definition 14 but need to revise defining relation (3.26) into (3.33).
Observe that when $\textbf{char}(q)=\ell$ $(\ge 3)$ the quantum divided power algebra $\mathcal A_q(m)$ is generated by elements
$x_i$  and $x_i^{(\ell)}$ $(i\in I_0)$ (see Proposition 2.4 in \cite{Hu}), otherwise, its generators are just $x_i$ $(i\in I_0)$. Hence we have the following
\begin{proposition} Assume that $\mathfrak G_q(m|n)$ is an associative $\Bbbk$-algebra generated by the quantum Grassmann
superalgebra $\Omega_q(m|n)$, as well as group-likes $\mathcal K_i\ (i\in I)$, associated with the bicharacter $\theta$ in $(3.10)$ defined on $\mathbb Z^m\times\mathbb Z_2^n$, subject to the relations below
\begin{gather}
\mathcal K_i\mathcal K_j=\mathcal K_j\mathcal K_i,
\qquad \mathcal K_i^{\pm1}\mathcal K_i^{\mp1}=1, \quad (i, j\in I);\qquad \mathcal K_j^2=1, \quad (j\in I_1)\notag\\
\mathcal K_i^\ell=1,\quad x_i^{(\ell)} \ \textit{are central}, \ (i\in I_0),\quad (\textit{when }\ \textbf{char}(q)=\ell\ \textit{is odd})\\
\mathcal K_ix_j\mathcal K_i^{-1}=q^{2\delta_{ij}}\theta(\epsilon_i,\epsilon_j)\,x_j,\quad (j\in I_0)\\
\mathcal K_ix_j\mathcal K_i^{-1}=(-1)^{\delta_{ij}}\theta(\epsilon_i,\epsilon_j)\,x_j,\quad (j\in I_1)\notag\\
x_ix_j=\theta(\epsilon_i,\epsilon_j)\,x_jx_i, \quad (i, j\in I)\notag\\
x_j^2=0, \quad (j\in I_1).\notag\\
\Delta: \mathfrak G_q(m|n)\longrightarrow
\mathfrak G_q(m|n)\otimes \mathfrak G_q(m|n)\notag\\
\Delta(\mathcal K_i^{\pm 1})=\mathcal K_i^{\pm1}\otimes \mathcal K_i^{\pm1}, \quad (i\in I)\notag\\
\Delta(x_i)=x_i\otimes 1+\mathcal K_i\otimes x_i, \quad (i\in I)\notag\\
\Delta(x_i^{(\ell)})=x_i^{(\ell)}\otimes 1+1\otimes x_i^{(\ell)}, \quad (i\in I_0)\\
\epsilon: \mathfrak G_q(m|n)\longrightarrow \Bbbk\notag\\
\epsilon(\mathcal K_i^{\pm1})=1,\qquad
\epsilon(x_i)=0,\notag\\
S: \mathfrak G_q(m|n)\longrightarrow \mathfrak G_q(m|n)\notag\\
S(\mathcal K_i^{\pm1})=\mathcal K_i^{\mp1},\notag\\
S(x_i)=-\mathcal K_i^{-1}x_i,\quad (i\in I),\notag \\
S(x_i^{(\ell)})=-x_i^{(\ell)}, \quad (i\in I_0). \notag
\end{gather}
Then $\mathfrak G_q(m|n)$ is a pointed Hopf algebra and contains $\Omega_q(m|n)$ as a Nichols subsuperalgebra,
as well as $\Bbbk [x_1^{(\ell)},\cdots,x_m^{(\ell)}]$ as usual Hopf polynomial central subalgebra.
\end{proposition}
\begin{proof}
From the revised relation (3.33), we see that $(\mathcal K_i\otimes x_i)(x_i\otimes 1)=q^2(x_i\otimes 1)(\mathcal K_i\otimes x_i)$, and furthermore leads to
$$\Delta(x_i^p)=(\Delta(x_i))^p=\sum_{r=0}^p{\,p\,\brack \,r\,}_q q^{\binom{p}2-\binom{r}2-\binom{p-r}2}x_i^{p-r}\mathcal K_i^r\otimes x_i^r,$$
so that we get $\Delta(x_i^{(p)})=\sum_{r=0}^pq^{\binom{p}2-\binom{r}2-\binom{p-r}2}x_i^{(p-r)}\mathcal K_i^r\otimes x_i^{(r)}$ for $p<\ell$, and  $\Delta(x_i^{(\ell)})=x_i^{(\ell)}\otimes 1+1\otimes x_i^{(\ell)}$, under the assumption of  $\textbf{char}(q)=\ell$ and $\mathcal K_i^\ell=1$.

Based on the above observation and by analogy of argument of Theorem 11, we can check the assertion holds.
\end{proof}

Furthermore, for the quantum restricted Grassmann subsuperalgebra $\Omega_q(m|n,\bold 1)$, we can consider its
bosonization. According to Proposition 21, we easily obtain the following
\begin{coro} Suppose $\textbf{char}(q)=\ell$ is odd. The bosonization of $\Omega_q(m|n,\bold 1)$ is an $(\ell^m\cdot 2^n)^2$-dimensional pointed Hopf algebra $\mathfrak G_q(m|n,\bold 1)$, which contains $\Omega_q(m|n,\bold 1)$ as its Nichols
subsuperalgebra and has the coradical $\Bbbk [\mathbb Z_\ell^m\times\mathbb Z_2^n]$ associated with the bicharacter $\theta$
defined over $\mathbb Z_\ell^m\times\mathbb Z_2^n$ $($see $(3.10))$, satisfying the relations below
\begin{gather*}
\mathcal K_i\mathcal K_j=\mathcal K_j\mathcal K_i,
\qquad \mathcal K_i^{\pm1}\mathcal K_i^{\mp1}=1,\\
\mathcal K_i^\ell=1, \quad (i\in I_0),\qquad \mathcal K_j^2=1, \quad (j\in I_1),\\
\mathcal K_ix_j\mathcal K_i^{-1}=q^{2\delta_{ij}}\theta(\epsilon_i,\epsilon_j)\,x_j,\quad (j\in I_0),\\
\mathcal K_ix_j\mathcal K_i^{-1}=(-1)^{\delta_{ij}}\theta(\epsilon_i,\epsilon_j)\,x_j,\quad (j\in I_1),\\
x_ix_j=\theta(\epsilon_i,\epsilon_j)\,x_jx_i, \quad (i, j\in I),\\
x_i^\ell=0, \quad (i\in I_0), \qquad x_j^2=0, \quad (j\in I_1).
\end{gather*}
and with the same comultiplication $\Delta$, counit $\epsilon$ and antipode $S$ as in $(3.27)$--$(3.29)$.
\end{coro}

\begin{remark}
For simplifying our discussion, both in Proposition 21 and Corollary 22, we assumed $\textbf{char}(q)=\ell$ is odd, otherwise, the relations would appear more complicate a bit $($see below, $(3.50)$--$(3,51)$, $(3.54)$$)$. In fact, in Proposition 21, we can cancel the defining condition $(3.32)$ whatever $\textbf{char}(q)=\ell$ is. On the other hand, we note that
two pointed Hopf algebras $\mathcal{TH}_q(m|n)$ and $\mathfrak G_q(m|n,\bold 1)$ in Corollaries 18 $\&$ 22 have the same dimension and almost the same relations except for the difference between $(3.26)$ and $(3.33)$.
\end{remark}

\subsection{Quantum Weyl algebra $\mathcal W_q(2(m|n))$ of $(m|n)$-type}
Now we return to Lemma 9 \& Theorem 11, Corollary 12. Recall that $\mathfrak D_q(m|n)$  (resp. $\mathfrak D_q(m|n, \bold 1)$) is the quantum (reps. restricted) differential Hopf algebra
 appearing in subsection 3.5. Lemma 9 says that the quantum (resp. restricted) Grassmann superalgebra $\Omega_q(m|n)$ (resp. $\Omega_q(m|n, \bold 1)$) is a $\mathfrak D_q(m|n)$-module superalgebra (resp. $\mathfrak D_q(m|n, \bold 1)$-module superalgebra). So we can make their smash product algebra $\Omega_q(m|n)\# \mathfrak D_q(m|n)$ (resp.
$\Omega_q(m|n, \bold 1)\# \mathfrak D_q(m|n, \bold 1)$) in a familiar fashion as in \cite{Mon}, \cite{Rad}, \cite{SW},
which is the same as $\Omega_q(m|n)\otimes \mathfrak D_q(m|n)$ (resp. $\Omega_q(m|n, \bold 1)\otimes \mathfrak D_q(m|n, \bold 1)$) as vector spaces but with the multiplication given by
\begin{gather}
\bigl((x^{(\alpha)}\otimes x^\mu)\#\partial_i\bigr)\circ\bigl((x^{(\beta)}\otimes x^\nu)\#d\bigr)=(x^{(\alpha)}\otimes x^\mu)\partial_i(x^{(\beta)}\otimes x^\nu)\#\sigma_i^{-1}d\\
+(x^{(\alpha)}\otimes x^\mu)(\Theta(-\epsilon_i)\sigma_i)(x^{(\beta)}\otimes x^\nu)\#\partial_id, \quad (i\in I_0)\notag\\
((x^{(\alpha)}\otimes x^\mu)\#\partial_j)\circ((x^{(\beta)}\otimes x^\nu)\#d)=(x^{(\alpha)}\otimes x^\mu)\partial_j(x^{(\beta)}\otimes x^\nu)\#d\\
+(x^{(\alpha)}\otimes x^\mu)(\Theta(-\epsilon_j)\tau_j)(x^{(\beta)}\otimes x^\nu)\#\partial_jd, \quad (j\in I_1)\notag\\
((x^{(\alpha)}\otimes x^\mu)\#g)\circ((x^{(\beta)}\otimes x^\nu)\#d)=(x^{(\alpha)}\otimes x^\mu)g(x^{(\beta)}\otimes x^\nu)\#gd,
\end{gather}
where $\Delta(\partial_i)=\partial_i\otimes\sigma_i^{-1}+\Theta(-\epsilon_i)\sigma_i\otimes\partial_i$ ($i\in I_0$), and $\Delta(\partial_j)=\partial_j\otimes 1+\Theta(-\epsilon_j)\tau_j\otimes\partial_j$ ($j\in I_1$), and
$\Delta(g)=g\otimes g$, for $x^{(\alpha)}\otimes x^\mu, x^{(\beta)}\otimes x^\nu\in\Omega_q$, and
$\partial_i, g, d\in\mathfrak D_q$, and $\Omega_q=\Omega_q(m|n)$ or $\Omega_q(m|n, \bold 1)$, $\mathfrak D_q=\mathfrak D_q(m|n)$ or $\mathfrak D_q(m|n, \bold 1)$.

Note that $\Omega_q(m|n)$ has generators $x_i$ $(i\in I)$ when $\textbf{char}(q)=0$, so is for $\Omega_q(m|n, \bold 1)$ when
$\textbf{char}(q)=\ell>0$; while $\Omega_q(m|n)$ has generators $x_i$ $(i\in I)$ and $x_j^{(\ell)}$ $(j\in I_0)$ when $\textbf{char}(q)=\ell>0$. On the other hand, both $\mathfrak D_q(m|n)$ and its truncated object $\mathfrak D_q(m|n, \bold 1)$ have the same generators $\partial_i$, $\Theta(\epsilon_i)$, $\sigma_i$ $(i\in I)$, and $\tau_j$ $(j\in I_1)$ whatever $\textbf{char}(q)$ is.

More precisely, we have the cross relations between generators
\begin{equation}
\begin{split}
((x^{(\alpha)}\otimes x^\mu)\#1)\circ(1\# d)&=(x^{(\alpha)}\otimes x^\mu)\#d,\\
(1\#\Theta(\epsilon_i))\circ(x_j\#1)&=\theta(\epsilon_i,\epsilon_j)x_j\#\Theta(\epsilon_i),\quad (i\in I)\\
(1\#\sigma_i)\circ(x_j\#1)&=q^{\delta_{ij}}x_j\#\sigma_i,\quad (i\in I)\\
(1\#\tau_i)\circ(x_j\#1)&=(-1)^{\delta_{ij}}x_j\#\tau_i, \quad (i\in I_1)\\
(1\#\partial_i)\circ(x_j\#1)&=\delta_{ij}\#\sigma_i^{-1}+\theta(\epsilon_j,\epsilon_i)q^{\delta_{ij}}x_j\#\partial_i,\quad (i\in I_0)\\
(1\#\partial_i)\circ(x_j\#1)&=\delta_{ij}\#1+\theta(\epsilon_j,\epsilon_i)(-1)^{\delta_{ij}}x_j\#\partial_i,\quad (i\in I_1)\\
(1\#\Theta(\epsilon_i))\circ(x_j^{(\ell)}\#1)&=\theta(\epsilon_i,\ell\epsilon_j)x_j^{(\ell)}\#\Theta(\epsilon_i),\quad (i\in I)\\
(1\#\sigma_i)\circ(x_j^{(\ell)}\#1)&=q^{\ell\delta_{ij}}x_j^{(\ell)}\#\sigma_i,\quad (i\in I)\\
(1\#\tau_i)\circ(x_j^{(\ell)}\#1)&=(-1)^{\ell\delta_{ij}}x_j^{(\ell)}\#\tau_i, \quad (i\in I_1)\\
(1\#\partial_i)\circ(x_j^{(\ell)}\#1)&=\delta_{ij}x_j^{(\ell-1)}\#\sigma_i^{-1}
+\theta(\ell\epsilon_j,\epsilon_i)q^{\ell\delta_{ij}}x_j^{(\ell)}\#\partial_i,\quad (i\in I_0)\\
(1\#\partial_i)\circ(x_j^{(\ell)}\#1)&=\delta_{ij}x_j^{(\ell-1)}\#1
+\theta(\ell\epsilon_j,\epsilon_i)(-1)^{\delta_{ij}\ell}x_j^{(\ell)}\#\partial_i,\quad (i\in I_1)
\end{split}
\end{equation}

Recall that $\textbf{char}(q)=\ell$ implies two cases: (i) $\ell$ is odd, and $\textbf{ord}(q)=\ell$; (ii) $\ell=2r$, and
$\textbf{ord}(q)=2\ell$, i.e., $q^\ell=-1$. Now identify elements $(x^{(\alpha)}\otimes x^\mu)\#d$ in $\Omega_q\#\mathfrak D_q$
with $(x^{(\alpha)}\otimes x^\mu)d$, then the smash product algebra $\Omega_q\#\mathfrak D_q$ containing
$\Omega_q$ and $\mathfrak D_q$ as subalgebras are just the {\it quantum differential operator algebras} $\textrm{Diff}_q(\Omega_q)$ over the quantum (restricted) Grassmann superalgebra $\Omega_q$, which will degenerate into
the usual differential operator algebra when $q$ takes $1$. Therefore, we arrive at the following
\begin{defi}
The quantum $($restricted$)$ Weyl algebra $\mathcal W_q(2(m|n))$ $($$\mathcal W_q(2(m|n), \bold 1)$$)$ of $(m|n)$-type is defined to be the quantum differential operator algebra $\textrm{Diff}_q(\Omega_q)$, namely, the smash product algebra $\Omega_q\#\mathfrak D_q$, which is an associative $\Bbbk$-algebra generated by the symbols $\Theta(\pm\epsilon_i)$,
$\sigma_i^{\pm1}$ $(i\in I)$, $\tau_j$ $(j\in I_1)$, $\partial_i$, $x_i$ $(i\in I)$, in addition,
$x_i^{(\ell)}$ $(i\in I_0)$ when $\textbf{char}(q)=\ell>0$, obeying the following cross relations, besides those relations in $\Omega_q$ and $\mathfrak D_q$:
\begin{gather}
\Theta(\epsilon_i)\circ x_j\circ\Theta(-\epsilon_i)=\theta(\epsilon_i, \epsilon_j)x_j,\\
\sigma_i\circ x_j\circ \sigma_i^{-1}=q^{\delta_{ij}}x_j, \qquad \tau_i\circ x_j\circ \tau_i=(-1)^{\delta_{ij}}x_j,\\
\partial_i\circ x_i-qx_i\circ\partial_i=\sigma_i^{-1},\quad (i\in I_0),
\end{gather}
\begin{gather}
\partial_i\circ x_i+x_i\circ\partial_i=1,\quad (i\in I_1),\\
\partial_i\circ x_j=\theta(\epsilon_j,\epsilon_i)x_j\circ\partial_i, \quad (i\ne j\in I),
\end{gather}
$(${\rm i}$)$ When $\ell$ is odd, i.e., $q^\ell=1$, in addition, we have:
\begin{gather}
\Theta(\epsilon_i)\circ x_j^{(\ell)}=x_j^{(\ell)}\circ\Theta(\epsilon_i),\quad (j\in I_0),\\
\sigma_i\circ x_j^{(\ell)}=x_j^{(\ell)}\circ \sigma_i, \quad (j\in I_0),\\
\tau_i\circ x_j^{(\ell)}=x_j^{(\ell)}\circ\tau_i, \quad (i\in I_1, j\in I_0),
\\
\partial_i\circ x_i^{(\ell)}-x_i^{(\ell)}\circ\partial_i=x_i^{(\ell-1)}\circ\sigma_i^{-1}, \quad (i\in I_0),\\
\partial_i\circ x_j^{(\ell)}=x_j^{(\ell)}\circ \partial_i, \quad (i\in I, j\in I_0, i\ne j),\\
x_i\circ x_j^{(\ell)}=x_j^{(\ell)}\circ x_i, \quad (i\in I, j\in I_0),
\end{gather}
$(${\rm ii}$)$ When $\ell$ is even, i.e., $q^\ell=-1$, in addition, we have:
\begin{gather}
\Theta(\epsilon_i)x_i^{(\ell)}=x_i^{(\ell)}\Theta(\epsilon_i),\qquad \sigma_ix_i^{(\ell)}=-x_i^{(\ell)}\sigma_i,\\
\Theta(\epsilon_i)x_j^{(\ell)}=-x_j^{(\ell)}\Theta(\epsilon_i),\qquad \sigma_ix_j^{(\ell)}=x_j^{(\ell)}\sigma_i,\qquad \tau_ix_j^{(\ell)}=x_j^{(\ell)}\tau_i,\quad (i\ne j),\\
\partial_i\circ x_i^{(\ell)}+x_i^{(\ell)}\circ\partial_i=x_i^{(\ell-1)}\circ\sigma_i^{-1},\\
\partial_i\circ x_j^{(\ell)}=-x_j^{(\ell)}\circ\partial_i,\quad (i\in I, j\in I_0, i\ne j),\\
x_ix_j^{(\ell)}=-(-1)^{\delta_{ij}}x_j^{(\ell)}x_i,\quad (i\in I, j\in I_0).
\end{gather}
\end{defi}

In what follows, the quantum Weyl algebra of $(m|n)$-type we constructed here will serves as an important framework
in which it allows us to realize the bosonization object of the quantum general linear superalgebra $U_q(\mathfrak{gl}(m|n))$
in terms of certain suitable quantum differential operators in $\textrm{Diff}_q(\Omega_q)$.

\section{The $\mathcal U_q$-module algebra structure over $\Omega_q$ and its simple modules}

\subsection{Bosonization $\mathcal U_q(\mathfrak{gl}(m|n))$ of quantum superalgebra $U_q(\mathfrak{gl}(m|n))$}
The quantum general linear superalgebra $U_q(\mathfrak{gl}(m|n))$ has been introduced in subsection 2.3. In order to define the Hopf algebra structure, we introduce the parity operator $\sigma$ on $U_q(\mathfrak{gl}(m|n))$, which is defined by
$\sigma(E_i)=(-1)^{p(E_i)}E_i$, $\sigma(F_i)=(-1)^{p(F_i)}F_i$, and $\sigma(K_i)=K_i$, for all $i\in I$. Clearly, such $\sigma$ defines an automorphism of $U_q(\mathfrak{gl}(m|n))$ of order $2$. Then we have the following
\begin{defi}
The bosonization of $U_q(\mathfrak{gl}(m|n))$ is defined to be the smash product algebra $\mathcal U_q(\mathfrak {gl}(m|n))=U_q(\mathfrak{gl}(m|n))\#\Bbbk [\mathbb Z_2]:=U_q(\mathfrak{gl}(m|n))\oplus U_q(\mathfrak{gl}(m|n))\sigma$,
which is an associative $\Bbbk$-algebra generated by those $x\in U_q(\mathfrak{gl}(m|n))$, and the parity element $\sigma$
with multiplication given by $\sigma^2=1$ and $\sigma\, x\,\sigma=\sigma(x)$ for any $x\in U_q(\mathfrak{gl}(m|n))$.
Now $\mathcal U_q(\mathfrak{gl}(m|n))$ has a Hopf algebra structure whose comultiplication is the algebra homomorphism $\Delta:
\mathcal U_q(\mathfrak{gl}(m|n))\longrightarrow \mathcal U_q(\mathfrak{gl}(m|n))\otimes \mathcal U_q(\mathfrak{gl}(m|n))$
specified by
\begin{gather}
\Delta(\sigma)=\sigma\otimes\sigma, \qquad \Delta(K_i)=K_i\otimes K_i,\\
\Delta(E_i)=E_i\otimes \mathcal K_i+\sigma^{p(i)}\otimes E_i,\\
\Delta(F_i)=F_i\otimes 1+\sigma^{p(i)}\mathcal K_i^{-1}\otimes F_i,
\end{gather}
where $\mathcal K_i=K_iK_{i+1}^{-1}$ for $i<m$, $\mathcal K_m=K_mK_{m+1}$, and $\mathcal K_i=K_i^{-1}K_{i+1}$ for $i>m$.
The antipode $S$ is therefore given by
\begin{equation}
\begin{split}
S(\sigma)=\sigma, &\qquad
S(K_i)=K_i^{-1},\\
S(E_i)&=-\sigma^{p(i)}E_i\mathcal K_i^{-1},\\
S(F_i)&=-\sigma^{p(i)}\mathcal K_iF_i,
\end{split}
\end{equation}
and the counit by
\begin{gather}
\epsilon(\sigma)=1=\epsilon(K_i)=\epsilon(\mathcal K_i),\\
\epsilon(E_i)=0=\epsilon(F_i).
\end{gather}
\end{defi}

\begin{remark}
Similarly, we can define the bosonization $\mathcal U_q(\mathfrak{sl}(m|n))$ of the quantum special linear superalgebra
$U_q(\mathfrak{sl}(m|n))$, as well as the bosonization $u_q(\mathfrak{gl}(m|n))$, $u_q(\mathfrak{sl}(m|n))$
of the quantum restricted superalgebra $\mathfrak u_q(\mathfrak{gl}(m|n))$, $\mathfrak u_q(\mathfrak{sl}(m|n))$, respectively, in the case when $\textbf{char}(q)=
\ell>0$.
\end{remark}

In what follows, for convenience of our statement, we shall adopt the symbol $\mathcal U_q$ to denote $\mathcal U_q(\mathfrak{gl}(m|n))$,  $\mathcal U_q(\mathfrak{sl}(m|n))$ respectively, or the restricted object $u_q(\mathfrak{gl}(m|n))$, $u_q(\mathfrak{sl}(m|n))$ respectively, when $\textbf{char}(q)=\ell>0$. Correspondingly, the symbol $\Omega_q$ denotes $\Omega_q(m|n)$, or  $\Omega_q(m|n, \bold 1)$,
when $\textbf{char}(q)=\ell>0$. The purpose introducing these bosonization objects here is based on the following consideration.

For quantum superalgebra (as Hopf superalgebra) $A:=U_q$ or $\mathfrak u_q$, denote by $_A \textbf{smod}$ the category of $A$-supermodules; let $\mathcal A:=\mathcal U_q$ or $u_q$ denote the bosonization (as Hopf algebra) of $A$, and
$_{\mathcal A}\textbf{mod}$ the category of $\mathcal A$-modules. Then we are able to embed any $A$-supermodule object
 of $_A \textbf{smod}$ into an object of $_{\mathcal A}\textbf{mod}$, and reciprocally, we view any object of $_{\mathcal A}\textbf{mod}$ with the grading defined by the action of parity element $\sigma$, via restriction to $A$, as an $A$-supermodule in $_A\textbf{smod}$ (cf. \cite{Hu}, \cite{Ma}, \cite{AAH}). We prefer to treat the representation theory of $A$ instead in the category $_{\mathcal A}\textbf{mod}$.

\subsection{$\Omega_q$ as $\mathcal U_q$-module algebra structure via quantum Weyl algebra}

In subsection 3.10, we have defined the quantum Weyl algebra $\mathcal W_q(2(m|n))$ of $(m|n)$-type as the quantum
differential operator algebra $\textrm{Diff}_q(\Omega_q(m|n))$ over the quantum Grassmann superalgebra $\Omega_q(m|n)$.
We expect the quantum algebra $\mathcal U_q(\mathfrak{gl}(m|n))$ or $\mathcal U_q(\mathfrak{sl}(m|n))$ can be realized
as a subquotient object of $\mathcal W_q(2(m|n))$, just as in the Lie superalgebra level $\mathfrak {gl}(m|n)$ or $\mathfrak
{sl}(m|n)$ can be realized via certain suitable differential operators.
This program which we think well addressed Manin's questions (\cite{Ma1}) can be achieved in the type $A(m|n)$ case as follows.


\begin{theorem}\label{th5}
For any $x^{(\alpha)}\otimes x^\mu\in\Omega_q:=\Omega_q(m|n)$, or $\Omega_q(m|n,\bold 1)$ when $\textbf{char}(q)=\ell>2$, set
 \begin{gather}
E_j.(x^{(\alpha)}\otimes x^\mu)= (x_j\partial_{j+1}\sigma_{j})(x^{(\alpha)}\otimes x^\mu), \quad (j\in J)\label{eq6}\\
F_j.(x^{(\alpha)}\otimes x^\mu)= (\sigma_j^{-1}x_{j+1}\partial_j)(x^{(\alpha)}\otimes x^\mu),\quad (j\in J)\label{eq7}\\
K_i.(x^{(\alpha)}\otimes x^\mu)=\sigma_i(x^{(\alpha)}\otimes x^\mu),  \quad (i\in I_0)\\
 K_i.(x^{(\alpha)}\otimes x^\mu)=(\sigma_i\tau_i)^{-1}(x^{(\alpha)}\otimes x^\mu), \quad (i\in I_1)\\
 \mathcal K_i.(x^{(\alpha)}\otimes x^\mu)=\sigma_i\sigma_{i+1}^{-1}(x^{(\alpha)}\otimes x^\mu), \quad (m\ne i\in I_0)\\
\mathcal K_m.(x^{(\alpha)}\otimes x^\mu)=\sigma_m\sigma_{m+1}\tau_{m+1}(x^{(\alpha)}\otimes x^\mu),\\
\mathcal K_i.(x^{(\alpha)}\otimes x^\mu)=\sigma_i^{-1}\sigma_{i+1}\tau_i\tau_{i+1}(x^{(\alpha)}\otimes x^\mu), \quad (i\in I_1)\label{eq8}\\
\sigma(x^{(\alpha)}\otimes x^\mu)=\tau(x^{(\alpha)}\otimes x^\mu)=(-1)^{|\mu|}x^{(\alpha)}\otimes x^\mu.\label{eq9}
\end{gather}
Formulae (4.7)--- (4.14) define a $\mathcal U_q$-module algebra structure over the quantum $($restricted$)$ Grassmann superalgebra $\Omega_q$, where $\mathcal U_q:=\mathcal U_q(\mathfrak{gl}(m|n))$,  $\mathcal U_q(\mathfrak{sl}(m|n))$ respectively, or the restricted object $u_q(\mathfrak{gl}(m|n))$, $u_q(\mathfrak{sl}(m|n))$ respectively, when $\textbf{char}(q)=\ell>2$.
\end{theorem}

\begin{proof} All the proof will be carried out for $\mathcal U_q(\mathfrak{gl}(m|n))$ or $\mathcal U_q(\mathfrak{sl}(m|n))$, the others are similar.

(I) First of all, we need to check the multiplication of $\Omega_q$ is a $\mathcal U_q$-module homomorphism. Since
$K_i$'s acts on $\Omega_q$ as automorphisms, it suffices to check that $E_i$ \& $F_i$ act on any product of two elements from
$\Omega_q$ as rule (4.2) or (4.3).

\smallskip
(1) For $i<m$: $E_i=x_i\partial_{i+1}\sigma_i$, $F_i=\sigma_i^{-1}x_{i+1}\partial_i$. Now $\forall\ x^{(\alpha)}\otimes x^\mu\in\Omega_q$, where $x^{(\alpha)}\in \mathcal A_q(m)$ or $\mathcal A_q(m,\bold1)$ when $\textbf{char}(q)=\ell>0$, $\mu=(\mu_1,\cdots,\mu_n)\in\mathbb Z_2^n$, note that
\begin{align*}
E_i.(x^{(\alpha)}\otimes x^\mu)&=[\,\alpha_i{+}1\,]\,x^{(\alpha+\epsilon_i-\epsilon_{i+1})}\otimes x^{\mu},\\
F_i.(x^{(\alpha)}\otimes x^\mu)&=[\,\alpha_{i+1}{+}1\,]\,x^{(\alpha-\epsilon_i+\epsilon_{i+1})}\otimes x^{\mu}.
\end{align*}
This is exactly reduced to the known case of $U_q(\mathfrak {gl}(m))$ acting on $\mathcal A_q$, which is true due to Theorem 3
(or see \cite{Hu}), where we used $q^{\mu*(\beta+\epsilon_i-\epsilon_{i+1})}=q^{\mu*\beta}$ for $i<m$ in the product of two elements of $\Omega_q$.

\smallskip
(2) For $i=m$: $E_m=x_m\partial_{m+1}\sigma_m$, $\mathcal K_m=\sigma_m\sigma_{m+1}\tau_{m+1}$, $\forall\ x^{(\alpha)}\otimes x^\mu\in\Omega_q$, where $\mu=(\mu_1,\cdots,\mu_n)\in\mathbb Z_2^n$, noting that $E_m(x^{(\alpha)}\otimes x^\mu)=\delta_{1,\mu_1}[\,\alpha_m{+}1\,]\,x^{(\alpha+\epsilon_m)}\otimes x^{\mu-\epsilon_{m+1}}$, we get
\begin{align*}
E_m&.(x^{(\alpha)}\otimes x^\mu)\mathcal K_m.(x^{(\beta)}\otimes x^\nu)+\sigma(x^{(\alpha)}\otimes x^\mu)E_m.(x^{(\beta)}\otimes x^\nu)\\
&=\delta_{1,\mu_1}q^{(\mu-\epsilon_{m+1})*\beta+\beta_m}q^{\delta_{1,\nu_1}}[\alpha_m{+}1]\,x^{(\alpha+\epsilon_m)}x^{(\beta)}\otimes x^{\mu-\epsilon_{m+1}}x^\nu
\\
&\quad+(-1)^{|\mu|}\delta_{1,\nu_1}q^{\mu*(\beta+\epsilon_m)}[\beta_m{+}1]\,x^{(\alpha)}x^{(\beta+\epsilon_m)}\otimes x^{\mu}x^{\nu-\epsilon_{m+1}}
\end{align*}
\begin{align*}
&=\delta_{1,\mu_1}q^{-\epsilon_{m+1}*\beta+\mu*\beta+\beta_m+\delta_{1,\nu_1}}(-q)^{\mu*\nu}q^{(\alpha+\epsilon_m)*\beta}
[\alpha_m{+}1]{\alpha{+}\beta{+}\epsilon_m\brack \alpha{+}\epsilon_m}\,x^{(\alpha+\beta+\epsilon_m)}{\otimes}
x^{(\mu-\epsilon_{m+1})+\nu}\\
&\quad+(-1)^{|\mu|}\delta_{1,\nu_1}q^{\mu*(\beta+\epsilon_m)}(-q)^{\mu*(\nu-\epsilon_{m+1})}
[\beta_m{+}1]q^{\alpha*(\beta{+}\epsilon_m)}{\alpha{+}\beta{+}\epsilon_m\brack \beta{+}\epsilon_m}\,x^{(\alpha+\beta+\epsilon_m)}\otimes x^{\mu+(\nu-\epsilon_{m+1})}\\
&=\bigl(\delta_{1,\mu_1}q^{\delta_{1,\nu_1}}+\delta_{1,\nu_1}(-q)^{\delta_{1,\mu_1}}\bigr)q^{\alpha*\beta+\mu*\beta}
(-q)^{\mu*\nu}[\alpha_m{+}\beta_m{+}1]{\alpha{+}\beta\brack \alpha}\,x^{(\alpha+\beta+\epsilon_m)}{\otimes} x^{\mu+\nu-\epsilon_{m+1}}\\
&=\bigl(\delta_{1,\mu_1}q^{\delta_{1,\nu_1}}+\delta_{1,\nu_1}(-q)^{\delta_{1,\mu_1}}\bigr)q^{\alpha*\beta+\mu*\beta}
(-q)^{\mu*\nu}{\alpha{+}\beta\brack \alpha}\,x_m\partial_{m+1}\sigma_m(x^{(\alpha+\beta)}\otimes x^{\mu+\nu})\\
&=\delta_{1,\mu_1+\nu_1}\,E_m.\Bigl((x^{(\alpha)}\otimes x^{\mu})(x^{(\beta)}\otimes x^{\nu})\Bigr)=E_m.\Bigl((x^{(\alpha)}\otimes x^{\mu})(x^{(\beta)}\otimes x^{\nu})\Bigr),
\end{align*}
where we used $(\epsilon_m-\epsilon_{m+1})*\beta=-\beta_m$ and $\mu*(\epsilon_m-\epsilon_{m+1})=|\mu|-(|\mu|-\delta_{1,\mu_1})=\delta_{1,\mu_1}$.

For $F_m=\sigma_m^{-1}x_{m+1}\partial_m$: $\forall\ x^{(\alpha)}\otimes x^\mu, \ x^{(\beta)}\otimes x^\nu\in\Omega_q$,
noting that $\sigma_m^{-1}x_{m+1}\partial_m(x^{(\alpha)}\otimes x^\mu)=\delta_{0,\mu_1}\,x^{(\alpha-\epsilon_m)}\otimes x^{\mu+\epsilon_{m+1}}$, and $\mathcal K_m^{-1}.(x^\mu)=\sigma_m^{-1}\sigma_{m+1}^{-1}\tau_{m+1}(x^\mu)=q^{-\delta_{1,\mu_1}} x^\mu$,
we have
\begin{align*}
&F_m.(x^{(\alpha)}\otimes x^\mu)(x^{(\beta)}\otimes x^\nu)+\sigma \mathcal K_m^{-1}.(x^{(\alpha)}\otimes x^\mu)F_m.(x^{(\beta)}\otimes x^\nu)\\
&\quad=\delta_{0,\mu_1}q^{(\mu+\epsilon_{m+1})*\beta}\,x^{(\alpha-\epsilon_m)}x^{(\beta)}\otimes x^{\mu+\epsilon_{m+1}}x^\nu\\
&\qquad+(-1)^{|\mu|}\delta_{0,\nu_1}q^{-\alpha_m+\mu*(\beta-\epsilon_m)}q^{-\delta_{1,\mu_1}}\,
x^{(\alpha)}x^{(\beta-\epsilon_m)}{\otimes} x^{\mu}x^{\nu+\epsilon_{m+1}}\\
&\quad=\delta_{0,\mu_1+\nu_1}q^{\mu*\beta+\alpha*\beta+(\epsilon_{m+1}-\epsilon_m)*\beta}(-q)^{(\mu+\epsilon_{m+1})*\nu}
{\alpha{+}\beta{-}\epsilon_m\brack \alpha{-}\epsilon_m}\,x^{(\alpha+\beta-\epsilon_m)}{\otimes}
x^{\mu+\nu+\epsilon_{m+1}}\\
&\qquad+\delta_{0,\mu_1+\nu_1}(-1)^{|\mu|}q^{\mu*\beta-|\mu|+\alpha*(\beta-\epsilon_m)}(-q)^{\mu*(\nu+\epsilon_{m+1})}
q^{-\alpha_m}{\alpha{+}\beta{-}\epsilon_m\brack \beta{-}\epsilon_m}\,x^{(\alpha+\beta-\epsilon_m)}{\otimes} x^{\mu+\nu+\epsilon_{m+1}}\\
&\quad=\delta_{0,\mu_1+\nu_1}q^{\mu*\beta+\alpha*\beta}(-q)^{\mu*\nu}\left(q^{\beta_m}{\alpha{+}\beta{-}\epsilon_m\brack \alpha{-}\epsilon_m}+q^{-\alpha_m}{\alpha{+}\beta{-}\epsilon_m\brack \beta{-}\epsilon_m}\right)\,x^{(\alpha+\beta+\epsilon_m)}{\otimes} x^{\mu+\nu-\epsilon_{m+1}}\\
&\quad=q^{\alpha*\beta+\mu*\beta}(-q)^{\mu*\nu}{\alpha{+}\beta\brack \alpha}\,F_m.(x^{(\alpha+\beta)}\otimes x^{\mu+\nu})=F_m.\Bigl((x^{(\alpha)}\otimes x^{\mu})(x^{(\beta)}\otimes x^{\nu})\Bigr),
\end{align*}
where we used $\epsilon_{m+1}*\nu=0$, $\alpha*\epsilon_m=0$, and $\mu*\epsilon_{m+1}=|\mu|$ for $\mu_1=0$ (which implies $(-1)^{|\mu|}q^{-|\mu|}(-q)^{\mu*\epsilon_{m+1}}=1$), as well as $q^{\beta_m}[\alpha_m]+q^{-\alpha_m}[\beta_m]=[\alpha_m{+}\beta_m]$.

\smallskip
(3) For $i\in I_1$, $E_i=x_i\partial_{i+1}\sigma_{i}$: Viewing $\mu=(\mu_1,{\cdots},\mu_n)\in\mathbb Z_2^n$ as $\mu=(0,{\cdots},0,\mu_1,{\cdots},\mu_n)$ $\in\mathbb Z^m\times\mathbb Z_2^n$, and
putting
$\mu{+}\epsilon_i{-}\epsilon_{i+1}=(0,\cdots,0,\mu_1,\cdots,\mu_{i-m}{+}1,\mu_{i+1-m}{-}1,\cdots,\mu_n)$, we have
$(\epsilon_i-\epsilon_{i+1})*\mu=-\mu_{i-m}$, and $\mu*(\epsilon_i-\epsilon_{i+1})=\mu_{i+1-m}$.
Therefore,
 $\forall\ x^{(\alpha)}\otimes x^\mu$,
we get
\begin{align*}
E_i.(x^{(\alpha)}{\otimes} x^\mu)&=x_i\partial_{i+1}\sigma_{i}(x^{(\alpha)}{\otimes} x^\mu)=\delta_{0,\mu_{i-m}}\delta_{1,\mu_{i+1-m}}(-q)^{(\epsilon_{i}-\epsilon_{i+1})*\mu}\,x^{(\alpha)}{\otimes} x^{\mu+\epsilon_i-\epsilon_{i+1}}\\
&=\delta_{0,\mu_{i-m}}\delta_{1,\mu_{i+1-m}}\,x^{(\alpha)}\otimes x^{\mu+\epsilon_i-\epsilon_{i+1}}=
x^{(\alpha)}\otimes E_i.(x^\mu).
\end{align*}
So, in the case when $\forall\ x^\mu,\ x^\nu\in \Lambda_{q^{-1}}(n)$ with $x^\mu x^\nu\ne 0$ (i.e., $\mu_j{+}\nu_j\in \mathbb Z_2$, $j\in I_1$), we have
\begin{align*}
E_i&.(x^\mu x^\nu)=(-q)^{\mu*\nu}E_i.(x^{\mu+\nu})
=(-q)^{\mu*\nu}\delta_{0,\mu_{i-m}+\nu_{i-m}}\delta_{1,\mu_{i+1-m}+\nu_{i+1-m}}
x^{\mu+\nu+\epsilon_i-\epsilon_{i+1}}\\
&=(-q)^{\mu*\nu}\delta_{0,\mu_{i-m}+\nu_{i-m}}(\delta_{1,\mu_{i+1-m}}\delta_{0,\nu_{i+1-m}}
+\delta_{0,\mu_{i+1-m}}\delta_{1,\nu_{i+1-m}})x^{\mu+\nu+\epsilon_i-\epsilon_{i+1}}\\
&=(-q)^{\mu*\nu}\delta_{0,\mu_{i-m}}\delta_{0,\nu_{i-m}}\delta_{1,\mu_{i+1-m}}\delta_{0,\nu_{i+1-m}}
x^{(\mu+\epsilon_i-\epsilon_{i+1})+\nu}\\
&\quad+(-q)^{\mu*\nu}\delta_{0,\mu_{i-m}}\delta_{0,\nu_{i-m}}\delta_{0,\mu_{i+1-m}}\delta_{1,\nu_{i+1-m}}
x^{\mu+(\nu+\epsilon_i-\epsilon_{i+1})}\\
&=(-q)^{(\mu+\epsilon_i-\epsilon_{i+1})*\nu}\delta_{0,\mu_{i-m}}\delta_{0,\nu_{i-m}}\delta_{1,\mu_{i+1-m}}\delta_{0,\nu_{i+1-m}}
x^{(\mu+\epsilon_i-\epsilon_{i+1})+\nu}\\
&\quad+(-q)^{\mu*(\nu+\epsilon_i-\epsilon_{i+1})}\delta_{0,\mu_{i-m}}\delta_{0,\nu_{i-m}}\delta_{0,\mu_{i+1-m}}\delta_{1,\nu_{i+1-m}}
x^{\mu+(\nu+\epsilon_i-\epsilon_{i+1})}\\
&=(\delta_{0,\mu_{i-m}}\delta_{1,\mu_{i+1-m}}
x^{\mu+\epsilon_i-\epsilon_{i+1}})(\delta_{0,\nu_{i-m}}\delta_{0,\nu_{i+1-m}}x^\nu)\\
&\quad+\delta_{0,\mu_{i-m}}\delta_{0,\mu_{i+1-m}}
x^{\mu}(\delta_{0,\nu_{i-m}}\delta_{1,\nu_{i+1-m}}x^{\nu+\epsilon_i-\epsilon_{i+1}})\\
&=E_i.(x^{\mu})(\delta_{0,\nu_{i-m}}\delta_{0,\nu_{i+1-m}}x^\nu)+x^{\mu}E_i.(x^{\nu})\\
&=E_i.(x^\mu)\mathcal K_i.(x^\nu)+x^\mu E_i.(x^\nu).
\end{align*}

However, there exists only one special case: $\exists$ $i\in I_1$ such that $\mu_{i-m}=0=\nu_{i-m}$, and $\mu_{i+1-m}=1=\nu_{i+1-m}$, i.e.,
$x^\mu x^\nu=0$, but $E_i(x^\mu)x^\nu\ne 0$, $x^\mu E_i(x^\nu)\ne 0$, we have to check that the following identity still holds,
 which facilitates us to fix our unique options so well for  $E_i$'s (4.7) \& $\mathcal K_i$'s (4.12) operator expression in $\mathcal W_q(2(m|n))$ when $i\in I_1$. \begin{align*}
E_i.&(x^\mu)\mathcal K_i.(x^\nu)+x^\mu E_i.(x^\nu)\\
&=\delta_{0,\nu_{i-m}}\delta_{1,\nu_{i+1-m}}\delta_{0,\mu_{i-m}}\delta_{1,\mu_{i+1-m}}
x^{\mu+\epsilon_i-\epsilon_{i+1}}(q^{\nu_{i+1-m}}x^\nu) \\
&\quad+\delta_{0,\mu_{i-m}}\delta_{1,\mu_{i+1-m}}\delta_{0,\nu_{i-m}}\delta_{1,\nu_{i+1-m}}
x^{\mu}x^{\nu+\epsilon_i-\epsilon_{i+1}}\\
&=q^{\nu_{i+1-m}}(-q)^{(\mu+\epsilon_i-\epsilon_{i+1})*\nu}
x^{\mu+\nu+\epsilon_i-\epsilon_{i+1}}\quad (\textit{noting } (\epsilon_i-\epsilon_{i+1})*\nu=-\nu_{i-m}=0)\\
&\quad+(-q)^{\mu*(\nu+\epsilon_i-\epsilon_{i+1})}x^{\mu+\nu+\epsilon_i-\epsilon_{i+1}}\quad (\textit{noting } \mu*(\epsilon_i-\epsilon_{i+1})=\mu_{i+1-m})\\
&=(-q)^{\mu*\nu}\Bigl(q^{\nu_{i+1-m}}+(-q)^{\mu_{i+1-m}}\Bigr)x^{\mu+\nu+\epsilon_i-\epsilon_{i+1}}\\
&=0=E_i.(x^\mu x^\nu).
\end{align*}

Hence,  $\forall\ x^{(\alpha)}\otimes x^\mu$, $x^{(\beta)}\otimes x^\nu\in\Omega_q$, noting that
 $(\mu{+}\epsilon_i{-}\epsilon_{i+1})*\beta=\mu*\beta$, we obtain the expected action rule.
\begin{align*}
E_i.&\bigl((x^{(\alpha)}\otimes x^\mu)(x^{(\beta)}\otimes x^\nu)\bigr)
=q^{\mu*\beta}x^{(\alpha)}x^{(\beta)}\otimes E_i.(x^\mu x^\nu)\\
&=q^{\mu*\beta}x^{(\alpha)}x^{(\beta)}\otimes \Bigl(E_i.(x^{(\mu)})\mathcal K_i.(x^{(\nu)})+x^{(\mu)}E_i.(x^{(\nu)})\Bigr)\\
&=(x^{(\alpha)}\otimes E_i.(x^\mu))(x^{(\beta)}\otimes \mathcal K_i.(x^{(\nu)}))+(x^{(\alpha)}\otimes x^\mu)(x^{(\beta)}\otimes E_i.(x^\nu))\\
&=E_i.(x^{(\alpha)}\otimes x^\mu)\mathcal K_i.(x^{(\beta)}\otimes x^\nu)+(x^{(\alpha)}\otimes x^\mu)E_i.(x^{(\beta)}\otimes x^\nu).
\end{align*}

\smallskip
For $i\in I_1$, $F_i=\sigma_i^{-1}x_{i+1}\partial_i$: Noting that
$$
F_i.(x^{(\alpha)}\otimes x^\mu)
=\delta_{1,\mu_{i-m}}\delta_{0,\mu_{i+1-m}}x^{(\alpha)}\otimes x^{\mu-\epsilon_i+\epsilon_{i+1}}=x^{(\alpha)}\otimes F_i.(x^{\mu}),
$$
so, for any $x^\mu,\ x^\nu\in\Lambda_{q^{-1}}$ with $x^\mu x^\nu\ne 0$ (i.e., $\mu_j{+}\nu_j\in \mathbb Z_2$, $j\in I_1$), we have
\begin{align*}
F_i.(x^\mu x^\nu)&=(-q)^{\mu*\nu}F_i.(x^{\mu+\nu})=\delta_{1,\mu_{i-m}+\nu_{i-m}}\delta_{0,\mu_{i+1-m}+\nu_{i+1-m}}
x^{\mu+\nu-\epsilon_i+\epsilon_{i+1}}\\
&=(-q)^{\mu*\nu}(\delta_{1,\mu_{i-m}}\delta_{0,\nu_{i-m}}+\delta_{0,\mu_{i-m}}\delta_{1,\nu_{i-m}})\delta_{0,\mu_{i+1-m}}\delta_{0,\nu_{i+1-m}}
x^{\mu+\nu-\epsilon_i+\epsilon_{i+1}}\\
&=(-q)^{\mu*\nu+\nu_{i-m}}\delta_{1,\mu_{i-m}}\delta_{0,\mu_{i+1-m}}\delta_{0,\nu_{i-m}}\delta_{0,\nu_{i+1-m}}
x^{(\mu-\epsilon_i+\epsilon_{i+1})+\nu}\\
&\quad+(-q)^{\mu*\nu+\mu_{i+1-m}}\delta_{0,\mu_{i-m}}\delta_{0,\mu_{i+1-m}}\delta_{1,\nu_{i-m}}\delta_{0,\nu_{i+1-m}}
x^{\mu+(\nu-\epsilon_i+\epsilon_{i+1})}\\
&=\delta_{0,\nu_{i-m}}\delta_{0,\nu_{i+1-m}}F_i.(x^\mu)x^\nu+\delta_{0,\mu_{i-m}}\delta_{0,\mu_{i+1-m}}
\mathcal K_i^{-1}.(x^{\mu}) F_i.(x^\nu)\\
&=F_i.(x^\mu)x^\nu+\mathcal K_i^{-1}.(x^{\mu}) F_i.(x^\nu).
\end{align*}

But there exists only one special case: $\exists$ $i\in I_1$ such that $\mu_{i-m}=1=\nu_{i-m}$, and $\mu_{i+1-m}=0=\nu_{i+1-m}$, that is $x^\mu x^\nu=0$ but $F_i(x^\mu)x^\nu\ne 0$ and $x^\mu F_i(x^\nu)\ne 0$, we have to check that the following identity still holds, which coincides with our unique options for $F_i$'s (4.7) \& $\mathcal K_i$'s (4.13) operator expression in $\mathcal W_q(2(m|n))$ when $i\in I_1$.
\begin{align*}
F_i.&(x^\mu)x^\nu+\mathcal K_i^{-1}.(x^\mu) F_i.(x^\nu)\\
&=\delta_{1,\mu_{i-m}}\delta_{1,\nu_{i-m}}\delta_{0,\mu_{i+1-m}}\delta_{0,\nu_{i+1-m}}
\,x^{\mu-\epsilon_i+\epsilon_{i+1}}x^\nu\\
&\quad+\delta_{1,\mu_{i-m}}\delta_{1,\nu_{i-m}}\delta_{0,\mu_{i+1-m}}\delta_{0,\nu_{i+1-m}}
(q^{\mu_{i-m}}x^{\mu})x^{\nu-\epsilon_i+\epsilon_{i+1}}\\
&=(-q)^{(\mu-\epsilon_i+\epsilon_{i+1})*\nu}
x^{\mu+\nu-\epsilon_i+\epsilon_{i+1}}+q^{\mu_{i-m}}(-q)^{\mu*(\nu-\epsilon_i+\epsilon_{i+1})}x^{\mu+\nu-\epsilon_i+\epsilon_{i+1}}\\
&=\Bigl((-q)^{\nu_{i-m}+\mu*\nu}+q^{\mu_{i-m}}(-q)^{\mu*\nu}\Bigr)x^{\mu+\nu-\epsilon_i+\epsilon_{i+1}}\\
&=0=F_i.(x^\mu x^\nu).
\end{align*}

Similar to the above argument, we easily achieve the desired action rule
$$F_i.((x^{(\alpha)}\otimes x^\mu)(x^{(\beta)}\otimes x^\nu))=F_i.(x^{(\alpha)}\otimes x^\mu)(x^{(\beta)}\otimes x^\nu)+
\mathcal K_i^{-1}.(x^{(\alpha)}\otimes x^\mu)F_i.(x^{(\beta)}\otimes x^\nu).$$

This completes the proof of step (I).

\smallskip
(II) Secondly, we need to verify that formulae (4.7)---(4.14)
satisfy all defining relations $(R1)$--$(R7)$ of $\mathcal U_q$.
It is sufficient to check these on a monomial $\Bbbk$-basis of $\Omega_q$.

From the fact that $U_q(\mathfrak{gl(m)})$ can be embedded into $\mathcal U_q(\mathfrak{gl}(m|n))$ as a subalgebra (see Definition 25)
and Theorem \ref{th1} (for details, see \cite{Hu}), we see that the actions of $E_j,\;F_j$ $(j< m)$ and
$K_i$ or, $\mathcal K_i$ $(i\in I_0)$ on $\Omega_q$ coincides with their actions on $\mathcal A_q$, the quantum (restricted)
divided power algebra, which have been examined in \cite{Hu} to satisfy the
relations: $(R1)$--$(R5)$.

Observing that relations $(R1)$--$(R2)$ are clear, in what follows, we need to check the remainders.

$(R3)$: we consider index $i\in J$ in three situations $i<m$, $i=m$ and $i>m$.

$(0)$\quad When $i<m$, $j<m$, this is true due to Theorem \ref{th1}.

$(1)$\quad When $i=m$, $j<m$, we have $i\neq j$, $q_i=q$, and $p(F_j)=0$ which implies that $(R3)$ becomes
$E_mF_j-F_jE_m=0$. Noting that
  \begin{align*}
E_mF_j.(x^{(\beta)}\otimes x^\mu)&=E_m.([\beta_{j+1}{+}1]\,x^{(\beta-\epsilon_j+\epsilon_{j+1})}\otimes x^\mu)\\
&=\delta_{1,\mu_1}[\beta_{j+1}{+}1][\beta_m{+}\delta_{m,j+1}{+}1]\,
x^{(\beta-\epsilon_j+\epsilon_{j+1}+\epsilon_m)}\otimes x^{\mu-\epsilon_{m+1}},\\
F_jE_m.(x^{(\beta)}\otimes x^\mu)&=F_j.(\delta_{m+1,i_1}[\beta_m+1]\,x^{(\beta+\epsilon_m)}\otimes x^{\mu-\epsilon_{m+1}}\\
&=\delta_{1,\mu_1}[\beta_m{+}1][\beta_{j+1}{+}\delta_{m,j+1}{+}1]\,
x^{(\beta-\epsilon_j+\epsilon_{j+1}+\epsilon_m)}\otimes x^{\mu-\epsilon_{m+1}},
\end{align*}
we obtain the required result in both cases $j+1=m$, \& $j+1\neq m$, and thus $(R3)$ holds in this case.
Similarly, we can check $E_iF_m=F_mE_i$ for $i<m$.

$(2)$\quad When $i=j=m$, we have $q_i=q$, $p(E_m)=p(F_m)=1$ and thus $(R3)$ turns to
  \begin{equation}\label{c1}
  E_mF_m+F_mE_m=\frac{\mathcal K_m-\mathcal K_m^{-1}}{q-q^{-1}}.
  \end{equation}

Note that $F_m=\sigma_m^{-1}x_{m+1}\partial_m$ and $E_m=x_m\partial_{m+1}\sigma_m$, so
\begin{gather*}
F_m.(x^{(\beta)}\otimes x^\mu)=\delta_{0,\mu_1}x^{(\beta-\epsilon_m)}\otimes x^{\mu+\epsilon_{m+1}},\\
  E_m.(x^{(\beta)}\otimes x^\mu)=\delta_{1,\mu_1}[\beta_m{+}1]\,x^{(\beta+\epsilon_m)}\otimes x^{\mu-\epsilon_{m+1}}.
   \end{gather*}
   Then we have
\begin{align*}
E_mF_m.(x^{(\beta)}\otimes x^\mu)&=\delta_{0,\mu_1}E_m.(x^{(\beta-\epsilon_m)}\otimes x^{\mu+\epsilon_{m+1}})=\delta_{0,\mu_1}[\beta_m]x^{(\beta)}\otimes x^\mu,\\
F_mE_m.(x^{(\beta)}\otimes x^\mu)&=\delta_{1,\mu_1}[\beta_m{+}1]F_m.(x^{(\beta+\epsilon_m)}\otimes x^{\mu-\epsilon_{m+1}})=\delta_{1,\mu_1}[\beta_m{+}1]x^{(\beta)}\otimes x^\mu.
\end{align*}
Adding them, we get
\begin{align*}\label{c2}
\bigl(E_mF_m+F_mE_m\bigr).(x^{(\beta)}\otimes x^\mu)&=(\delta_{0,\mu_1}[\beta_m]+\delta_{1,\mu_1}[\beta_m{+}1])x^{(\beta)}\otimes x^\mu\\
&=[\beta_m{+}\delta_{1,\mu_1}]\,x^{(\beta)}\otimes x^\mu=
\Big(\frac{\mathcal K_m-\mathcal K_m^{-1}}{q-q^{-1}}\Big)(x^{(\beta)}\otimes x^\mu),
\end{align*}
where we used $\mathcal K_m=\sigma_m\sigma_{m+1}\tau_{m+1}$.

$(3)$\quad When $i=m$, $j>m$, then we have $i\ne j,\ q_i=q,\ p(F_j)=0$, thus $(R3)$ becomes $E_mF_j-F_jE_m=0$. Note $F_j=\sigma_j^{-1}x_{j+1}\partial_j$, $F_j(x^{(\beta)}{\otimes} x^\mu)=\delta_{1,\mu_{j-m}}\delta_{0,\mu_{j+1-m}}x^{(\beta)}{\otimes} x^{\mu-\epsilon_j+\epsilon_{j+1}}$. Then we have
\begin{align*}
E_mF_j.&(x^{(\beta)}\otimes x^\mu)=E_m.(\delta_{1,\mu_{j-m}}\delta_{0,\mu_{j+1-m}}x^{(\beta)}{\otimes} x^{\mu-\epsilon_j+\epsilon_{j+1}})\\
&=\delta_{1,\mu_{j-m}}\delta_{0,\mu_{j+1-m}}\delta_{1,\mu_1-\delta_{j-m,1}}[\beta_m{+}1]\,x^{(\beta+\epsilon_m)}{\otimes} x^{\mu-\epsilon_{m+1}-\epsilon_j+\epsilon_{j+1}},\\
F_jE_m.&(x^{(\beta)}\otimes x^\mu)=F_j.(\delta_{1,\mu_1}[\beta_m{+}1]x^{(\beta+\epsilon_m)}\otimes x^{\mu-\epsilon_{m+1}})\\
&=\delta_{1,\mu_1}[\beta_m{+}1]\delta_{1,\mu_{j-m}-\delta_{j-m,1}}\delta_{0,\mu_{j+1-m}}
x^{(\beta+\epsilon_m)}\otimes x^{\mu-\epsilon_{m+1}-\epsilon_j+\epsilon_{j+1}}.
\end{align*}
Clearly, $E_mF_j=F_jE_m$ for $j>m$. Similarly, we can check $E_iF_m=F_mE_i$ for $i>m$.

$(4)$\quad When $i,\; j>m$, we have $q_i=q^{-1}$ and $p(F_j)=p(E_i)=0$, and need to check the following form of $(R3)$
  \begin{equation}\label{c4}
E_iF_j-F_jE_i=\delta_{ij}\frac{\mathcal K_i-\mathcal K_i^{-1}}{q^{-1}-q}.
  \end{equation}
Indeed, noting that $E_i=x_i\partial_{i+1}\sigma_i$, $F_i=\sigma_i^{-1}x_{i+1}\partial_i$,
$\mathcal K_i=\sigma_i^{-1}\sigma_{i+1}\tau_i\tau_{i+1}$, and
\begin{gather*}
E_i.(x^{(\alpha)}\otimes x^\mu)=\delta_{0,\mu_{i-m}}\delta_{1,\mu_{i+1-m}}
x^{(\alpha)}\otimes x^{\mu+\epsilon_i-\epsilon_{i+1}}=x^{(\alpha)}\otimes x^{\mu+\epsilon_i-\epsilon_{i+1}},\\
F_i.(x^{(\alpha)}\otimes x^\mu)=
\delta_{1,\mu_{i-m}}\delta_{0,\mu_{i+1-m}}x^{(\alpha)}{\otimes} x^{\mu-\epsilon_i+\epsilon_{i+1}}=
x^{(\alpha)}{\otimes} x^{\mu-\epsilon_i+\epsilon_{i+1}},
\end{gather*}
when $i=j$, we have
\begin{align*}
\bigl(E_iF_i-F_iE_i\bigr).(x^{(\alpha)}\otimes x^\mu)
&=\bigl(\delta_{1,\mu_{i-m}}\delta_{0,\mu_{i+1-m}}-\delta_{0,\mu_{i-m}}\delta_{1,\mu_{i+1-m}}\bigr)
x^{(\alpha)}\otimes x^{\mu}\\
&=\frac{q^{-(\mu,\epsilon_{i-m}-\epsilon_{i+1-m})}-
q^{(\mu,\epsilon_{i-m}-\epsilon_{i+1-m})}}{q^{-1}-q}x^{(\alpha)}\otimes x^\mu\\
&=\frac{\sigma_i^{-1}\sigma_{i+1}\tau_i\tau_{i+1}-\sigma_i\sigma_{i+1}^{-1}\tau_i\tau_{i+1}}{q^{-1}-q}.(x^{(\alpha)}\otimes x^\mu)\\&
=\frac{\mathcal K_i-\mathcal K_i^{-1}}{q^{-1}-q}.(x^{(\alpha)}\otimes x^\mu);
\end{align*}
when $i=j-1$, we have
\begin{align*}
\bigl(E_iF_j-F_jE_i\bigr).(x^{(\alpha)}\otimes x^\mu)
&=E_i.(x^{(\alpha)}\otimes x^{\mu-\epsilon_j+\epsilon_{j+1}})
-F_j.(x^{(\alpha)}\otimes x^{\mu+\epsilon_i-\epsilon_{i+1}})\\
&=x^{(\alpha)}\otimes x^{\mu+\epsilon_i-2\epsilon_j+\epsilon_{j+1}}-x^{(\alpha)}\otimes x^{\mu+\epsilon_i-2\epsilon_{i+1}+
\epsilon_{j+1}}=0,
\end{align*}
where we used the convention: $x^\nu=0$ if $\exists\ i_0$ such that $\nu_{i_0}<0$ or $>1$;
when $i<j-1$, obviously, we have $E_iF_j=F_jE_i$.
Similarly we can check that $(R3)$ holds for $i>j$.

$(R4)$: These are clear.

$(R5)$: Note that the assumptions in $(R5)$ are $|\,i{-}j\,|=1$ and $i\neq m$. It follows from Theorem \ref{th1} (for details, see \cite{Hu}) that $(R5)$ holds for the case $1\leq i,\;j<m$ and $|\,i{-}j\,|=1$ when acting on $\Omega_q$. Firstly, we check the relation of $E_i$'s in the case when $i>m$ and $j=i-1$ and the remainder for $E_i$'s with $i>m$ and $j=i+1$,
as well as $F_i$'s can be dealt with in a similar way.

$(1)$\quad When $i=m+1$, $j=m$ and the relation of $E_i$'s becomes
\begin{equation}\label{c6}
    E_{m+1}^2E_m-(q+q^{-1})E_{m+1}E_mE_{m+1}+E_mE_{m+1}^2=0.
    \end{equation}
Note that $E_m=x_m\partial_{m+1}\sigma_m$ and $E_{m+1}=x_{m+1}\partial_{m+2}\sigma_{m+1}$, and
\begin{gather*}
E_m.(x^{(\alpha)}\otimes x^\mu))=[\,\alpha_m{+}1\,]\,x^{(\alpha+\epsilon_m)}\otimes x^{\mu-\epsilon_{m+1}},\\
E_{m+1}.(x^{(\alpha)}\otimes x^\mu)=x^{(\alpha)}\otimes x^{\mu+\epsilon_{m+1}-\epsilon_{m+2}}.
\end{gather*}
So, no matter which one of terms $E_{m+1}^2E_m,\ E_{m+1}E_mE_{m+1}$ and $E_mE_{m+1}^2$ acting on
any basis element $x^{(\alpha)}\otimes x^\mu$, gives such an element $x^{(\alpha+\epsilon_m)}\otimes x^{\mu+\epsilon_{m+1}-2\epsilon_{m+2}}$ even with non-zero coefficient, which vanishes since $\mu_{m+2}-2<0$.
This gives the required relation (\ref{c6}).

$(2)$\quad When $i>m+1$, $j=i-1$,
note that $E_i=x_i\partial_{i+1}\sigma_i$, and $E_i.(x^{(\alpha)}\otimes x^\mu)=x^{(\alpha)}\otimes x^{\mu+\epsilon_i-\epsilon_{i+1}}$. Based on the same reason as the case when $i=m+1$, each term $E_i^2E_{i-1}$,
$E_iE_{i-1}E_i$, or $E_{i-1}E_i^2$ trivially acts on any basis element $x^{(\alpha)}\otimes x^\mu$.
This gives $(R5)$.

$(R6)$: The relation $E_m^2=0$ follows from the same reason as the above proof in $(R5)$ and the relation $F_m^2=0$ follows from the fact that $x_{m+1}^2=0$ in the exterior part $\Lambda_{q^{-1}}(n)$.

$(R7)$: We now check the first relation for $E$-part, as the 2nd one can be checked similarly.
Noting that $E_i=x_i\partial_{i+1}\sigma_i$, for $i=m{-}1,\; m,\; m{+}1$, and
\begin{gather*}
E_{m-1}.(x^{(\alpha)}\otimes x^\mu)=[\,\alpha_{m-1}{+}1\,]\,x^{(\alpha+\epsilon_{m-1}-\epsilon_m)}\otimes x^\mu,\\
E_m.(x^{(\alpha)}\otimes x^\mu)=[\,\alpha_m{+}1\,]\,x^{(\alpha+\epsilon_m)}\otimes x^{\mu-\epsilon_{m+1}},\\
E_{m+1}.(x^{(\alpha)}\otimes x^\mu)=x^{(\alpha)}\otimes x^{\mu+\epsilon_{m+1}-\epsilon_{m+2}},
\end{gather*}
we have
\begin{align*}
&E_{m-1}E_mE_{m+1}E_m.(x^{(\alpha)}\otimes x^\mu)\\
&\qquad=E_{m-1}E_mE_{m+1}.([\,\alpha_m{+}1\,]\,x^{(\alpha+\epsilon_m)}\otimes x^{\mu-\epsilon_{m+1}})\\
&\qquad=E_{m-1}E_m.([\,\alpha_m{+}1\,]\,x^{(\alpha+\epsilon_m)}\otimes x^{\mu-\epsilon_{m+2}}\\
&\qquad=E_{m-1}.([\,\alpha_m{+}1\,][\,\alpha_m{+}2\,]\,x^{(\alpha+2\epsilon_m)}\otimes x^{\mu-\epsilon_{m+1}-\epsilon_{m+2}})\\
&\qquad=[\,\alpha_{m-1}{+}1\,][\,\alpha_m{+}1\,][\,\alpha_m{+}2\,]\,x^{(\alpha+\epsilon_{m-1}+\epsilon_m)}\otimes
x^{\mu-\epsilon_{m+1}-\epsilon_{m+2}},\\
&E_mE_{m-1}E_mE_{m+1}.(x^{(\alpha)}\otimes x^\mu)\\
&\qquad=E_mE_{m-1}E_m.(\delta_{0,\mu_1}\,x^{(\alpha)}\otimes x^{\mu+\epsilon_{m+1}-\epsilon_{m+2}})\\
&\qquad=E_mE_{m-1}.(\delta_{0,\mu_1}[\,\alpha_m{+}1\,]\,x^{(\alpha+\epsilon_m)}\otimes x^{\mu-\epsilon_{m+2}})\\
&\qquad=E_m.(\delta_{0,\mu_1}[\,\alpha_m{+}1\,][\,\alpha_{m-1}{+}1\,]\,x^{(\alpha+\epsilon_{m-1})}\otimes x^{\mu-\epsilon_{m+2}})\\
&\qquad=\delta_{0,\mu_1}\delta_{1,\mu_1}[\,\alpha_m{+}1\,][\,\alpha_{m-1}{+}1\,][\,\alpha_m{+}1\,]\,x^{(\alpha+\epsilon_{m-1}+\epsilon_m)}\otimes x^{\mu-\epsilon_{m+1}-\epsilon_{m+2}}=0,\\
&E_{m+1}E_mE_{m-1}E_m.(x^{(\alpha)}\otimes x^\mu)\\
&\qquad=E_{m+1}E_mE_{m-1}.([\,\alpha_m{+}1\,]\,x^{(\alpha+\epsilon_m)}\otimes x^{\mu-\epsilon_{m+1}})\\
&\qquad=E_{m+1}E_m.([\,\alpha_{m-1}{+}1\,][\,\alpha_m{+}1\,]\,x^{(\alpha+\epsilon_{m-1})}\otimes x^{\mu-\epsilon_{m+1}})\\
&\qquad=E_{m+1}.(0)=0,\\
&E_mE_{m+1}E_mE_{m-1}.(x^{(\alpha)}\otimes x^\mu)\\
&\qquad=E_mE_{m+1}E_m.([\,\alpha_{m-1}{+}1\,]\,x^{(\alpha+\epsilon_{m-1}-\epsilon_m)}\otimes x^\mu)\\
&\qquad=E_mE_{m+1}.([\,\alpha_{m-1}{+}1\,][\,\alpha_m\,]\,x^{(\alpha+\epsilon_{m-1})}\otimes x^{\mu-\epsilon_{m+1}})\\
&\qquad=E_m.([\,\alpha_{m-1}{+}1\,][\,\alpha_m\,]\,x^{(\alpha+\epsilon_{m-1})}\otimes x^{\mu-\epsilon_{m+2}})\\
&\qquad=[\,\alpha_{m-1}{+}1\,][\,\alpha_m\,][\,\alpha_m{+}1\,]\,x^{(\alpha+\epsilon_{m-1}+\epsilon_m)}\otimes x^{\mu-\epsilon_{m+1}-\epsilon_{m+2}},\\
&E_mE_{m-1}E_{m+1}E_m.(x^{(\alpha)}\otimes x^\mu)\\
&\qquad=E_mE_{m-1}E_{m+1}.([\,\alpha_m{+}1\,]\,x^{(\alpha+\epsilon_m)}\otimes x^{\mu-\epsilon_{m+1}})\\
&\qquad=E_mE_{m-1}.([\,\alpha_m{+}1\,]\,x^{(\alpha+\epsilon_m)}\otimes x^{\mu-\epsilon_{m+2}})\\
&\qquad=E_m.([\,\alpha_{m-1}{+}1\,][\,\alpha_m{+}1\,]\,x^{(\alpha+\epsilon_{m-1})}\otimes x^{\mu-\epsilon_{m+2}})\\
&\qquad=[\,\alpha_{m-1}{+}1\,][\,\alpha_m{+}1\,]^2\,x^{(\alpha+\epsilon_{m-1}+\epsilon_m)}\otimes x^{\mu-\epsilon_{m+1}-\epsilon_{m+2}}.
\end{align*}
Note that $[\alpha_m{+}2]+[\alpha_m]-(q{+}q^{-1})[\alpha_m{+}1]=0$. Therefore, we get
\begin{align*}
&\Big(E_{m-1}E_mE_{m+1}E_m +E_mE_{m-1}E_mE_{m+1}+E_{m+1}E_mE_{m-1}E_m+E_mE_{m+1}E_mE_{m-1}\\
&\qquad-(q+q^{-1})E_mE_{m-1}E_{m+1}E_m\Big).(x^{(\alpha)}\otimes x^\mu)\\
&=[\alpha_{m-1}{+}1][\alpha_m{+}1]\Bigl([\alpha_m{+}2]+[\alpha_m]-(q{+}q^{-1})[\alpha_m{+}1]\Bigr)
x^{(\alpha+\epsilon_{m-1}+\epsilon_m)}\otimes x^{\mu-\epsilon_{m+1}-\epsilon_{m+2}}=0.
\end{align*}

\smallskip
This completes the proof of (II) both for $\mathcal U_q(\mathfrak{gl}(m|n))$ and $\mathcal U_q(\mathfrak{sl}(m|n))$.

\smallskip
As for the case $\text{char}(q)=\ell>0$, it follows straightforwardly from formulae (4.7)---(4.14) that
\begin{gather*}
E_j^\ell\,|_{\Omega_q(m|n,\bold{1})}\equiv0,\qquad
F_j^\ell\,|_{\Omega_q(m|n,\mathbf {1})}\equiv0, \qquad \textit{for }\ \;j\, (\ne m)\in J,\\
E_m^2=0,\qquad F_m^2=0,\\
(K_i^{2\ell}-1)\,|_{\Omega_q(m|n,\bold{1})}\equiv0,\qquad\textit{for }\ \; i \in I,
\end{gather*}
Therefore the quantum restricted Grassmann algebra $\Omega_q(m|n,\bold{1})$
is naturally equipped with a $u_q$-module algebra structure,
where
$$u_q:=\mathcal U_q/\bigl(\,E_i^\ell,\ F_i^\ell \ (i\,(\ne m)\in J); \ E_m^2, \ F_m^2, \ K_i^{2\ell}-1\ (i\in I))\,\bigr),$$
and $\mathcal U_q=\mathcal U_q(\mathfrak{gl}(m|n))$ or $\mathcal U_q(\mathfrak{sl}(m|n))$.
\end{proof}

\subsection{The submodule structures on homogeneous spaces of $\Omega_q$}
For each $0\leq j\leq n$, we denote by
\begin{displaymath}
\Lambda_{q^{-1}}(n)^{(j)}:= \textrm{span}_{\Bbbk}\{\,x^\mu\mid |\,\mu\,|=j\,\}.
\end{displaymath}
It follows from the definition that $\Lambda_{q^{-1}}(n)=\bigoplus_{j=0}^n\Lambda_{q^{-1}}(n)^{(j)}$. For any
$t\in \mathbbm{Z}_+$, define
\begin{displaymath}
\Omega_q^{(t)}:=\bigoplus_{i+j=t}\mathcal{A}_q^{(i)}\otimes\Lambda_{q^{-1}}(n)^{(j)},
\end{displaymath}
where $\Omega_q=\Omega_q(m|n)$ or, $\Omega_q(m|n,\bold 1)$ only when $\textbf{char}(q)=\ell>2$, and
$\mathcal A_q=\mathcal A_q(m)$ or, $\mathcal A_q(m,\bold 1)$ only when $\textbf{char}(q)=\ell>2$.

It is clear that the set $\{\,x^{(\alpha)}\otimes x^\mu\in\Omega_q\mid |\,\alpha\,|+|\,\mu\,|=t\,\}$
forms a $\Bbbk$-basis of the homogeneous space $\Omega_q^{(t)}$. Therefore, we have
$\Omega_q=\bigoplus_{t\geq0}\Omega_q^{(t)}$, in particular, $\Omega_q(m|n)=\bigoplus_{t\geq 0}\Omega_q(m|n)^{(t)}$
and $\Omega_q(m|n,\bold{1})=\bigoplus_{t=0}^{N+n}\Omega_q(m|n,\bold {l})^{(t)}$, where $N=|\,\tau(m)\,|=m(\ell{-}1)$.

\begin{theorem}
Each subspace $\Omega_q(m|n)^{(t)}$ is a $\mathcal U_q$-submodule of $\Omega_q(m|n)$ and
$\Omega_q(m|n, \bold 1)^{(t)}$ is a $u_q$-submodule of $\Omega_q(m|n, \bold 1)$ when $\textbf{char}(q)=\ell>2$.
\begin{enumerate}[$($i$)$]
\item \label{p1}If $\textbf{char}(q)=0$, then each submodule $\Omega_q(m|n)^{(t)}\cong V(t\omega_1)$
is a simple module generated by highest weight vector $x^{(\bold{t})}\otimes 1$, where $\bold{t}=t\epsilon_1=(t,0,\cdots,0)$.

\item If $\textbf{char}(q)=\ell$, then the submodule
$\Omega_q(m|n,\bold {1})^{(t)}\cong V((\ell{-}1{-}t_i)\omega_{i-1}{+}t_i\omega_i)$, where
 $t=(i{-}1)(\ell{-}1)+t_i$ $(0\leq t_i\leq \ell{-}1,\ i\in I_0)$, is a simple module generated by highest
 weight vector $x^{(\bold{t})}\otimes 1$, where
 $\bold{t}=(\underbrace{\ell{-}1,\cdots,\ell{-}1}_{i{-}1},{t_i},0,\cdots,0)$; and
the submodule $\Omega_q(m|n,\bold {1})^{(t)}\cong V((\ell{-}2)\omega_m{+}\omega_{m+p})$,
where $t=N{+}p$, $0\le p\leq n$, is a simple module generated by highest weight vector
$x^{(\tau(m))}\otimes x_{m+1}\cdots x_{m+p}$.
\end{enumerate}
\end{theorem}
\begin{proof}
By abuse of notation, let us denote $U:=\mathcal U_q$ or, $u_q$ when $\textbf{char}(q)=\ell$.
It follows from the calculation results of formulae (4.7)---(4.8) that the action of $U$ on $\Omega_q$
stabilizes $\Omega_q^{(t)}$. Therefore each homogeneous subspace $\Omega_q^{(t)}$ is a $U$-submodule of $\Omega_q$.

\smallskip
(1) Now we first prove claim \eqref{p1}.

Note that $\bold{char}(q)=0$. For any $i\in I,\; j\in J$, we have
\begin{align*}
E_j.(x^{(\bold{t})}\otimes 1)=0,\quad
K_i.(x^{(\bold{t})}\otimes 1)=q^{\delta_{i,1}t}x^{(\bold{t})}\otimes 1
=q^{(t\epsilon_1,\epsilon_i)}x^{(\bold{t})}\otimes 1,
\end{align*}
which implies that $x^{(\bold{t})}\otimes 1$ is a highest weight vector with highest weight
$t\epsilon_1=t\omega_1$.

For any $x^{(\alpha)}\otimes x^\mu\in \Omega_q(m|n)^{(t)}$, we assume that $|\,\mu\,|=s$ for some positive integer $s$
such that $s\leq n$. For each $i\in I_0$, set $t_i=t{-}s{-}\sum_{j\leq i}\alpha_j$. Thus we have
\begin{displaymath}
\alpha_1=t{-}s{-}t_1, \alpha_2=t_1{-}t_2, \cdots, \alpha_{m-1}=t_{m-2}{-}t_{m-1},\text{ and }\alpha_m=t_{m-1}.
\end{displaymath}
Repeatedly using the respective actions of $F_j$'s, by $(2.9)$, we get
\begin{align*}
&F_{m-1}^{t_{m-1}+s}\cdots F_2^{t_2+s} F_1^{t_1+s}.(x^{(t\epsilon_1)}\otimes 1)\\
&\qquad=[\,t_1{+}s\,]!\cdots[\,t_{m-1}{+}s\,]!\,x^{((t-s-t_1)\epsilon_1+(t_1-t_2)\epsilon_2+\cdots+(t_{m-2}-t_{m-1})\epsilon_{m-1}
+(t_{m-1}+s)\epsilon_m)}\otimes1\\
&\qquad=[\,t_1{+}s\,]!\cdots[\,t_{m-1}{+}s\,]!\,x^{(\alpha_1\epsilon_1+\cdots+\alpha_{m-1}\epsilon_{m-1}+(\alpha_m+s)\epsilon_m)}
\otimes1,
\end{align*}
where $[m]!=[m]\cdots[2][1]$ for $m\in\mathbbm{Z}_+$. If $|\,\mu\,|=s=0$,
we complete the proof that $x^{(\alpha)}\otimes x^\mu$ is generated by highest weight vector
$x^{(t\epsilon_1)}\otimes 1$. If $|\,\mu\,|=s>0$, then we may assume
$x^\mu=x_{i_1}\cdots x_{i_s}\ (m+1\leq i_1<\cdots<i_s\leq m+n)$ without loss of generality. Thus we have
\begin{align*}
\Big\{&(F_{i_1-1}F_{i_1-2}\cdots F_{m+1}F_{m})(F_{i_2-1}F_{i_2-2}
\cdots F_{m+1}F_{m})\cdots\\
& \times(F_{i_s-1}F_{i_s-2}\cdots F_{m+1}F_{m})\Big\}.(x^{(\alpha_1\epsilon_1+\cdots+\alpha_{m-1}\epsilon_{m-1}+
(\alpha_m+s)\epsilon_m)}\otimes1)\\
&\qquad\qquad\qquad\qquad\qquad\qquad=\ x^{(\alpha_1\epsilon_1+\cdots+\alpha_{m-1}\epsilon_{m-1}+\alpha_m\epsilon_m)}\otimes x_{i_1}\cdots x_{i_s}\\
&\qquad\qquad\qquad\qquad\qquad\qquad=x^{(\alpha)}\otimes x^\mu.
\end{align*}
From the previous two steps, we obtain
\begin{align*}
&\Big\{(F_{i_1-1}F_{i_1-2}\cdots F_{m+1}F_{m})(F_{i_2-1}F_{i_2-2}\cdots F_{m+1}F_{m})\cdots
(F_{i_s-1}F_{i_s-2}\cdots F_{m+1} F_{m})\\
&\quad\times(F_{m-1}^{t_{m-1}+s}\cdots F_2^{t_2+s}F_1^{t_1+s})\Big\}.(x^{(t\epsilon_1)}\otimes 1)
=[\,t_1{+}s\,]!\cdots[\,t_{m-1}{+}s\,]!\,x^{(\alpha)}\otimes x^\mu.
\end{align*}
Thus we show that $\Omega_q(m|n)^{(t)}$ is indeed a highest weight module,
generated by highest weight vector $x^{(t\epsilon_1)}\otimes1$ with highest weight $t\epsilon_1=t\omega_1$.

Furthermore, note that $x^{(\alpha)}\otimes x^\mu\in \Omega_q(m|n)^{(t)}$
(i.e., $|\,\alpha{+}\mu\,|=t$) can be lifted to the highest weight vector
$x^{(t\epsilon_1)}\otimes1$ by repeated actions of some $E_j$'s $(j\in J)$. If $|\,\mu\,|=s=0$,
$x^{(\alpha)}\otimes x^\mu=x^{(\alpha)}\otimes 1$ is reduced to the case $\mathcal A_q(m)^{(t)}$ which has been
proved in \cite{Hu}.
If $|\,\mu\,|=s>0$, that is, $x^\mu=x_{i_1}\cdots x_{i_s}$ $(m{+}1\leq i_1<\cdots<i_s\leq m{+}n)$, then we have
\begin{align}
&\Big\{(E_mE_{m+1}\cdots E_{i_s-2}E_{i_s-1})\cdots(E_mE_{m+1}\cdots E_{i_2-2}E_{i_2-1})(E_mE_{m+1}\cdots
E_{i_1-2}E_{i_1-1})\Big\}.\\
&\qquad\qquad(x^{(\alpha)}\otimes x_{i_1}\cdots x_{i_s})=
[\,\alpha_m{+}1\,][\,\alpha_m{+}2\,]\cdots[\,\alpha_m{+}s\,]\,x^{(\alpha+s\epsilon_m)}\otimes 1.\notag
\end{align}
Meanwhile,
\begin{align}
&E_1^{\alpha_2+\cdots+\alpha_m+s}\cdots E_{m-2}^{\alpha_{m-1}+\alpha_m+s}E_{m-1}^{\alpha_m+s}.
(x^{(\alpha+s\epsilon_m)}\otimes1)\\
&\quad=\prod_{i=1}^{m-1}([\,\alpha_i{+}1\,][\,\alpha_i{+}2\,]\cdots
[\,\alpha_i{+}\alpha_{i+1}{+}\cdots{+}\alpha_{m}{+}s\,])\,x^{((\alpha_1
+\alpha_2+\cdots+\alpha_m+s)\epsilon_1)}\otimes1\notag\\
&\quad=([\,t{-}s{-}t_1{+}1\,][\,t{-}s{-}t_1{+}2\,]\cdots[\,t\,])\prod_{i=2}^{m-1}
([\,t_{i-1}{-}t_i{+}1\,][\,t_{i-1}{-}t_i{+}2\,]\cdots[\,t_{i-1}{+}s\,])\notag\\
&\qquad\times
x^{(t\epsilon_1)}\otimes1\ne 0.\notag
\end{align}

{\it Simplicity criteria}:
Now suppose that $0\ne V\subseteq \Omega_q(m|n)^{(t)}$ is a minimal submodule, i.e., simple submodule,
$\exists$ $0\ne v=\sum_{\alpha,\mu}
a_{\alpha,\mu}x^{(\alpha)}\otimes x^\mu\in V$ with $a_{\alpha,\mu}\ne0$ and $|\,\alpha{+}\mu\,|=t$, among those
$m{+}n$-tuples $(\alpha,\mu)\in \mathbb Z^m\times\mathbb Z_2^n$, with respect to the lexicographical order,
there exist a minimal $m{+}n$-tuple $(\alpha_0,\mu_0)$ and a fixed word $\omega$ with largest length consisting of a certain
$E_j$'s $(j\in J)$ such that $\omega.v=a_{\alpha_0,\mu_0}\omega.x^{(\alpha_0)}\otimes x^{\mu_0}
=*a_{\alpha_0,\mu_0}\,x^{(t\epsilon_1)}\otimes 1\in V$,
by (4.18) \& (4.19).
Thus we acquire that $V=\Omega_q(m|n)^{(t)}\cong V(t\omega_1)$, a simple highest weight module with highest weight
$t\omega_1$.

\smallskip
(2) In the case when $\bold{char}(q)=\ell\geq3$, we consider the module structure
of $\Omega_q(m|n,\bold{1})^{(t)}$ by dividing the possible value of $t$ into the following two cases.

(I) For $0\leq t\leq N$, there exist $i$ and $t_i$ such that
$1\leq i\leq m,\,t=(i-1)(\ell-1)+t_i$ with $0\leq t_i\leq \ell-1$.
Consider the vector $x^{(\bold{t})}\otimes1\in \Omega_q(m|n,\bold{1})^{(t)}$,
where $\bold{t}=(\ell-1)\epsilon_1+\cdots+(\ell-1)\epsilon_{i-1}+t_i\epsilon_i=(\ell{-}1{-}t_i)\omega_{i-1}{+}t_i\omega_i$.
Noting that $[\,\ell\,]=0$ and using the action formulae in Theorem 27, we have
\begin{align*}
E_j.(x^{(\bold{t})}\otimes1)&=0,\quad (j\in J)\\
K_j.(x^{(\bold{t})}\otimes1)&=q^{\delta_{j,i-1}(\ell-1)+\delta_{j,i}t_i}x^{(\bold{t})}\otimes 1=q^{(\bold{t},\epsilon_j)}x^{(\bold{t})}\otimes 1,\quad (j\in I_0)\\
K_j.(x^{(\bold{t})}\otimes1)&=x^{(\bold{t})}\otimes1, \quad (j\in I_1)\\
\mathcal K_j.(x^{(\bold{t})}\otimes1)&=q^{(\bold{t},\epsilon_j-\epsilon_{j+1})}x^{(\bold{t})}\otimes 1, \quad (j\in I_0)\\
\mathcal K_j.(x^{(\bold{t})}\otimes1)&=x^{(\bold{t})}\otimes1, \quad (j\in I_1)
\end{align*}
which imply that $x^{(\bold{t})}\otimes1$ is a highest weight vector with
highest weight $(\ell-1-t_i)\omega_{i-1}+t_i\omega_i$.

Then for any $x^{(\alpha)}\otimes x^\mu\in \Omega_q(m|n,\bold {1})^{(t)}$ with $|\alpha|+|\,\mu\,|=t$\,
(where we assume $0\leq |\,\mu\,|=s\leq n$), we have $|\alpha|=t{-}s=(i{-}1)(\ell{-}1){+}t_i{-}s\geq 0$,
$0\leq\alpha_j\leq \ell{-}1 \ (1\leq j\leq m)$. And there exists a minimal nonnegative integer
$k$ such that $k(\ell{-}1){+}t_i{-}s\in \mathbbm{Z_+}$. Denote by $t_{i-k}'=k(\ell{-}1){+}t_i{-}s$.
Due to the minimality of $k$, we get $0\leq t_{i-k}'<\ell{-}1$. Note that the relative highest weight $\mbox{\boldmath$\nu$}$
with respect to the quantum (Hopf) subalgebra
$u_q(\mathfrak{gl}(m))$ or $u_q(\mathfrak{sl}(m))$, corresponding
to length $|\alpha|=t{-}s=(i{-}k{-}1)(\ell{-}1){+}t_{i-k}'$, is $\mbox{\boldmath$\nu$}=(\ell{-}1)\epsilon_1{+}\cdots{+}
(\ell{-}1)\epsilon_{i-k-1}{+}t_{i-k}'\epsilon_{i-k}$.
If $s=0$ (i.e., $x^\mu=1$), then we get $k=0$, $t_i'=t_i$ and $|\alpha|=t=(i{-}1)(\ell{-}1){+}t_i$.
If $s>0$, we write $x^\mu=x_{i_1}\cdots x_{i_s}$ $(m{+}1\leq i_1<\cdots<i_s\leq m{+}n)$.

Moreover, if $t_i{-}s\in\mathbbm{Z_+}$, then $k=0$ and $t_{i}'=t_i{-}s$. Thus we have
\begin{align*}
&\Big\{(F_{i_{1}-1}F_{i_{1}-2}\cdots F_{m+1}F_{m})(F_{i_{2}-1}F_{i_{2}-2}\cdots F_{m+1} F_{m})\cdots
(F_{i_s-1}F_{i_s-2}\cdots F_{m+1}F_{m})\times\\
&\quad\times(F_{m-1}^{s}\cdots F_{i+1}^{s}F_i^{s})\Big\}.(x^{(\bold{t})}\otimes 1)
=([\,s\,]!)^{m-i}x^{((\ell-1)\epsilon_1+\cdots+(\ell-1)\epsilon_{i-1}+t_i'\epsilon_i)}\otimes x_{i_1}\cdots x_{i_s}\\
&\qquad=*\,x^{(\mbox{\boldmath$\nu$})}\otimes x^\mu\neq 0.
\end{align*}
So vector $x^{((\ell-1-t_i')\omega_{i-1}+t_i'\omega_i)}\otimes x^\mu$
is generated by highest weight vector $x^{(\bold{t})}\otimes1$.

If $t_i{-}s<0$, then $k>0$. Thus for $\bold t=(\ell{-}1)\epsilon_1{+}\dots{+}(\ell{-}1)\epsilon_{i-1}{+}t_i\epsilon_i$, we have
\begin{align*}
&\left\{(F_{i_{(s-t_i+1)}-1}F_{i_{(s-t_i+1)}-2}\cdots F_{m+1}F_{m})(F_{i_{(s-t_i+2)}-1}F_{i_{(s-t_i+2)}-2}
\cdots F_{m+1} F_{m})\right.\cdots\times\\
&\qquad \times(F_{i_s-1}F_{i_s-2}\cdots F_{m+1}F_{m})
(F_{m-1}^{t_i}\cdots F_{i+1}^{t_i}F_i^{t_i})\Big\}.(x^{(\bold{t})}\otimes 1)\\
&\quad\qquad=([\,t_i\,]!)^{m-i}x^{((\ell-1)\epsilon_1+\cdots+(\ell-1)\epsilon_{i-1})}\otimes
\underbrace{x_{i_{s-t_i+1}}\cdots x_{i_s}}_{t_i}\neq 0.\\
\end{align*}
So, for $t_{i-k}'=k(\ell{-}1){-}(s{-}t_i)\ge0$, i.e., $s{-}t_i{-}k(\ell{-}1){+}t_{i-k}'{+}1=1$, we have
\begin{align*}
&\Big\{(F_{i_{(s-t_i-k(\ell-1)+t_{i-k}'+1)}-1}\cdots F_{m+1}F_{m})
(F_{i_{(s-t_i-k(\ell-1)+t_{i-k}'+2)}-1}\cdots F_{m+1}F_{m})
\cdots\times\\
&\quad\times(F_{i_{(s-t_i-(k-1)(\ell-1))}-1}
\cdots F_{m+1}F_{m})
(F_{m-1}^{\ell-1-t_{i-k}'}\cdots F_{i-k+1}^{\ell-1-t_{i-k}'}F_{i-k}^{\ell-1-t_{i-k}'})\times\\
&\quad\times(F_{i_{(s-t_i-(k-1)(\ell-1)+1)}-1}\cdots F_{m+1}F_{m})
\cdots(F_{i_{(s-t_i-(k-2)(\ell-1))}-1}\cdots F_{m+1}F_{m})\times\\
&\quad\times(F_{m-1}^{\ell-1}\cdots F_{i-k+2}^{\ell-1}F_{i-k+1}^{\ell-1})\cdots(F_{i_{(s-t_i-2(\ell-1)+1)}-1}\cdots
F_{m+1}F_{m})\cdots\times\\
&\quad\times(F_{i_{(s-t_i-(\ell-1))}-1}\cdots F_{m+1}F_{m})(F_{m-1}^{\ell-1}\cdots F_{i-1}^{\ell-1}F_{i-2}^{\ell-1})
(F_{i_{(s-t_i-(\ell-1)+1)}-1}\cdots F_{m+1}F_{m})\cdots\times\\
&\quad\times(F_{i_{(s-t_i)}-1}\cdots F_{m+1}F_{m})
(F_{m-1}^{\ell-1}\cdots F_i^{\ell-1}F_{i-1}^{\ell-1})
\Big\}.
(x^{((\ell-1)\epsilon_1+\cdots+(\ell-1)\epsilon_{i-1})}\otimes x_{i_{s-t_i+1}}\cdots x_{i_s})\\
&\qquad=([\,\ell{-}1{-}t_{i-k}'\,]!)^{m-i+k}([\,\ell{-}1\,]!)^{(k-1)(m-i)+\frac{k(k-1)}{2}}
 x^{((\ell-1)\epsilon_1+\cdots+(\ell-1)\epsilon_{i-k-1}
+t_{i-k}'\epsilon_{i-k})}\\
&\quad\qquad\otimes x_{i_1}\cdots x_{i_s}=*'\,x^{(\mbox{\boldmath$\nu$})}\otimes x^\mu \neq\ 0.
\end{align*}
That is, the vector $x^{(\mbox{\boldmath$\nu$})}\otimes x^\mu$ is generated by highest weight
vector $x^{(\bold{t})}\otimes 1$.

Moreover, since Proposition 4.2 in \cite{Hu} has proved that for any $1\le t' \le N$, $\mathcal A_q(m,\bold 1)^{(t')}$
is a simple $\mathfrak u_q(\mathfrak{gl}(m))$-
or $\mathfrak u_q(\mathfrak{sl}(m))$-module in the case when $\textbf{char}(q)=\ell>2$, we deduce that
any vector $x^{(\alpha)}\otimes x^\mu\in \Omega_q(m|n,\bold 1)^{(t)}$ with
$t'=t-s=|\,\alpha\,|=(i{-}k{-}1)(\ell{-}1){+}t_{i-k}'=|\,${\boldmath$\nu$}$\,|$ is generated by
$x^{(\mbox{\boldmath$\nu$})}\otimes x^\mu$ in terms of some suitable operators $F_j$'s $(j<m)$
only from the subalgebra $\mathfrak u_q(\mathfrak{sl}(m))$,
where $\mbox{\boldmath$\nu$}:=(\ell{-}1)\epsilon_1{+}\cdots+(\ell{-}1)\epsilon_{i-k-1}{+}t_{i-k}'\epsilon_{i-k}$, is as above.

As $x^{(\alpha)}
\otimes x^\mu\in\Omega_q(m|n,\bold{1})^{(t)}$ (with $|\,\alpha\,|+|\,\mu\,|=t$) is arbitrary,
$\Omega_q(m|n,\bold{1})^{(t)}$ is a highest
weight module generated by highest weight vector $x^{(\bold{t})}\otimes 1$.

\smallskip
In what follows, we will prove $\Omega_q(m|n,\bold 1)^{(t)}$ is simple as $u_q(\mathfrak {gl}(m|n))$- or
$u_q(\mathfrak {sl}(m|n))$-module. To this end, first of all, we need to prove the following fact that
any monomial vector $x^{(\alpha)}\otimes x^\mu\in \Omega_q(m|n,\bold{1})^{(t)}$
can be lifted to highest weight vector $x^{(\bold{t})}\otimes1$ by some suitable actions of $E_j$'s $(j\in J)$.
We do it in two steps.
Firstly, we lift vector $x^{(\alpha)}\otimes x^\mu$ to $x^{(\mbox{\boldmath$\nu$})}\otimes x^\mu$,
which is due to the special case
appeared in Claim (B) of the proof of Proposition 3.5 in \cite{GH} for $\alpha\in\mathcal A_q(m, \bold 1)^{(t')}$
with $t'=t{-}s=|\,\alpha\,|=|\,\mbox{\boldmath$\nu$}\,|$
and $\mathfrak u_q(\mathfrak{sl}(m))$
when $\bold m=\bold 1$ (the restricted case). Secondly, we lift vector $x^{(\mbox{\boldmath$\nu$})}\otimes x^\mu$ to
highest weight vector $x^{(\bold{t})}\otimes1$.

Denote $r$ by the last ordinal number with $\alpha_r\neq 0$ for
the $m$-tuple $\alpha=(\alpha_1,\cdots,\alpha_m)\in\mathcal A_q(m, \bold 1)^{(t')}$.
Thus $r\geq i{-}k$ if $t_{i-k}'\neq 0$, and $r\geq i{-}k{-}1$ if $t_{i-k}'=0$.

\smallskip
Step 1. Case (a): If $t_{i-k}'\geq\alpha_r\,(>0)$, then 
\begin{align*}
E_{i-k}^{\alpha_r}\cdots E_{r-2}^{\alpha_r}E_{r-1}^{\alpha_r}.(x^{(\mbox{\boldmath$\nu$}'+\alpha_r\epsilon_r)}\otimes x^\mu)
&=([\,\alpha_r\,]!)^{r-i+k-1}[\,t_{i-k}'{-}\alpha_r{+}1\,][\,t_{i-k}'{-}\alpha_r{+}2\,]\cdots[\,t_{i-k}'\,]\\
&\quad\times x^{(\mbox{\boldmath$\nu$})}\otimes x^\mu\neq 0,
\end{align*}
where $\mbox{\boldmath$\nu$}':=\mbox{\boldmath$\nu$}-\alpha_r\epsilon_{i-k}=(\ell{-}1)\epsilon_1{+}\cdots{+}
(\ell{-}1)\epsilon_{i-k-1}{+}(t_{i-k}'{-}\alpha_r)\epsilon_{i-k}$.

Case (b): If $\alpha_r>t_{i-k}'\,(\geq 0)$, then 
\begin{align*}
&\bigl(E_{i-k}^{t_{i-k}'}\cdots E_{r-2}^{t_{i-k}'} E_{r-1}^{t_{i-k}'})(E_{i-k-1}^{\alpha_r-t_{i-k}'}\cdots
E_{r-2}^{\alpha_r-t_{i-k}'} E_{r-1}^{\alpha_r-t_{i-k}'}\bigr).(x^{(\mbox{\boldmath$\nu$}'+\alpha_r\epsilon_r)}\otimes x^\mu)\\
&\qquad=\bigl([\,\alpha_r{-}t_{i-k}'\,]![\,t_{i-k}'\,]!\bigr)^{r-i+k}[\,\ell{-}\alpha_r{+}
t_{i-k}'\,][\,\ell{-}\alpha_r{+}t_{i-k}'{+}1\,]\cdots[\,\ell{-}1\,]\, x^{(\mbox{\boldmath$\nu$})}\otimes x^\mu\neq 0,
\end{align*}
where
$$\mbox{\boldmath$\nu$}':=\mbox{\boldmath$\nu$}{-}(\alpha_r{-}t_{i-k}')
\epsilon_{i-k-1}-t_{i-k}'\epsilon_{i-k}=(\ell{-}1)\epsilon_1{+}\cdots{+}(\ell{-}1)\epsilon_{i-k-2}{+}(\ell{-}1{-}\alpha_r{+}
t_{i-k}')\epsilon_{i-k-1}.$$

Inductively, on the pair $(\alpha':=\alpha{-}\alpha_r\epsilon_r, \mbox{\boldmath$\nu$}')$,
there exists a word $\eta'$ constructed by some $E_j$'s
$(j<r{-}1< m)$ in $u_q(\mathfrak{sl}(m))
\cong u_q(\mathfrak{sl}(m|0))\subseteq u_q(\mathfrak{sl}(m|n))\subset u_q(\mathfrak{gl}(m|n))$  such that $\eta'(x^{(\alpha')}
\otimes x^\mu)=\eta'(x^{({\alpha}')})
\otimes x^\mu=c\,x^{(\mbox{\boldmath$\nu$}')}\otimes x^\mu$ $(c\in \Bbbk^\times)$. Note that $E_j(x^{(\alpha_r\epsilon_r)})=0$
and $\mathcal K_j(x^{(\alpha_r\epsilon_r)})=x^{(\alpha_r\epsilon_r)}$ for those $j<r{-}1$ and $\Delta(E_j)=E_j\otimes
\mathcal K_j+1\otimes E_j$. Using the $u_q(\mathfrak{sl}(m))$-module algebra structure on $\mathcal A_q(m,\bold 1)$ (see
Theorem 4.1 in \cite{Hu}),
we thus obtain
\begin{align*}
\eta'.(x^{(\alpha)}\otimes x^\mu)&=
\eta'.(x^{(\alpha')}x^{(\alpha_r\epsilon_r)}
\otimes x^\mu)
=\eta'.(x^{(\alpha')}x^{(\alpha_r\epsilon_r)})
\otimes x^\mu\\
&=\eta'.(x^{(\alpha')})\,\mathcal K_{\eta'}.(x^{(\alpha_r\epsilon_r)})+x^{(\alpha')}\,\eta'.(x^{(\alpha_r\epsilon_r)})
\otimes x^\mu\\
&=c\,x^{(\mbox{\boldmath$\nu$}')}x^{(\alpha_r\epsilon_r)}\otimes x^\mu\\
&=c\,x^{(\mbox{\boldmath$\nu$}'+\alpha_r\epsilon_r)}
\otimes x^\mu\neq0.
\end{align*}

Thereby, this finishes the lifting from $x^{(\alpha)}\otimes x^\mu$ to $x^{(\mbox{\boldmath$\nu$})}\otimes x^\mu$
via the above two cases.

\smallskip
Step 2. Recall that $t_{i-k}'=k(\ell{-}1){-}s{+}t_i\ge 0$ $(0\leq t_i, t_{i-k}'\leq \ell{-}1)$ with $k$ the minimal
nonnegative integer. 
If $s=0$, we have $k=0$, $t_i'=t_i$, namely, $x^{(\mbox{\boldmath$\nu$})}\otimes x^\mu=x^{(\bold{t})}\otimes1$.
If $s>0$, write $x^\mu=x_{i_1}\cdots x_{i_s}\,(m{+}1\leq i_1<\cdots<i_s\leq m{+}n)$.
If $t_i{-}s\ge 0$, then $k=0$, $t_{i}'=t_i{-}s$. Observing that $\mbox{\boldmath$\nu$}
=(\ell{-}1)\epsilon_1{+}\cdots{+}(\ell{-}1)\epsilon_{i-1}
{+}(t_i{-}s)\epsilon_{i}$, we have
\begin{align*}
&\Big\{(E_i^{s}\cdots E_{m-2}^{s}E_{m-1}^{s})(E_{m}E_{m+1}\cdots E_{i_{s}-2}E_{i_{s}-1})\cdots(E_{m}E_{m+1}\cdots E_{i_{2}-2}E_{i_{2}-1})\\
&\qquad\qquad\times(E_{m}E_{m+1}\cdots E_{i_{1}-2}E_{i_{1}-1})\Big\}.
(x^{(\mbox{\boldmath$\nu$})}
\otimes x_{i_1}\cdots x_{i_s})\\
&\qquad\qquad\quad=([\,s\,]!)^{m-1-i}[\,t_i'{+}1\,][\,t_i'{+}2\,]\cdots[\,t_i\,]\,x^{(\bold{t})}\otimes 1\neq 0.
\end{align*}
So vector $x^{(\mbox{\boldmath$\nu$})}\otimes x^\mu$ can be lifted to highest weight
vector $x^{(\bold{t})}\otimes1$.

If $t_i{-}s<0$, then $k>0$.
Noting that $\mbox{\boldmath$\nu$}
=(\ell{-}1)\epsilon_1{+}\cdots{+}(\ell{-}1)\epsilon_{i-k-1}{+}t'_{i-k}\epsilon_{i-k}$,
and $t_{i-k}'=k(\ell{-}1){-}(s{-}t_i)$ $(0\leq t_{i-k}'<\ell{-}1)$, we have
\begin{align*}
&\Big\{(E_{i-k}^{\ell-1-t_{i-k}'}\cdots E_{m-2}^{\ell-1-t_{i-k}'}E_{m-1}^{\ell-1-t_{i-k}'})(E_{m}E_{m+1}\cdots
E_{i_{(\ell-1-t_{i-k}')}-2}E_{i_{(\ell-1-t_{i-k}')}-1})\cdots\\
&\quad\times(E_{m}E_{m+1}\cdots E_{i_2-2}E_{i_2-1})(E_{m}E_{m+1}\cdots E_{i_1-2}E_{i_1-1})\Big\}.(x^{(\mbox{\boldmath$\nu$})}
\otimes x_{i_1}\cdots x_{i_s})\\
&\qquad=([\,\ell{-}1{-}t_{i-k}'\,]!)^{m-1-i+k}[\,t_{i-k}'{+}1\,][\,t_{i-k}'{+}2\,]\cdots[\,\ell{-}1\,]\\
&\quad\qquad\times x^{((\ell-1)\epsilon_1+\cdots+(\ell-1)\epsilon_{i-k})}\otimes x_{i_{\ell-t_{i-k}'}}\cdots x_{i_s}\neq 0,\\
&\Big\{(E_{i}^{t_i}\cdots E_{m-2}^{t_i}E_{m-1}^{t_i})(E_{m}E_{m+1}\cdots E_{i_{(k(\ell-1)-t_{i-k}'+t_i)}-1})
\cdots(E_{m}E_{m+1}\cdots E_{i_{(k(\ell-1)-t_{i-k}'+2)}-1})\\
&\quad\times(E_{m}E_{m+1}\cdots E_{i_{(k(\ell-1)-t_{i-k}'+1)}-1})
(E_{i-1}^{\ell-1}\cdots E_{m-2}^{\ell-1}E_{m-1}^{\ell-1})(E_{m}E_{m+1}\cdots E_{i_{(k(\ell-1)-t_{i-k}')}-1})\cdots\\
&\quad\times(E_{m}E_{m+1}\cdots E_{i_{((k-1)(\ell-1)-t_{i-k}'+2)}-1})(E_{m}E_{m+1}\cdots E_{i_{((k-1)(\ell-1)-t_{i-k}'+1)}-1})\cdots\\
&\quad\times(E_{i-k+1}^{\ell-1}\cdots E_{m-2}^{\ell-1}E_{m-1}^{\ell-1})(E_{m}E_{m+1}\cdots E_{i_{(2(\ell-1)-t_{i-k}'})-1})\cdots
(E_{m}E_{m+1}\cdots E_{i_{(\ell-t_{i-k}'+1)}-1})\\
&\quad\times(E_{m}E_{m+1}\cdots E_{i_{(\ell-t_{i-k}')}-1})
\Big\}.
(x^{((\ell-1)\epsilon_1+\cdots+(\ell-1)\epsilon_{i-k})}\otimes x_{i_{\ell-t_{i-k}'}}\cdots x_{i_s})\\
&\qquad= ([\,t_i\,]!)^{m-i+1}([\,\ell{-}1\,]!)^{(k-1)(m-i)+\frac{(k-1)(k+2)}{2}}
 x^{(\bold{t})}\otimes 1\neq\ 0,
\end{align*}
where $\bold t=(\ell{-}1)\epsilon_1{+}\cdots{+}(\ell{-}1)\epsilon_{i-1}{+}t_i\epsilon_i$.
Hence, vector $x^{(\mbox{\boldmath$\nu$})}\otimes x_{i_1}\cdots x_{i_s}$
can be lifted to highest weight vector $x^{(\bold{t})}\otimes 1$.

In summary, following the above two steps, we can lift any monomial $x^{(\alpha)}\otimes x^\mu$
to highest weight vector $x^{(\bold{t})}\otimes 1$.

Using the simplicity criteria argument of the last paragraph of part (1) of this proof,
we deduce that $\Omega_q(m|n,\bold{1})^{(t)}\cong V((\ell{-}1{-}t_i)\omega_{i-1}{+}t_i\omega_i)$
is a simple $u_q$-module in case (I).

\smallskip
(II) If $N<t\leq N{+}n$, there exists $p$ with $0<p\leq n$ such that $t=N{+}p$. We consider the vector
$x^{(\tau(m))}\otimes x^\mu$, where $x^\mu=x_{m+1}\cdots x_{m+p}\in \Omega_q(m|n,\bold{1})^{(t)}$,
$\tau(m)=(\ell{-}1,\cdots,\ell{-}1)=(\ell{-}1)\epsilon_1{+}\cdots{+}(\ell{-}1)\epsilon_{m}$.
Note that $\mathcal K_i=K_iK_{i+1}^{-1}$ for $i\in J$, and $K_j.(x_i)=q_i^{\delta_{ij}}x_i$ for $i, j\in I$,
where $q_i=q$ for $i\in I_0$
and $q_i=q^{-1}$ for $i\in I_1$.
By Theorem 27, we have
\begin{align*}
E_j.(x^{(\tau(m))}\otimes x^\mu)&=0, \quad(j\in J)
\end{align*}
\begin{align*}
K_j.(x^{(\tau(m))}\otimes x^\mu)&=q^{-1}x^{(\tau(m))}\otimes x^\mu=q^{((\ell-2)\omega_m{+}\omega_{m+p},\epsilon_j)}x^{(\tau(m))}\otimes x^\mu, \quad(j\in I_0)\\
K_j.(x^{(\tau(m))}\otimes x^\mu)&=q^{-1}x^{(\tau(m))}\otimes x^\mu, \quad(m<j\le m{+}p)\\
K_j.(x^{(\tau(m))}\otimes x^\mu)&=x^{(\tau(m))}\otimes x^\mu,\quad (j>m{+}p)\\
\mathcal K_i.(x^{(\tau(m))}\otimes x^\mu)&=x^{(\tau(m))}\otimes x^\mu, \quad (i<m)\\
\mathcal K_m.(x^{(\tau(m))}\otimes x^\mu)&=q^{\ell-1}q\,x^{(\tau(m))}\otimes x^\mu=x^{(\tau(m))}\otimes x^\mu,\\
\mathcal K_i.(x^{(\tau(m))}\otimes x^\mu)&=x^{(\tau(m))}\otimes x^\mu, \quad(m<i<m{+}p \, \textrm{ or},\, i>m{+}p)\\
\mathcal K_{m{+}p}.(x^{(\tau(m))}\otimes x^\mu)&=q^{-1}x^{(\tau(m))}\otimes x^\mu
=q_{m+p}^{((\ell-2)\omega_m{+}\omega_{m+p},\epsilon_{m+p}-\epsilon_{m+p+1})}x^{(\tau(m))}\otimes x^\mu,
\end{align*}
which imply that $x^{(\tau(m))}\otimes x^\mu$ is a highest weight vector with highest weight $(\ell{-}2)\omega_m{+}\omega_{m+p}$.

For any vector $x^{(\alpha)}\otimes x^\mu\in
\Omega_q(m|n,\bold{1})^{(t)}$ with $|\,\alpha\,|+|\,\mu\,|=t=N{+}p$, due to $N<t$ and $|\,\alpha\,|\leq N$,
we have $|\,\mu\,|=s\ge p$, $|\,\alpha\,|=t{-}s=N{-}(s{-}p),
\ 0\leq s{-}p\leq n{-}1$, $0\leq\alpha_i\leq \ell{-}1$ $(1\leq i\leq m)$. Then there exist $j$ and $t_j$ such that
$1\leq j\leq m,\,s{-}p=(j{-}1)(\ell{-}1){+}t_j$ with $0\leq t_j\leq \ell{-}1$. Hence $|\,\alpha\,|=N{-}(j{-}1)(\ell{-}1){-}t_j
=(m{-}j)(\ell{-}1){+}\ell{-}1{-}t_j$. Keep the notation $x^\mu=x_{i_1}\cdots x_{i_s}\ (1\leq i_1<\cdots<i_s\leq m{+}n)$.
Since $s{-}p\geq 0$ and $i_1<\cdots<i_s$, we have $i_{s-k}\geq m{+}p{-}k\geq m{+}1$, for each $0\leq k\leq p{-}1$. Since
\begin{align*}
&(F_{i_s-1}\cdots F_{m+p})(x^{(\tau(m))}\otimes x_{m+1}\cdots x_{m+p})
=x^{(\tau(m))}\otimes x_{m+1}\cdots x_{m+p-1}x_{i_s}\neq0,
\end{align*}
 we obtain
\begin{align*}
&\Big\{(F_{i_{(s-p+1)}-1}\cdots F_{m+1})(F_{i_{(s-p+2)}-1}\cdots F_{m+2})\cdots(F_{i_{(s-1)}-1}\cdots F_{m+p-1})\times\\
&\qquad\times(F_{i_s-1}\cdots F_{m+p})\Big\}.(x^{(\tau(m))}\otimes x_{m+1}\cdots x_{m+p})\\
&\qquad\qquad=x^{(\tau(m))}\otimes x_{i_{s-p+1}}x_{i_{s-p+2}}\cdots x_{i_{s-1}}x_{i_s}\neq 0.
\end{align*}
Moreover, noting that $s{-}p=(j{-}1)(\ell{-}1){+}t_j$, we have
\begin{align*}
&\Big\{(F_{i_{(s-p-(j-1)(\ell-1)+1)}-1}\cdots F_{m+1}F_{m})\cdots
(F_{i_{(s-p-(j-2)(\ell-1))}-1}\cdots F_{m+1}F_{m})\times\\
&\quad\times (F_{m-1}^{\ell-1}F_{m-2}^{\ell-1}\cdots F_{m-j+2}^{\ell-1})\cdots(F_{i_{(s-p-3(\ell-1)+1)}-1}\cdots F_{m+1}F_{m})\cdots\\
&\quad\times(F_{i_{(s-p-2(\ell-1))}-1}\cdots F_{m+1}F_{m})(F_{m-1}^{\ell-1}F_{m-2}^{\ell-1})(F_{i_{(s-p-2(\ell-1)+1)}-1}\cdots
F_{m+1}F_{m})\cdots\\
&\quad\times(F_{i_{(s-p-(\ell-1))}-1}\cdots F_{m+1}F_{m})(F_{m-1}^{\ell-1})
(F_{i_{(s-p-(\ell-1)+1)}-1}\cdots F_{m+1}F_{m})\cdots\\
&\quad\times(F_{i_{(s-p)}-1}\cdots F_{m+1}F_{m})
\Big\}.(x^{(\tau(m))}\otimes x_{i_{s-p+1}}x_{i_{s-p+2}}\cdots x_{i_{s-1}}x_{i_s})\\
&\qquad=([\,\ell{-}1\,]!)^{\frac{(j-1)(j-2)}{2}}
x^{((\ell-1)\epsilon_1+\cdots+(\ell-1)\epsilon_{m-j+1})}\otimes x_{i_{s-p-(j-1)(\ell-1)+1}}\cdots x_{i_s}\\
&\qquad=([\,\ell{-}1\,]!)^{\frac{(j-1)(j-2)}{2}}
x^{((\ell-1)\epsilon_1+\cdots+(\ell-1)\epsilon_{m-j+1})}\otimes x_{i_{(t_j+1)}}\cdots x_{i_s}\neq 0,
\end{align*}
\begin{align*}
&\Big\{(F_{i_1-1}\cdots F_{m+1}F_{m})(F_{i_2-1}\cdots F_{m+1}F_{m})\cdots
(F_{i_{(t_j)}-1}\cdots F_{m+1}F_{m})\times\\
&\quad\times(F_{m-1}^{t_j}F_{m-2}^{t_j}\cdots F_{m-j+1}^{t_j})\Big\}.
(x^{((\ell-1)\epsilon_1+\cdots+(\ell-1)\epsilon_{m-j+1})}\otimes x_{i_{(t_j+1)}}\cdots x_{i_s})\\
&\qquad=([\,t_j\,]!)^{j-1}x^{((\ell-1)\epsilon_1+\cdots+(\ell-1)\epsilon_{m-j}+(\ell-1-t_j)\epsilon_{m-j+1})}
\otimes x_{i_1}\cdots x_{i_s}\ \neq\ 0.
\end{align*}
Namely, we get that vector $x^{((\ell-1)\epsilon_1+\cdots+(\ell-1)\epsilon_{m-j}+
(\ell-1-t_j)\epsilon_{m-j+1})}\otimes x^\mu$ with $|\,\mu\,|=s\ge p$ is
generated by highest weight vector $x^{(\tau(m))}\otimes x_{m+1}\cdots x_{m+p}$.

As $|\,\alpha\,|=(m{-}j)(\ell{-}1){+}\ell{-}1{-}t_j$, by the same argument as the last paragraph of Step 1 in part (I),
any vector $x^{(\alpha)}\otimes x^\mu$ is generated by the
$x^{((\ell-1)\epsilon_1+\cdots+(\ell-1)\epsilon_{m-j}+(\ell-1-t_j)\epsilon_{m-j+1})}\otimes x^\mu$
via some suitable actions of $F_i$'s in $\mathfrak u_q(\mathfrak{sl}(m))$. Hence, for $N<t\leq N{+}n$,
$\Omega_q(m|n,\bold{1})^{(t)}$
is a highest weight module generated by highest weight vector $x^{(\tau(m))}\otimes x_{m+1}\cdots x_{m+p}$.

On the other hand, any vector $x^{(\alpha)}\otimes x^\mu\in \Omega_q(m|n,\bold{1})^{(t)}$ with $|\,\alpha\,|{+}|\,\mu\,|=N{+}p$
 can be lifted to highest weight vector $x^{(\tau(m))}\otimes x_{m+1}\cdots x_{m+p}$
by some actions of $E_j$'s $(j\in J)$. Indeed, we can lift the vector $x^{(\alpha)}\otimes x^\mu$
to $x^{((\ell-1)\epsilon_1+\cdots+(\ell-1)\epsilon_{m-j}+(\ell-1-t_j)\epsilon_{m-j+1})}\otimes x^\mu$ by
following Step 1 in part (I) similarly, so here we don't reduplicate the details.

Observing that $s=t_j{+}(j{-}1)(\ell{-}1){+}p$ or, $s{-}t_j{-}(j{-}1)(\ell{-}1)=p$, and
$x^\mu=x_{i_1}\cdots x_{i_s}$ $(m{+}1\leq i_1<\cdots<i_s\leq m{+}n)$,
we have
\begin{align*}
&\Big\{(E_{m-j+1}^{t_j}\cdots E_{m-2}^{t_j}E_{m-1}^{t_j})(E_m\cdots E_{i_{t_j}-2}E_{i_{t_j}-1})
\cdots(E_m\cdots E_{i_2-2}E_{i_2-1})\times\\
&\quad\times(E_m\cdots E_{i_1-2}E_{i_1-1})
\Big\}.
(x^{((\ell-1)\epsilon_1+\cdots+(\ell-1)\epsilon_{m-j}+(\ell-1-t_j)\epsilon_{m-j+1})}\otimes x_{i_1}{\cdots} x_{i_s})\\
&\qquad=([\,t_j\,]!)^{j-1}[\,\ell{-}t_j\,][\,\ell{-}t_j{+}1\,]\cdots[\,\ell{-}1\,]\,x^{((\ell-1)\epsilon_1+\cdots
+(\ell{-}1)\epsilon_{m-j+1})}\otimes x_{i_{(t_j+1)}}\cdots x_{i_s}\neq 0,\\
&\Big\{(E_mE_{m+1}{\cdots} E_{i_{(t_j+(j-1)(\ell{-}1))}-1}){\cdots}(E_mE_{m+1}{\cdots}
E_{i_{(t_j+(j-2)(\ell-1)+1)}-1})(E_{m-1}^{\ell-1})\\
&\quad\times(E_mE_{m+1}{\cdots} E_{i_{(t_j+(j-2)(\ell-1))}-1}){\cdots}(E_mE_{m+1}{\cdots}
E_{i_{(t_j+2(\ell-1)+1)}-1})(E_{m-j+3}^{\ell-1}{\cdots} E_{m-2}^{\ell-1}E_{m-1}^{\ell-1})\\
&\quad\times(E_mE_{m+1}{\cdots} E_{i_{(t_j+2(\ell-1))}-1}){\cdots}(E_mE_{m+1}{\cdots}
E_{i_{(t_j+\ell)}-1})(E_{m-j+2}^{\ell-1}{\cdots} E_{m-2}^{\ell-1}E_{m-1}^{\ell-1})\\
&\quad\times(E_m{\cdots} E_{i_{(t_j+\ell-1)}-1}){\cdots}(E_m{\cdots} E_{i_{(t_j+1)}-1})
\Big\}.(x^{((\ell-1)\epsilon_1+\cdots+(\ell-1)\epsilon_{m-j+1})}\otimes x_{i_{(t_j+1)}}{\cdots} x_{i_s})\\
&\qquad=([\,\ell{-}1\,]!)^{\frac{j(j-1)}{2}}x^{(\tau(m))}\otimes
\underbrace{x_{i_{(t_j+(j-1)(\ell-1)+1)}}{\cdots} x_{i_s}}_p\neq 0,\\
&\Big\{(E_{m+p}E_{m+p+1}{\cdots} E_{i_{(t_j+(j-1)(\ell-1)+p)}-1}){\cdots}(E_{m+2}E_{m+3}{\cdots}
E_{i_{(t_j+(j-1)(\ell-1)+2)}-1})\times
\\
&\quad\times(E_{m+1}E_{m+2}{\cdots} E_{i_{(t_j+(j-1)(\ell-1)+1)}-1})\Big\}.
(x^{(\tau(m))}\otimes x_{i_{(t_j+(j-1)(\ell-1)+1)}}{\cdots} x_{i_s})\\
&\qquad=x^{(\tau(m))}\otimes x_{m+1}{\cdots} x_{m+p}\neq 0.
\end{align*}
Therefore, we get highest weight vector
$x^{(\tau(m))}\otimes x^\mu$ by starting to lift any monomial $x^{(\alpha)}\otimes x^\mu\in \Omega_q(m|n,\bold{1})^{(t)}$.

Based on the above fact just proved and using the same argument (for the simplicity criteria) in the last paragraph of part (1), we conclude that $\Omega_q(m|n,\bold{1})^{(t)}\ (N<t=N{+}p\leq N{+}n)$ is a
simple $u_q$-module isomorphic to $V((\ell{-}2)\omega_m{+}\omega_{m+p})$.
\end{proof}

Note that the conclusions of Theorem 28 (ii) in the overlapping cases: $i=m$, $t_m=\ell{-}1$ and $p=0$ are the same,
namely, both are $V((\ell{-}1)\omega_m)$.

\begin{coro}
\begin{enumerate}[$($i$)$]
\item \label{p1}If $\textbf{char}(q)=0$, then for $t\le n$,
$$\dim_k\Omega_q(m|n)^{(t)}=\dim_k V(t\omega_1)=
\sum_{0\le s\le t}\binom{m{+}t{-}s{-}1}{t{-}s}\binom{n}{s};$$
and for $t>n$, $$\dim_k\Omega_q(m|n)^{(t)}=\dim_k V(t\omega_1)=
\sum_{0\le s\le n}\binom{m{+}t{-}s{-}1}{t{-}s}\binom{n}{s}.$$

\item If $\textbf{char}(q)=\ell$, then  for $t\le n$, $$\dim_k\Omega_q(m|n,\bold {1})^{(t)}=\dim_k V((\ell{-}1{-}t_i)\omega_{i-1}{+}t_i\omega_i)=
\sum_{0\le s\le t}\binom{n}{s}\dim_k\mathcal A_q^{(t-s)}(m;\bold 1);$$ and for $t>n$,
$$
\dim_k\Omega_q(m|n,\bold {1})^{(t)}=\dim_k V((\ell{-}1{-}t_i)\omega_{i-1}{+}t_i\omega_i)=\sum_{0\le s\le n}\binom{n}{s}\dim_k\mathcal A_q^{(t-s)}(m;\bold 1),
$$
where $\dim_k\mathcal A_q^{(s)}(m;\bold 1)=\sum_{i=0}^{\lfloor
\frac{s}{\ell}\rfloor}(-1)^i\binom{m}{i}\binom{m{+}s{-}i\ell{-}1}{n{-}1}$,
where by $\lfloor x\rfloor$ means the
integer part of $x\in\mathbb Q$.
\end{enumerate}
\end{coro}
\begin{proof}
Note that $\dim_k \mathcal A_q^{(s)}(m)=\dim_k \text{Sym}^s V=\binom{m{+}s{-}1}{s}$ (see \cite{G} or \cite{Hum}) when $\textbf{char}(q)=0$
and $\dim_k V=m$, while $\dim_k  \mathcal A_q^{(s)}(m;\bold 1)=\sum_{i=0}^{\lfloor
\frac{s}{\ell}\rfloor}(-1)^i\binom{m}{i}\binom{m{+}s{-}i\ell{-}1}{m{-}1}$ (see Corollary 2.6 \cite{GH}) when $\textbf{char}(q)=\ell>0$.
Then the results follow from the proof of Theorem 28.
\end{proof}

\section{Manin dual of quantum affine $(m|n)$-superspace $A_q^{m|n}$, quantum Grassmann dual superalgebra $\Omega_q^!(m|n)$ and its module algebra}

\subsection{Manin dual of quantum affine $(m|n)$-superspace $A_q^{m|n}$}
According to Manin's concept (see \cite{M}), the Manin dual of the quantum affine $m$-space $A_q^{m|0}$ is the quantum exterior $m$-space $A_q^{0|m}=\Lambda_q(m)$, and vise versa. So the Manin dual object of the quantum affine $(m|n)$-superspace $A_q^{m|n}=
A_q^{m|0}\otimes A_{q^{-1}}^{0|n}$ is $(A_q^{m|n})^!=A_q^{0|m}\otimes A_{q^{-1}}^{n|0}$, as vector spaces.
\begin{defi}
The Manin dual, $(A_q^{m|n})^!$, of the quantum affine $(m|n)$-superspace $A_q^{m|n}$
is defined to be the quotient of the free associative algebra $\Bbbk\{\,y_i\mid i\in I\,\}$ over $\Bbbk$ by the quadratic ideal
$J(\bold V)$ generated by
\begin{equation}
\begin{aligned}\label{eq0}
y_i^2, \ y_jy_i+q^{-1}y_iy_j, \quad\textit{for }\ i\in I_0, \ j\in I,\ j>i;\\
y_jy_i-q^{-1}y_iy_j, \quad\textit{for }\ i,\ j\in I_1, \ j>i.
\end{aligned}
\end{equation}
\end{defi}

Obviously, $(A_q^{m|n})^!\cong A_q^{0|m}\otimes_{\Bbbk} A_{q^{-1}}^{n|0}=\Lambda_q(m)\otimes A_{q^{-1}}^{n|0}=\Lambda_q(m)_{\bar 0}\otimes_{\Bbbk} A_{q^{-1}}^{n|0} \oplus  \Lambda_q(m)_{\bar 1}\otimes_{\Bbbk}A_{q^{-1}}^{n|0}$, as vector spaces. In fact, the Manin dual $(A_q^{m|n})^!$ is an associative $\Bbbk$-superalgebra which has a natural monomial basis consisting of $\{y^{\langle \mu, \alpha\rangle}:=y^\mu\otimes y^{\alpha}\mid \mu=\langle i_1, \cdots, i_m\rangle\in \mathbb Z_2^m,\ \alpha=(\alpha_1,\cdots,\alpha_n)\in \mathbb Z_+^n\}$, where $y^\mu=y_{i_1}\cdots y_{i_m}$ and $y^\alpha=y_1^{\alpha_1}\cdots y_n^{\alpha_n}$.
It is convenient for us to consider $(m+n)$-tuple $\langle \mu, \alpha\rangle\in \mathbbm Z_+^{m+n}$ with the convention:
 $y^{\langle \mu, \alpha\rangle}=0$ for $\langle \mu,\alpha\rangle\not\in \mathbb Z_2^m\times\mathbb Z_+^n$, and
 extend the $*$-product in Lemma 2 to the $(m+n)$-tuples in $\mathbbm Z^{m+n}$. We still have the following
\begin{equation}
\langle\mu, \alpha\rangle * \langle \nu, \beta\rangle=\mu*\nu+\alpha*\beta+\alpha*\nu=\mu*\nu+\alpha*\beta+|\nu||\alpha|.
\end{equation}
Now we can write down its multiplication formula explicitly on $(A_q^{m|n})^!$ as follows
\begin{equation}
y^{\langle \mu, \alpha\rangle}\cdot y^{\langle \nu, \beta\rangle}
=(-1)^{\alpha*\beta}(-q)^{-\langle\mu, \alpha\rangle*\langle\nu, \beta\rangle} y^{\langle \mu+\nu, \alpha+\beta\rangle}.
\end{equation}

\subsection{Quantum dual Grassmann superalgebra as $\mathcal U_q$-module algebra}

Parallel to the quantum Grassmann superalgebra $\Omega_q(m|n)$, we now introduce its Manin dual object
$\Omega_q^!(m|n)$, which we call it the quantum dual Grassmann superalgebra.

\begin{defi}
The quantum dual Grassmann superalgebra $\Omega_q^!(m|n)$ is defined as a superspace over $\Bbbk$ with the multiplication given by
\begin{equation}
(x^\mu\otimes x^{(\alpha)})\cdot(x^\nu\otimes x^{(\beta)})=(-q)^{-\alpha*\nu}x^\mu x^\nu\otimes x^{(\alpha)}x^{(\beta)},
\end{equation}
for any $x^\mu, x^\nu\in\Lambda_q(m),\, x^{(\alpha)}, x^{(\beta)}\in\mathcal{A}_{q^{-1}}(n)$,
which is an associative $\Bbbk$-superalgebra.

When $\text{\bf char}(q)=\ell\geq 3$, $\Omega_q^!(m|n,\bold{1}):=\Lambda_q(m)\otimes\mathcal{A}_{q^{-1}}(n,\bold {1})$
is a sub-superalgebra, which is referred to as the quantum restricted dual Grassmann superalgebra.
\end{defi}

From Subsections \ref{s2.5} and \ref{s3.1}, we see that the set
\begin{displaymath}
\{\,x^\mu\otimes x^{(\alpha)}\mid\mu\in \mathbbm Z_2^m,\,\alpha\in\mathbbm Z_+^n\,\}
\end{displaymath}
forms a monomial $\Bbbk$-basis of $\Omega_q^!(m|n)$,  and the set
\begin{displaymath}
\{\,x^\mu\otimes x^{(\alpha)}\mid\mu\in\mathbbm Z_2^m,\, \alpha\in\mathbbm Z_+^n,\,\alpha\leq\tau(n)=(\ell{-}1,\cdots,\ell{-}1)\,\}
\end{displaymath}
forms a monomial $\Bbbk$-basis of $\Omega_q^!(m|n,\bold{1})$.

Set $\mathcal A_{q^{-1}}:=\mathcal A_{q^{-1}}(n)$ or, $\mathcal A_{q^{-1}}(n, \bold 1)$. Over the quantum (restricted) dual Grassmann superalgebra $\Omega_q^!=
\Lambda_q(m)\otimes (\mathcal A_q^{-1})_{\bar 0}\bigoplus \Lambda_q(m)\otimes (\mathcal A_{q^{-1}})_{\bar 1}$ with
$(\mathcal A_{q^{-1}})_{\bar i}$ and $\bar i\equiv i$ $(\textit{mod \,} 2)$, we will need the parity automorphism
$\tau': \Omega_q^!\longrightarrow \Omega_q^!$ of order $2$ defined by
\begin{equation}
\tau'(x^\mu\otimes x^{(\alpha)})=(-1)^{|\alpha|}x^\mu\otimes x^{(\alpha)}, \quad\forall\  x^\mu\otimes x^{(\alpha)}\in\Omega_q^!.
\end{equation}

By Proposition 4.3 \cite{Hu}, we know that $\Lambda_q(m)$ is a $U_q(m)$- (or $u_q(m)$-)
module with the same realization of Theorem 4.1 \cite{Hu}, where $U_q(m)=U_q(\mathfrak{gl}(m))$ or
$U_q(\mathfrak{sl}(m))$, $u_q=u_q(\mathfrak{gl}(m))$ or $u_q(\mathfrak{sl}(m))$.
Meanwhile, $\mathcal A_{q^{-1}}$ is
$U_{q^{-1}}(n)$- (or $u_{q^{-1}}(n)$-) module with the realization given by Theorem 4.1 \cite{Hu} (instead $q$ by $q^{-1}$).
Based on Remark 2.3, we guess the following

\begin{theorem}
For any $x^\mu\otimes x^{(\alpha)}\in\Omega_q^!:=\Omega_q^!(m|n)$ or, $\Omega_q^!(m|n,\bold 1)$ when $\textbf{char}(q)=\ell>2$, set
\begin{gather}
E_j.(x^\mu\otimes x^{(\alpha)})= (x_j\partial_{j+1}\sigma_{j})(x^\mu\otimes x^{(\alpha)}), \quad (j\in J)\label{eq6}\\
F_j.(x^\mu\otimes x^{(\alpha)})= (\sigma_j^{-1}x_{j+1}\partial_j)(x^\mu\otimes x^{(\alpha)}), \quad (j\in J)\label{eq7}\\
K_j.(x^\mu\otimes x^{(\alpha)})=\sigma_j(x^\mu\otimes x^{(\alpha)}), \quad (j\in I) \\
 \mathcal K_j.(x^\mu\otimes x^{(\alpha)})=\sigma_j\sigma_{j+1}^{-1}(x^\mu\otimes x^{(\alpha)}), \quad (j\in J) \\
\sigma(x^\mu\otimes x^{(\alpha)})=\tau'(x^\mu\otimes x^{(\alpha)})=(-1)^{|\alpha|}x^\mu\otimes x^{(\alpha)}.\label{eq9}
\end{gather}
Formulae \eqref{eq6}---\eqref{eq9} define a $\mathcal U_q$-module algebra structure over the quantum $($restricted$)$ dual  Grassmann superalgebra $\Omega_q^!$, where $\mathcal U_q:=\mathcal U_q(\mathfrak{gl}(m|n))$,  $\mathcal U_q(\mathfrak{sl}(m|n))$ respectively, or the restricted object $u_q(\mathfrak{gl}(m|n))$, $u_q(\mathfrak{sl}(m|n))$ respectively, when $\textbf{char}(q)=\ell>2$.
\end{theorem}

Before giving the proof, let us make some definitions.
Recall that: in $\Lambda_q(m)$, for $1\le i<j\le m$, we have $x_jx_i=-q^{-1}x_ix_j$. So, for
$\epsilon_j,\, \nu\in\mathbb Z_2^m$, we have
\begin{gather}
x_j.(x^\nu)=(-q)^{-\epsilon_j*\nu}\delta_{\nu_j,0}x^{\nu+\epsilon_j}, \quad(j\in I_0)\\
\partial_j.(x^\nu)=(-q)^{\epsilon_j*\nu}\delta_{\nu_j,1}x^{\nu-\epsilon_j}, \quad(j\in I_0).
\end{gather}
While for $j\in I_1,\ x^{(\alpha)}\in \mathcal A_{q^{-1}}$, we have
\begin{gather}
x_j.(x^{(\alpha)})=q^{-\epsilon_j*\alpha}[\,\alpha_j{+}1\,]\,x^{(\alpha+\epsilon_j)}, \quad (j\in I_1)\\
\partial_j.(x^{(\alpha)})=q^{\epsilon_j*\alpha}x^{(\alpha-\epsilon_j)}, \quad(j\in I_1)
\end{gather}
Meanwhile,
since $x_j.x^\nu=(-q)^{-\epsilon_j*\nu}x^\nu.x_j$ for $j\in I_1$ and $x^\nu\in \Lambda_q(m)$,
this allows us to make convention:
\begin{gather}
x_j.(x^\nu\otimes x^{(\alpha)})=(-q)^{-\epsilon_j*\nu}x^\nu\otimes x_j.(x^{(\alpha)}), \quad (j\in I_1)\\
\partial_j.(x^\nu\otimes x^{(\alpha)})=(-q)^{\epsilon_j*\nu}x^\nu\otimes \partial_j.(x^{(\alpha)}), \quad (j\in I_1).
\end{gather}
These lead to the actions of $E_j$, $F_j$ and $\mathcal K_j^{\pm 1}$ on $\Lambda_q(m)\otimes\mathcal A_{q^{-1}}$ when $j\in I_1$ essentially direct acting on the second factor $\mathcal A_{q^{-1}}$, which has been well given by Theorem 4.1 \cite{Hu}
(instead of $q$ there by $q^{-1}$): $E_j.(x^\nu\otimes x^{(\alpha)})=x^\nu\otimes E_j.(x^{(\alpha)})=x^\nu\otimes x_j\partial_{j+1}\sigma_j.(x^{(\alpha)})$, $F_j.(x^\nu\otimes x^{(\alpha)})=x^\nu\otimes F_j.(x^{(\alpha)})=x^\nu\otimes \sigma_j^{-1}x_{j+1}\partial_j.(x^{(\alpha)})$, where
\begin{align}
\sigma_j(x^{(\alpha)})&=q^{-\alpha_j}x^{(\alpha)}, \quad(j\in I_1).
\end{align}

\begin{proof} (I) Firstly, in order to prove the module structure on $\Omega_q^!(m|n)$, 
it suffices to check relations: (R3) for the cases when $i\le m$, (R5) for the case when $i=m{+}1$ and $j=m$,
(R6) \& (R7).

(R3): For $i=j<m$: Noticing (5.11) \& (5.12), we have
\begin{gather*}
E_i.(x^\nu\otimes x^{(\alpha)})
=(-q)^{(\epsilon_{i+1}-\epsilon_i)*\nu}\delta_{0,\nu_i}\delta_{1,\nu_{i+1}}
x^{\nu+\epsilon_i-\epsilon_{i+1}}\otimes x^{(\alpha)}=\delta_{0,\nu_i}\delta_{1,\nu_{i+1}}
x^{\nu+\epsilon_i-\epsilon_{i+1}}\otimes x^{(\alpha)},\\
F_i.(x^\nu\otimes x^{(\alpha)})
=(-q)^{(\epsilon_i-\epsilon_{i+1})*\nu+1}\delta_{1,\nu_i}\delta_{0,\nu_{i+1}}
x^{\nu-\epsilon_i+\epsilon_{i+1}}\otimes x^{(\alpha)}=\delta_{1,\nu_i}\delta_{0,\nu_{i+1}}
x^{\nu-\epsilon_i+\epsilon_{i+1}}\otimes x^{(\alpha)}.
\end{gather*}
Clearly, when $|\,i{-}j\,|>1$, we have $E_iF_j=F_jE_i$, and when $|\,i{-}j\,|=1$, we can assume $j=i{+}1$,
then
\begin{equation*}
\begin{split}
E_iF_{i+1}.(x^\nu\otimes x^{(\alpha)})&=E_i.(\delta_{1,\nu_{i+1}}\delta_{0,\nu_{i+2}}
x^{\nu-\epsilon_{i+1}+\epsilon_{i+2}}\otimes x^{(\alpha)})=0,\\
F_{i+1}E_i.(x^\nu\otimes x^{(\alpha)})&=F_{i+1}.(\delta_{0,\nu_i}\delta_{1,\nu_{i+1}}
x^{\nu+\epsilon_i-\epsilon_{i+1}}\otimes x^{(\alpha)})=0,
\end{split}
\end{equation*}
namely, $E_iF_{i+1}=F_{i+1}E_i$, for $i<m$. Similarly, we have $E_{i+1}F_i=F_iE_{i+1}$, for $i<m$.

And on the other hand, we get
\begin{equation*}
\begin{split}
(E_iF_i-F_iE_i).(x^\nu\otimes x^{(\alpha)})&=(E_iF_i-F_iE_i).(x^\nu)\otimes x^{(\alpha)}\\
&=(\delta_{1,\nu_i}\delta_{0,\nu_{i+1}}-\delta_{0,\nu_i}\delta_{1,\nu_{i+1}})
x^{\nu}\otimes x^{(\alpha)}\\
&=\frac{\mathcal K_i-\mathcal K_i^{-1}}{q-q^{-1}}.(x^\nu\otimes x^{(\alpha)}),
\end{split}
\end{equation*}
where $\mathcal K_i=\sigma_i\sigma_{i+1}^{-1}$ for $i<m$, and
\begin{align}
\sigma_i(x^\nu)&=q^{\nu_i}x^\nu, \quad (i\in I_0).
\end{align}

For $i=j=m$: By (5.11) --- (5.16), we have
\begin{gather}
E_m.(x^\nu\otimes x^{(\alpha)})=x_m\partial_{m+1}\sigma_m(x^\nu\otimes x^{(\alpha)})
=
\delta_{0,\nu_{m}}x^{\nu+\epsilon_m}\otimes x^{(\alpha-\epsilon_{m+1})},\\
F_m.(x^\nu\otimes x^{(\alpha)})=\sigma_m^{-1}x_{m+1}\partial_m(x^\nu\otimes x^{(\alpha)})
=
\delta_{1,\nu_m}[\,\alpha_{1}{+}1\,]\,x^{\nu-\epsilon_m}\otimes x^{(\alpha+\epsilon_{m+1})}.
\end{gather}
Therefore, we obtain
\begin{align*}
\Bigl(E_mF_m+F_mE_m\Bigr).(x^\nu\otimes x^{(\alpha)})
&=\bigl(\delta_{1,\nu_m}[\,\alpha_1{+}1\,]+\delta_{0,\nu_m}[\,\alpha_{1}\,]\bigr)(x^\nu\otimes x^{(\alpha)})\\
&=\frac{\mathcal K_m-\mathcal K_m^{-1}}{q-q^{-1}}.(x^\nu\otimes x^{(\alpha)}),
\end{align*}
where $\mathcal K_m=\sigma_m\sigma_{m+1}^{-1}$.

For $i=m$, $j<m$, we have $i\neq j$, $q_i=q$, and $p(F_j)=0$ which implies that $(R3)$ becomes
$E_mF_j-F_jE_m=0$. Noting that $F_j=\sigma_j^{-1}x_{j+1}\partial_j$, $E_m=x_m\partial_{m+1}\sigma_m$, by (5.13) \& (5.14),
  \begin{align*}
E_mF_j.(x^\mu\otimes x^{(\alpha)})&=E_m.(\delta_{\mu_j,1}\delta_{\mu_{j+1},0}x^{\mu-\epsilon_j+\epsilon_{j+1}}\otimes x^{(\alpha)})\\
&=\delta_{\mu_j,1}\delta_{\mu_{j+1},0}(-q)^{(\epsilon_{m+1}-\epsilon_m)*(\mu-\epsilon_j+\epsilon_{j+1})}
x^{\mu-\epsilon_j+\epsilon_{j+1}+\epsilon_m}\otimes x^{(\alpha-\epsilon_{m+1})},\\
F_jE_m.(x^\mu\otimes x^{(\alpha)})&=F_j.((-q)^{(\epsilon_{m+1}-\epsilon_m)*\mu}x^{\mu+\epsilon_m}\otimes x^{(\alpha-\epsilon_{m+1})})\\
&=\delta_{\mu_j,1}\delta_{\mu_{j+1},0}(-q)^{(\epsilon_{m+1}-\epsilon_m)*\mu}
x^{\mu-\epsilon_j+\epsilon_{j+1}+\epsilon_m}\otimes x^{(\alpha-\epsilon_{m+1})},
\end{align*}
we obtain the required result in both cases $j{+}1=m$, \& $j{+}1<m$, and thus $(R3)$ holds in this case.

For $i=m$, $j>m$, then we have $i\ne j,\ q_i=q$, and $p(F_j)=0$, thus $(R3)$ becomes $E_mF_j-F_jE_m=0$. Note $F_j=\sigma_j^{-1}x_{j+1}\partial_j$, $F_j(x^\mu{\otimes}x^{(\alpha)})=[\,\alpha_{j+1}{+}1\,]\,x^\mu{\otimes} x^{(\alpha-\epsilon_j+\epsilon_{j+1})}$. Then we have
\begin{align*}
E_mF_j.(x^\mu\otimes x^{(\alpha)})&=E_m.([\,\alpha_{j+1}{+}1\,]\,x^\mu\otimes x^{(\alpha-\epsilon_j+\epsilon_{j+1})})\\
&=[\,\alpha_{j+1}{+}1\,]\,\delta_{\mu_m,0}\,x^{\mu+\epsilon_m}\otimes x^{(\alpha-\epsilon_{m+1}-\epsilon_j+\epsilon_{j+1})},\\
F_jE_m.(x^\mu\otimes x^{(\alpha)})&=F_j.(\delta_{\mu_m,0}x^{\mu+\epsilon_m}\otimes x^{(\alpha-\epsilon_{m+1})})\\
&=\delta_{\mu_m,0}[\,\alpha_{j+1}{+}1\,]\,
x^{\mu+\epsilon_m}\otimes x^{(\alpha-\epsilon_{m+1}-\epsilon_j+\epsilon_{j+1})}.
\end{align*}
Clearly, $E_mF_j=F_jE_m$ for $j>m$.

Similarly, we can check $E_iF_m=F_mE_i$ for $i\ne m$.

(R5): It suffices to consider the case when $i=m{+}1$, $j=m$ to check that $F^2_{m+1}F_m-(q{+}q^{-1})F_{m+1}F_mF_{m+1}+F_mF_{m+1}^2=0$. This is true, because noting that $F_i=\sigma_i^{-1}x_{i+1}\partial_i$,
we have
\begin{equation*}
\begin{split}
F_{m+1}^2F_m.(x^\mu\otimes x^{(\alpha)})&=F_{m+1}^2.([\,\alpha_1{+}1\,]x^{\mu-\epsilon_m}\otimes x^{(\alpha+\epsilon_{m+1})})\\
&=[\,\alpha_1{+}1\,][\,\alpha_2{+}1\,][\,\alpha_2{+}2\,]\,x^{\mu-\epsilon_m}\otimes x^{(\alpha-\epsilon_{m+1}+2\epsilon_{m+2})},\\
F_{m+1}F_mF_{m+1}.(x^\mu\otimes x^{(\alpha)})&=F_{m+1}F_m.([\,\alpha_2{+}1\,]\,x^\mu\otimes x^{(\alpha-\epsilon_{m+1}+\epsilon_{m+2})})\\
&=[\,\alpha_1\,][\,\alpha_2{+}1\,][\,\alpha_2{+}2\,]\,x^{\mu-\epsilon_m}\otimes x^{(\alpha-\epsilon_{m+1}+2\epsilon_{m+2})},\\
F_mF_{m+1}^2.(x^\mu\otimes x^{(\alpha)})&=F_m.([\,\alpha_2{+}1\,][\,\alpha_2{+}2\,]\,x^{\mu}\otimes x^{(\alpha-2\epsilon_{m+1}+2\epsilon_{m+2})})\\
&=[\,\alpha_1{-}1\,][\,\alpha_2{+}1\,][\,\alpha_2{+}2\,]\,x^{\mu-\epsilon_m}\otimes x^{(\alpha-\epsilon_{m+1}+2\epsilon_{m+2})}.
\end{split}
\end{equation*}
Similarly, we can check $E^2_{m+1}E_m-(q{+}q^{-1})E_{m+1}E_mE_{m+1}+E_mE_{m+1}^2=0$.

(R6): $E_m^2=0=F_m^2$ is due to $x_m^2=0$ in $\Lambda_q(m)$.

(R7): Noting that $E_i=x_i\partial_{i+1}\sigma_i$, and $x_m^2=0$, we have
\begin{equation*}
\begin{split}
E_{m-1}E_mE_{m+1}E_m.(x^\mu\otimes x^{(\alpha)})&=E_{m-1}E_mE_{m+1}.(x^{\mu+\epsilon_m}\otimes x^{(\alpha-\epsilon_{m+1})})\\
&=E_{m-1}E_m.([\,\alpha_1\,]\,x^{\mu+\epsilon_m}\otimes x^{(\alpha-\epsilon_{m+2})})=0,\\
E_mE_{m-1}E_mE_{m+1}.(x^\mu\otimes x^{(\alpha)})&=E_mE_{m-1}E_m.([\,\alpha_1{+}1\,]\,x^\mu\otimes x^{(\alpha+\epsilon_{m+1}-\epsilon_{m+2})})\\
&=E_mE_{m-1}.([\,\alpha_1{+}1\,]\,x^{\mu+\epsilon_m}\otimes x^{(\alpha-\epsilon_{m+2})})\\
&=E_m.([\,\alpha_1{+}1\,]\,x^{\mu+\epsilon_{m-1}}\otimes x^{(\alpha-\epsilon_{m+2})})\\
&=[\,\alpha_1{+}1\,]\,x^{\mu+\epsilon_{m-1}+\epsilon_m}\otimes x^{(\alpha-\epsilon_{m+1}-\epsilon_{m+2})},\\
E_{m+1}E_mE_{m-1}E_m.(x^\mu\otimes x^{(\alpha)})&=E_{m+1}E_mE_{m-1}.(x^{\mu+\epsilon_m}\otimes x^{(\alpha-\epsilon_{m+1})})\\
&=E_{m+1}E_m.(x^{\mu+\epsilon_{m-1}}\otimes x^{(\alpha-\epsilon_{m+1})})\\
&=E_{m+1}.(x^{\mu+\epsilon_{m-1}+\epsilon_m}\otimes x^{(\alpha-2\epsilon_{m+1})})\\
&=[\,\alpha_1{-}1\,]\,x^{\mu+\epsilon_{m-1}+\epsilon_m}\otimes x^{(\alpha-\epsilon_{m+1}-\epsilon_{m+2})},
\end{split}
\end{equation*}
\begin{equation*}
\begin{split}E_mE_{m+1}E_mE_{m-1}.(x^\mu\otimes x^{(\alpha)})&=E_mE_{m+1}E_m.(\delta_{\mu_m,1}x^{\mu+\epsilon_{m-1}-\epsilon_m}\otimes x^{(\alpha)})\\
&=E_mE_{m+1}.(\delta_{\mu_m,1}x^{\mu+\epsilon_{m-1}}\otimes x^{(\alpha-\epsilon_{m+1})})\\
&=E_m.([\,\alpha_1\,]\delta_{\mu_m,1}\,x^{\mu+\epsilon_{m-1}}\otimes x^{(\alpha-\epsilon_{m+2})})\\
&=[\,\alpha_1\,]\delta_{\mu_m,1}\,x^{\mu+\epsilon_{m-1}+\epsilon_m}\otimes x^{(\alpha-\epsilon_{m+1}-\epsilon_{m+2})}\\
&=0, \quad (x_m^2=0),\\
E_mE_{m-1}E_{m+1}E_m.(x^\mu\otimes x^{(\alpha)})&=E_mE_{m-1}.([\,\alpha_1\,]\,x^{\mu+\epsilon_m}\otimes x^{(\alpha-\epsilon_{m+2})})\\
&=E_m.([\,\alpha_1\,]\,x^{\mu+\epsilon_{m-1}}\otimes x^{(\alpha-\epsilon_{m+2})})\\
&=[\,\alpha_1\,]\,x^{\mu+\epsilon_{m-1}+\epsilon_m}\otimes x^{(\alpha-\epsilon_{m+1}-\epsilon_{m+2})}.
\end{split}
\end{equation*}
This gives the required relation: $E_{m-1}E_mE_{m+1}E_m+E_mE_{m-1}E_mE_{m+1}+E_{m+1}E_mE_{m-1}E_m+E_mE_{m+1}E_mE_{m-1}
-(q{+}q^{-1})E_mE_{m-1}E_{m+1}E_m=0$.

Similarly, it is easy to check that:
$F_{m-1}F_mF_{m+1}F_m+F_mF_{m-1}F_mF_{m+1}+F_{m+1}F_mF_{m-1}F_m+F_mF_{m+1}F_mF_{m-1}
-(q{+}q^{-1})F_mF_{m-1}F_{m+1}F_m=0$.

\smallskip
(II) Secondly, it suffices to check that the quantum dual Grassmann algebra $\Omega_q^!(m|n)$ is a $\mathcal U_q(\mathfrak{gl}(m|n))$-module algebra. To this end, we only check here that $E_m$ and $F_m$ act on product
element $(x^\mu\otimes x^{(\alpha)})(x^\nu\otimes x^{(\beta)})$ via $\Delta(E_m)=E_m\otimes \mathcal K_m+\sigma\otimes E_m$ and $\Delta(F_m)=F_m\otimes 1+\sigma\mathcal K_m^{-1}\otimes F_m$, respectively.

(1) For $i=m$: $E_m=x_m\partial_{m+1}\sigma_m$, $\mathcal K_m=\sigma_m\sigma_{m+1}^{-1}$, $\forall\ x^\mu\otimes x^{(\alpha)} \in\Omega_q^!$, where $\mu=(\mu_1,\cdots,\mu_n)\in\mathbb Z_2^m$, noting that $E_m(x^\mu\otimes x^{(\alpha)})=\delta_{0,\mu_m}\,x^{\mu+\epsilon_m}\otimes x^{(\alpha-\epsilon_{m+1})}$, we get
\begin{align*}
E_m&.(x^\mu\otimes x^{(\alpha)})\mathcal K_m.(x^\nu\otimes x^{(\beta)})+\sigma(x^\mu\otimes x^{(\alpha)})E_m.(x^\nu\otimes x^{(\beta)})\\
&=\delta_{0,\mu_m}q^{\nu_m+\beta_1}(-q)^{-(\alpha-\epsilon_{m+1})*\nu}\,x^{\mu+\epsilon_m}x^{\nu}\otimes x^{(\alpha-\epsilon_{m+1})}x^{(\beta)}
\\
&\ +(-1)^{|\alpha|}\delta_{0,\nu_m}(-q)^{-\alpha*(\nu+\epsilon_m)}\,x^{\mu}x^{\nu+\epsilon_m}\otimes x^{(\alpha)}x^{(\beta-\epsilon_{m+1})}\\
&=\delta_{0,\mu_m}q^{\nu_m+\beta_1}(-q)^{-(\alpha-\epsilon_{m+1})*\nu-(\mu+\epsilon_m)*\nu}\,q^{-(\alpha-\epsilon_{m+1})*\beta}
{\alpha{+}\beta{-}\epsilon_{m+1}\brack \beta}\,x^{\mu+\nu+\epsilon_{m}}{\otimes} x^{(\alpha+\beta-\epsilon_{m+1})}
\\
&\ +(-1)^{|\alpha|}\delta_{0,\nu_m}(-q)^{-\alpha*(\nu+\epsilon_m)-\mu*(\nu+\epsilon_m)}q^{-\alpha*(\beta-\epsilon_{m+1})}
{\alpha{+}\beta{-}\epsilon_{m+1}\brack \alpha}\,x^{\mu+\nu+\epsilon_{m}}\otimes x^{(\alpha+\beta-\epsilon_{m+1})}\\
&=\delta_{0,\mu_m+\nu_m}\frac{\bigl(q^{2\nu_m+\beta_1}(-1)^{\nu_m}[\,\alpha_1\,]+q^{-\alpha_1}[\,\beta_1\,]\bigr)q^{-\alpha*\beta}
(-q)^{-\alpha*\nu-\mu*\nu}}{[\alpha_{1}{+}\beta_{1}]}{\alpha{+}\beta\brack \alpha}\times \\
&\qquad \times x^{\mu+\nu+\epsilon_{m}}{\otimes}x^{(\alpha+\beta-\epsilon_{m+1})} \\
&=q^{-\alpha*\beta}
(-q)^{-\alpha*\nu-\mu*\nu}{\alpha{+}\beta\brack \alpha}\,x_m\partial_{m+1}\sigma_m(x^{\mu+\nu}\otimes x^{(\alpha+\beta)})\\
&=E_m.\Bigl((x^{\mu}\otimes x^{(\alpha)})(x^{\nu}\otimes x^{(\beta)})\Bigr),
\end{align*}
where we used $-(\epsilon_m{-}\epsilon_{m+1})*\nu=\nu_m$ and $\epsilon_{m+1}*\beta=0=\mu*\epsilon_m$, $\alpha*\epsilon_m=|\,\alpha\,|$, and $(-1)^{|\,\alpha\,|}(-q)^{-|\,\alpha\,|}q^{\alpha*\epsilon_{m+1}}=q^{-\alpha_1}$ and $q^{\beta_1}[\,\alpha_1\,]+q^{-\alpha_1}[\,\beta_1\,]=[\,\alpha_1{+}\beta_1\,]$.

For $F_m=\sigma_m^{-1}x_{m+1}\partial_m$: $\forall\ x^\mu\otimes x^{(\alpha)}, \ x^\nu\otimes x^{(\beta)}\in\Omega_q^!$,
noting that $\sigma_m^{-1}x_{m+1}\partial_m(x^\mu\otimes x^{(\alpha)})=\delta_{1,\mu_m}[\,\alpha_1{+}1\,]\,x^{\mu-\epsilon_{m}}\otimes x^{(\alpha+\epsilon_{m+1})}$, and $\mathcal K_m^{-1}.(x^\mu)=\sigma_m^{-1}\sigma_{m+1}(x^\mu)=q^{-\mu_m} x^\mu$,
we have
\begin{align*}
&F_m.(x^\mu\otimes x^{(\alpha)})(x^\nu\otimes x^{(\beta)})+\sigma \mathcal K_m^{-1}.(x^\mu\otimes x^{(\alpha)})F_m.(x^\nu\otimes x^{(\beta)})\\
&=\delta_{1,\mu_m}(-q)^{-(\alpha+\epsilon_{m+1})*\nu-(\mu-\epsilon_m)*\nu}q^{-(\alpha+\epsilon_{m+1})*\beta}\,
[\alpha_1{+}1]{\alpha{+}\beta{+}\epsilon_{m+1}\brack \beta}\,x^{\mu+\nu-\epsilon_{m}}\otimes x^{(\alpha+\beta+\epsilon_{m+1})}\\
&\ +(-1)^{|\alpha|}\delta_{1,\nu_m}(-q)^{-(\mu+\alpha)*(\nu-\epsilon_m)}q^{-\mu_m-\alpha_1-\alpha*(\beta+\epsilon_{m+1})}
[\beta_1{+}1]{\alpha{+}\beta{+}\epsilon_{m+1}\brack \alpha}
x^{\mu+\nu-\epsilon_{m}}{\otimes} x^{(\alpha+\beta+\epsilon_{m+1})}\\
&=(\delta_{1,\mu_m}(-q)^{-\nu_m}+\delta_{1,\nu_m}q^{-\mu_m})(-q)^{-\alpha*\nu-\mu*\nu}q^{-\alpha*\beta}[\,\alpha_1{+}\beta_1{+}1\,]{\alpha{+}\beta\brack \alpha}\,x^{\mu+\nu-\epsilon_{m}}{\otimes}x^{(\alpha+\beta+\epsilon_{m+1})} \\
&=\delta_{1,\mu_m+\nu_m}(-q)^{-\alpha*\nu-\mu*\nu}q^{-\alpha*\beta}[\,\alpha_1{+}\beta_1{+}1\,]{\alpha{+}\beta\brack \alpha}\,x^{\mu+\nu-\epsilon_{m}}{\otimes}x^{(\alpha+\beta+\epsilon_{m+1})} \\
&=(-q)^{-\alpha*\nu-\mu*\nu}q^{-\alpha*\beta}{\alpha{+}\beta\brack \alpha}\,F_m.(x^{\mu+\nu}\otimes x^{(\alpha+\beta)} )=F_m.\Bigl((x^{\mu}\otimes x^{(\alpha)})(x^{\nu}\otimes x^{(\beta)})\Bigr),
\end{align*}
where we used $(\epsilon_m-\epsilon_{m+1})*\nu=-\nu_m$, $\epsilon_{m+1}*\beta=0$, $\mu*\epsilon_{m}=0$ and $\alpha*(\epsilon_m-\epsilon_{m+1})=\alpha_1$.

(2) Both for $i<m$ and $i>m$: we easily check that
\begin{gather*}
E_i.(x^\mu\otimes x^{(\alpha)})=E_i.(x^\mu)\otimes x^{(\alpha)}, \quad (i<m)\\
F_i.(x^\mu\otimes x^{(\alpha)})=F_i.(x^\mu)\otimes x^{(\alpha)}, \quad (i<m)\\
E_i.(x^\mu\otimes x^{(\alpha)})=x^\mu\otimes E_i.(x^{(\alpha)}), \quad (i>m)\\
F_i.(x^\mu\otimes x^{(\alpha)})=x^\mu\otimes F_i.(x^{(\alpha)}), \quad (i>m).
\end{gather*}
By analogue of the arguments used in Steps (I) (3) \& (I) (1) of the proof of Theorem 27, we can
examine that that $E_i$ and $F_i$ ($i\ne m$) act on product
element $(x^\mu\otimes x^{(\alpha)})(x^\nu\otimes x^{(\beta)})$ via $\Delta(E_i)=E_i\otimes \mathcal K_i+1\otimes E_i$ and $\Delta(F_i)=F_i\otimes 1+\mathcal K_i^{-1}\otimes F_i$ $(i\ne m)$, respectively.

This completes the proof.
\end{proof}

From the proof of Theorem 32, we see that $\Omega_q^!(m|n)$ as the $\mathcal U_q$-module algebra containing
$\Lambda_q(m)$ as a part of its even part has different
superalgebra structure from that of $\Omega_q(m|n)$.

\subsection{The submodule structures on homogeneous spaces of $\Omega_q^!$}
For each $0\leq j\leq n$, we denote by
\begin{displaymath}
\Lambda_q(m)^{(i)}:= \textrm{span}_{\Bbbk}\{\,x^\mu\mid |\,\mu\,|=i\,\}.
\end{displaymath}
It follows from the definition that $\Lambda_q(m)=\bigoplus_{i=0}^m\Lambda_q(m)^{(i)}$. For any
$t\in \mathbbm{Z}_+$, define
\begin{displaymath}
{\Omega_q^!}^{(t)}:=\bigoplus_{i+j=t}\Lambda_q(m)^{(i)}\otimes\mathcal{A}_{q^{-1}}^{(j)},
\end{displaymath}
where $\Omega_q^!=\Omega_q^!(m|n)$ or, $\Omega_q^!(m|n,\bold 1)$ only when $\textbf{char}(q)=\ell>2$, and
$\mathcal A_{q^{-1}}=\mathcal A_{q^{-1}}(n)$ or, $\mathcal A_{q^{-1}}(n,\bold 1)$ only when $\textbf{char}(q)=\ell>2$.

It is clear that the set $\{\,x^\mu\otimes x^{(\alpha)}\in\Omega_q^!\mid |\,\alpha\,|+|\,\mu\,|=t\,\}$
forms a $\Bbbk$-basis of the homogeneous space ${\Omega_q^!}^{(t)}$. Therefore, we have
$\Omega_q^!=\bigoplus_{t\geq0}{\Omega_q^!}^{(t)}$, in particular, $\Omega_q^!(m|n)=\bigoplus_{t\geq 0}\Omega_q^!(m|n)^{(t)}$
and $\Omega_q^!(m|n,\bold{1})=\bigoplus_{t=0}^{m+N}\Omega_q^!(m|n,\bold {l})^{(t)}$, where $N=|\,\tau(n)\,|=n(\ell{-}1)$.
\begin{theorem}
Each subspace $\Omega_q^!(m|n)^{(t)}$ is a $\mathcal U_q$-submodule of $\Omega_q^!(m|n)$ and
$\Omega_q^!(m|n, \bold 1)^{(t)}$ is a $u_q$-submodule of $\Omega_q^!(m|n, \bold 1)$ when $\textbf{char}(q)=\ell>2$.
\begin{enumerate}[$($i$)$]
\item \label{p1}If $\textbf{char}(q)=0$, then for $t\le m$, each submodule $\Omega_q^!(m|n)^{(t)}\cong V(\omega_t)$
is a simple module generated by highest weight vector $x^{\omega_t}\otimes 1$, where $\omega_t=\epsilon_1+\cdots+\epsilon_t;$ and for $t>m$, each submodule $\Omega_q^!(m|n)^{(t)}\cong V(\omega_m+(t{-}m)\epsilon_{m+1})$
is a simple module generated by highest weight vector $x^{\omega_m}\otimes x^{((t-m)\epsilon_{m+1})}$.

\item If $\textbf{char}(q)=\ell$, then  for $t\le m$, the submodule
$\Omega_q^!(m|n,\bold {1})^{(t)}\cong V(\omega_t);$ and for $t>m$, each submodule
$\Omega_q^!(m|n,\bold {1})^{(t)}\cong V(\omega_m{+}(\ell{-}1)\epsilon_{m+1}{+}\cdots{+}(\ell{-}1)
\epsilon_{m+(i-1)}{+}t_i\epsilon_{m+i})$ is a simple module generated by highest weight vector
$x^{\omega_m}\otimes x^{((\ell{-}1)\epsilon_{m+1}+\cdots+(\ell{-}1)\epsilon_{m+(i-1)}+t_i\epsilon_{m+i})}$
where $t{-}m=(i{-}1)(\ell{-}1)+t_i$, 
$(0\leq t_i\leq \ell{-}1)$.
\end{enumerate}
\end{theorem}
\begin{proof} The proof is similar to that of Theorem 27, and left to the reader.
\end{proof}

\begin{coro}
\begin{enumerate}[$($i$)$]
\item \label{p1}If $\textbf{char}(q)=0$, then for $t\le m$, $$\dim_k\Omega_q^!(m|n)^{(t)}=\dim_k V(\omega_t)=
\sum_{0\le s\le t}\binom{m}{s}\binom{n{+}t{-}s{-}1}{t{-}s};$$
and for $t>m$, $$\dim_k\Omega_q^!(m|n)^{(t)}=\dim_k V(\omega_m+(t{-}m)\epsilon_{m+1})=
\sum_{0\le s\le m}\binom{m}{s}\binom{n{+}t{-}s{-}1}{t{-}s}.$$

\item If $\textbf{char}(q)=\ell$, then  for $t\le m$, $$\dim_k\Omega_q^!(m|n,\bold {1})^{(t)}=\dim_k V(\omega_t)=
\sum_{0\le s\le t}\binom{m}{s}\dim_k\mathcal A_{q^{-1}}^{(t-s)}(n;\bold 1);$$ and for $t>m$,
\begin{equation*}
\begin{split}
\dim_k\Omega_q^!(m|n,\bold {1})^{(t)}&=\dim_k V(\omega_m{+}(\ell{-}1)\epsilon_{m+1}{+}\cdots{+}(\ell{-}1)
\epsilon_{m+(i-1)}{+}t_i\epsilon_{m+i})\\
&=\sum_{0\le s\le m}\binom{m}{s}\dim_k\mathcal A_{q^{-1}}^{(t-s)}(n;\bold 1),
\end{split}
\end{equation*}
where $\dim_k\mathcal A_{q^{-1}}^{(s)}(n;\bold 1)=\sum_{i=0}^{\lfloor
\frac{s}{\ell}\rfloor}(-1)^i\binom{n}{i}\binom{n{+}s{-}i\ell{-}1}{n{-}1}$,
where by $\lfloor x\rfloor$ means the
integer part of $x\in\mathbb Q$.
\end{enumerate}
\end{coro}
\begin{proof}
Note that $\dim_k \mathcal A_{q^{-1}}^{(s)}(n)=\dim_k \text{Sym}^s V=\binom{n{+}s{-}1}{s}$ (see \cite{G}) when $\textbf{char}(q)=0$
and $\dim_k V=n$, while $\dim_k  \mathcal A_{q^{-1}}^{(s)}(n;\bold 1)=\sum_{i=0}^{\lfloor
\frac{s}{\ell}\rfloor}(-1)^i\binom{n}{i}\binom{n{+}s{-}i\ell{-}1}{n{-}1}$ (see Corollary 2.6 \cite{GH}) when $\textbf{char}(q)=\ell>0$.
Then the results follow from the proof of Theorem 33.
\end{proof}

\bigskip
\bigskip
\centerline{\bf ACKNOWLEDGMENT}

\bigskip

The main results of this paper have been reported in several workshops, e.g., a workshop of Yangzhou University, Jan. 23, 2018; and a workshop of Nanjing University of Information Science and Technology, Dec. 22, 2018;
and International Workshop of Hopf Algebras and Tensor Categories, Nanjing University, Xianlin Campus, Sept. 9 --- 13, 2019.
The second author is supported by the NSFC (Grant No. 11771142). The fourth author is supported by the Swedish G\"oran Gustafsson Stiftelse.

\end{document}